\documentclass[10pt]{article}
\usepackage[utf8]{inputenc} 
\usepackage[T1]{fontenc}
\usepackage{lmodern}

\usepackage{amsthm,amsmath,amsfonts,amssymb,stmaryrd,bbm}
\usepackage{mathtools}
\usepackage[shortlabels]{enumitem}
\usepackage{mathrsfs} 

\usepackage{tocloft}

\usepackage{subcaption} 

\usepackage{graphicx}
\usepackage[dvipsnames]{xcolor}
\usepackage{caption,tikz}

\usepackage[english]{babel}
\usepackage[colorlinks=true, linkcolor=blue, urlcolor=black, citecolor=blue,pdfstartview=FitH]{hyperref}

\newcommand{\ba}{{\bf{a}}}
\newcommand{\bb}{{\bf{b}}}

\numberwithin{equation}{section}

\let\originalleft\left
\let\originalright\right
\renewcommand{\left}{\mathopen{}\mathclose\bgroup\originalleft}
\renewcommand{\right}{\aftergroup\egroup\originalright}

\newlength{\bibitemsep}
\newlength{\bibparskip}
\let\oldthebibliography\thebibliography
\renewcommand\thebibliography[1]{\oldthebibliography{#1}
  \setlength{\parskip}{\bibitemsep}
  \setlength{\itemsep}{\bibparskip}}

\DeclareMathOperator{\re}{Re}
\DeclareMathOperator{\im}{Im}

\DeclareMathOperator{\Var}{Var}

\DeclareMathOperator{\OO}{O}
\DeclareMathOperator{\oo}{o}

\DeclareMathOperator{\supp}{supp}

\DeclareMathOperator{\Arg}{Arg}

\newcommand{\ii}{\mathrm{i}}

\newcommand{\diff}{\mathop{}\mathopen{}\mathrm{d}}
\newcommand{\rd}{\mathop{}\mathopen{}\mathrm{d}}

\usepackage{bm}

\newcommand{\CC}{\mathbb{C}}
\theoremstyle{plain} 
\newtheorem{theorem}{Theorem}[section]
\newtheorem{lemma}[theorem]{Lemma}
\newtheorem{corollary}[theorem]{Corollary}
\newtheorem{proposition}[theorem]{Proposition}

\newtheorem{remark}[theorem]{Remark}

\renewcommand{\epsilon}{\varepsilon}
\newcommand{\e}{{\varepsilon}}
\renewcommand{\leq}{\leqslant}
\renewcommand{\geq}{\geqslant}

\renewcommand{\P}{\mathbb{P}}
\newcommand{\E}{\mathbb{E}}
\newcommand{\R}{\mathbb{R}}
\newcommand{\C}{\mathbb{C}}

\newcommand{\1}{\mathbbm{1}}

\newcommand{\cC}{\mathcal{C}}

\newcommand{\cN}{\mathcal{N}}
\newcommand{\cP}{\mathcal{P}}
\newcommand{\cR}{\mathcal{R}}

\newcommand{\pa}[1]{\left({#1}\right)}

\newcommand{\abs}[1]{\lvert #1 \rvert}
\newcommand{\absa}[1]{\left\lvert #1 \right\rvert}

\newcommand{\vertiii}[1]{{\left\vert\kern-0.25ex\left\vert\kern-0.25ex\left\vert #1 
    \right\vert\kern-0.25ex\right\vert\kern-0.25ex\right\vert}}

\theoremstyle{plain} 

\makeatletter
\renewcommand{\subsection}{\@startsection
{subsection}
{2}
{0mm}
{-\baselineskip}
{0 \baselineskip}
{\normalfont\bf\itshape}} 
\makeatother

\makeatletter
\renewcommand{\subsubsection}{\@startsection
{subsubsection}
{3}
{0mm}
{-\baselineskip}
{0 \baselineskip}
{\normalfont\bf\itshape}} 
\makeatother

\def\author#1{\par
    {\centering{\authorfont#1}\par\vspace*{0.05in}}
}

\def\titlefont{\fontsize{13}{15}\bfseries\boldmath\selectfont\centering{}}
\def\authorfont{\fontsize{13}{15}}
\def\abstractfont{\fontsize{8}{10}}

\let\affiliationfont\rhfont

\def\address#1{\par
    {\centering{\affiliationfont#1\par}}\par\vspace*{11pt}
}

\def\title#1{
    \thispagestyle{plain}
    \vspace*{-14pt}
    \vskip 79pt
    {\centering{\titlefont #1\par}}%
    \vskip 1em
}

\def\titlefont{\fontsize{13}{15}\bfseries\boldmath\selectfont\centering{}}
\def\authorfont{\fontsize{13}{15}}
\def\abstractfont{\fontsize{8}{10}}

\let\affiliationfont\rhfont

\def\address#1{\par
    {\centering{\affiliationfont#1\par}}\par\vspace*{11pt}
}

\def\body{
\setcounter{footnote}{0}
\def\thefootnote{\alph{footnote}}
\def\@makefnmark{{$^{\rm \@thefnmark}$}}
}

\def\title#1{
    \thispagestyle{plain}
    \vspace*{-14pt}
    \vskip 79pt
    {\centering{\titlefont #1\par}}%
    \vskip 1em
}

\renewenvironment{abstract}{\par%
    \vspace*{6pt}\noindent 
    \abstractfont
    \noindent\leftskip10pt\rightskip10pt
}{%
  \par}

\usepackage{tocloft}

\makeatletter
\renewcommand{\section}{\@startsection
{section}
{1}
{-3mm}
{-2\baselineskip}
{1\baselineskip}
{\normalfont\large\scshape\centering}} 
\makeatother

\makeatletter
\newcommand{\mynameis}[1]{#1\renewcommand{\@currentlabel}{#1}}
\makeatother

\begin{document}

~\vspace{-1cm}

\title{Optimal local law and central limit theorem for $\beta$-ensembles}

\vspace{0.7cm}
\noindent

 \begin{minipage}[c]{0.31\textwidth}
 \author{Paul Bourgade}
\address{Courant Institute\\ bourgade@cims.nyu.edu}
 \end{minipage}
 \begin{minipage}[c]{0.31\textwidth}
 \author{Krishnan Mody}
\address{Courant Institute\\ km2718@nyu.edu}
 \end{minipage}
\begin{minipage}[c]{0.31\textwidth}
 \author{Michel Pain}
\address{Courant Institute\\ michel.pain@nyu.edu}
 \end{minipage}

\vspace{0.1cm}
\noindent
 
\begin{abstract}
\indent In the setting of generic $\beta$-ensembles, we use the loop equation hierarchy to prove a local law with optimal error up to a constant, valid on any scale including microscopic. This local law has the following consequences. 
(i) The optimal rigidity scale of the ordered particles is of order $(\log N)/N$ in the bulk of the spectrum. 
(ii) Fluctuations of the particles satisfy a central limit theorem with covariance corresponding to a logarithmically correlated field; in particular each particle in the bulk fluctuates on scale $\sqrt{\log N}/N$.
(iii) The logarithm of the electric potential also satisfies a logarithmically correlated central limit theorem.

Contrary to much progress on random matrix universality, these results do not proceed  by comparison. Indeed, they are new for the Gaussian $\beta$-ensembles. By comparison techniques, (ii) and (iii) also hold for Wigner matrices.
\end{abstract}

{\renewcommand{\contentsname}{}
	\hypersetup{linkcolor=black}
	\tableofcontents
}

\vspace{0.1cm}

\section{Introduction}

The $\beta$-ensembles are both a generalization of the Gaussian orthogonal, unitary and symplectic ensembles, and a
natural statistical physics model, the 1d log-gas.
Their distribution, on $N$ points $\lambda_1\leq \dots\leq\lambda_N$, is 
\begin{equation}
	\label{eq:beta_ensembles_density}
	\diff \mu_N (\lambda_1,\dots,\lambda_N)
	\coloneqq \frac{1}{Z_N}
	\prod_{1 \leq k < l \leq N} 
	\absa{\lambda_k - \lambda_l}^\beta 
	e^{- \frac{\beta N}{2} \sum_{k=1}^N V(\lambda_k)} 
	\rd \lambda_1 \dots \rd \lambda_N.
\end{equation}
The initial motivation for this article is to provide a characterization of the fluctuations of individual eigenvalues. Specifically, we ask does the convergence
\begin{equation}\label{eqn:question}
\frac{\lambda_k-\E (\lambda_k)}{\sigma}\to X
\end{equation}
hold in distribution for some $\sigma$ depending on $N,k,\beta$ and some random variable $X$?
What is the  decay of correlations  between the particles? 
In \cite[Section 3.4]{Tao2012}, based on an approximation of the Hamiltonian in (\ref{eq:beta_ensembles_density}) by a quadratic form, and wavelet calculations, Tao developed a heuristic picture for (\ref{eqn:question}). It suggests the fluctuations of eigenvalues 
converge to a limiting log-correlated Gaussian vector. 

Via different arguments, this paper makes these predictions rigorous (see Corollary \ref{corr:eig_fluct}), first phrasing the problem in terms of
fluctuations of $\sum f(\lambda_i)$, with $f$ an indicator function.  Despite considerable attention, fluctuations of such linear statistics were previously obtained only for integrable models ($\beta=1,2,4$) or smooth enough $f$. Our method also applies to the $f=\log$ singularity\footnote{For higher order singularities such as $f=|x|^{-\alpha}$, the fluctuations  are supposedly of order $N^{\alpha}$, non-Gaussian, and essentially local functions of the limiting point process.}, i.e. the electric potential, or log-characteristic polynomial (Theorem \ref{thm:log_corr_field}).

For the proof, we combine equations from the loop equation hierarchy to obtain an optimal local law (Theorem \ref{th:local_law_bulk}) which holds on the microscopic scale. This allows treatment of singular test functions, and non-linear statistics such as $\max|\lambda_i-\E(\lambda_i)|$, identifying the true rigidity scale of the particles (Corollary \ref{cor:rigidity}).

\subsection{Optimal local law. } 

In this paper, $\beta>0$ is fixed and our assumptions on $V$  are the following.
\begin{enumerate}[label=(A\arabic*)]
\item $V$ is analytic on $\R$.  \label{assumption:analytic}%
\item \label{assumption:V'_at_infinity}
There exist constants $M_0,C,c > 0$ such that 
$$ 
	V'(x) \geq c \quad \text{and} \quad
	\sup_{y \in [M_0,x]} \frac{V'(y)}{y} \leq C V(x) 
	\quad \text{for all} \quad x \geq M_0,
$$
and similar estimates apply for $x\leq -M_0$, i.e. the above holds for $\tilde V(x):=V(-x)$.
\item \label{assumption:off-criticality} Under the previous assumptions, it is known that $\E[N^{-1}\sum_{i=1}^N\delta_{\lambda_i}]$ converges weakly to a probability measure $\mu_V$, with density $\varrho_V$. 
We assume $\varrho_V$ is positive and supported on a single interval $[A,B]$ (one-cut hypothesis), with square root singularities at $A$ and $B$ (generic behavior), see (\ref{eqn:rhoV}).
\item \label{assumption:large_dev} The function $x \mapsto V(x)/2 - \int \log\abs{x-t} \rd \mu_V(t)$ achieves its minimum value only on the interval $[A,B]$.
\end{enumerate}

To state our first results, we introduce the following notation.
For $z$ with $\im z \neq 0$, let
\[
	s_N(z) = \frac{1}{N} \sum_{k=1}^N \frac{1}{\lambda_k - z},\ \ \ 
	m_V(z) = \int_{\mathbb{R}}\frac{\rd \mu_V(x)}{x-z},\ \ \ {\rm and\ define\ }\gamma_k\ {\rm through}\ \int_{-\infty}^{\gamma_k}\rd\mu_V=\frac{k}{N}.
\]
\begin{theorem} \label{th:local_law_bulk}
    There exist $\tilde{\eta},C>0$  such that for any $q \geq 1$, $N \geq 1$ and  $z = E + \ii \eta$ with $0 < \eta \leq \tilde{\eta}$ and $A-\eta \leq E \leq B+\eta$, we have
	\[
		\E\left[ \absa{ s_N(z) - m_V(z) }^q \right]
		\leq \frac{(Cq)^{q/2}}{(N\eta)^q}
			+ \frac{(Cq)^q}{N^q \abs{z-A}^{q/2} \abs{z-B}^{q/2}}.
	\]
\end{theorem}

\begin{remark} \label{rem:comment_local_law1}
Proposition \ref{prop:local_law_past_the_edge} extends this local law to  $E \notin [A-\eta,B+\eta]$, with the slightly worse bound $\frac{(Cq)^{2q}}{(N\eta)^q}$.
\end{remark}

\begin{remark} \label{rem:comment_local_law2} For quadratic $V$, that is for the Gaussian $\beta$-ensembles, the second term in the bound of the local law can be removed (see Remark \ref{rem:gaussian}): for $0 < \eta \leq \tilde{\eta}$ and $A-\eta \leq E \leq B+\eta$, we have
$\E\left[ \absa{ s_N(z) - m(z) }^q \right]
		\leq \frac{(Cq)^{q/2}}{(N\eta)^q}$.
\end{remark}

\begin{remark} \label{rem:comment_local_law3} The technical assumption \ref{assumption:V'_at_infinity}, which
states that the potential grows at least linearly,
can be replaced to cover the case of $V$ growing slower than $x^2$, see Remark \ref{rem:alternative_assumption}.
\end{remark}
Theorem \ref{th:local_law_bulk} is an important ingredient for the following estimates on the location of the particles, improving the polynomial error terms of \cite{BouErdYau2014a,BouErdYau2012,BouErdYau2014b}. The logarithmic factor in these rigidity bounds is optimal. Indeed, 
 \cite{ClaFahLamWeb2019} shows that when $\beta=2$, eigenvalues in the bulk can fluctuate from their expected locations by as much as $c(\log N)/N$. Below and in the remainder of this paper, we denote $\hat k=\min(k,N+1-k)$.
\begin{corollary}
\label{cor:rigidity}
For any $D>0$, there exists $C>0$ such that, for any $a_N/\log N\to\infty$, for $N$ large enough,
\[
	\mathbb{P} \left( \max_{k\in\llbracket a_N,N-a_N\rrbracket} \abs{\lambda_k-\gamma_k} > C(\log N)N^{-\tfrac{2}{3}}{\hat k}^{-\tfrac{1}{3}} \right) \leq N^{-D}.
\]
\end{corollary}

Universality of the Tracy-Widom distribution \cite{KriRidVir2016,BouErdYau2014b,BekFigGui2015} suggests $\mathbb{P}(\lambda_N>B+xN^{-2/3})\leq e^{-c x^{3/2}}$, and \cite{LedRid2010} establishes this right tail for the G$\beta$E. 
For general $V$, we prove an exponential decay with exponent $3/4$, as a consequence of our local law outside of the trapezoidal region where Theorem \ref{th:local_law_bulk} holds.
\begin{corollary} \label{cor:rigidity_past_the_edge}
There are constants $c,C>0$ such that, for any $N \geq 1$ and $x \in [0,N^{2/3}]$,
\begin{align*} 
\P\left( \exists k \in \llbracket 1,N \rrbracket, 
\lambda_k \notin \left[A-x N^{-2/3}, B + x N^{-2/3}\right] \right) 
\leq C \exp \big(-c x^{3/4} \big).
\end{align*}
\end{corollary}

Finally, we note that Theorem \ref{th:local_law_bulk} easily gives subsequential convergence of the point process at microscopic scale at any energy level $E$. For  $E$ in the bulk or at the exact edge much more is known, namely fixed energy universality \cite{Shc2014,LanSosYau2019,KriRidVir2016,BouErdYau2014b,BekFigGui2015}. The result is new for varying $E$ at intermediate energy levels. More importantly, it shows tightness directly follows from loop equations.

\begin{corollary} \label{cor:tightness}
Let  $E=E_N\in[A,B]$ be a deterministic sequence and let $\ell(E)$ be as in (\ref{def:l(E)_and_kappa(E)}) below. Then the point process $\sum_i \delta_{(\lambda_i-E)/\ell(E)}$ is tight for the vague topology.
\end{corollary}

\subsection{Logarithmically correlated field. } 
Before stating our second main result, we introduce some further notation.
We consider the principal branch of the logarithm, extended to the negative real numbers by continuity from above, that is $\log(r e^{\ii \theta}) = \log(r) + \ii \theta$ for any $r>0$ and $\theta \in (-\pi,\pi]$.
As is usual, we define $z^\alpha$ by $e^{\alpha \log(z)}$. 
For any $E \in \R$, we set
\begin{equation} \label{def:l(E)_and_kappa(E)}
    \kappa(E) \coloneqq \abs{A-E} \wedge \abs{B-E}
	\qquad \text{and} \qquad
    \ell(E) \coloneqq 
    \begin{cases} 
    N^{-1} \kappa(E)^{-1/2} & \text{if } E \in [A+N^{-2/3},B-N^{-2/3}], \\
    N^{-2/3} & \text{otherwise}.
    \end{cases} 
\end{equation}
The length $\ell(E)$ is the microscopic scale at $E \in [A,B]$, that is the typical spacing between particles close to $E$. We also define
\begin{equation}
	\label{def:L_N}
	L_N(E) 
	\coloneqq \sum_{j=1}^N \log\pa{E-\lambda_j} 
	- N \int \log\pa{E-x} \diff \mu_V(x).
\end{equation}
The following theorem says that fluctuations of the field $L_N$ are asymptotically Gaussian 
with log-correlated structure independent of $V$ up to the edge. Further, the real and imaginary parts of $L_N$ are asymptotically independent. 	
We refer to Section \ref{sec:History} for a review of previous related results. 

\begin{theorem}
	\label{thm:log_corr_field}
	For fixed $m \geq 1$, let $(E_1,\dots,E_m)_{N \geq 1}$ be energy levels in $[A,B]$ possibly depending on $N$. Let $\delta_i=\frac{1}{4} (\frac{2}{\beta}-1)
	\log (\kappa(E_i) \vee N^{-2/3})$ for $1 \leq i \leq m$. 
	Assume that for each $1 \leq i,j \leq m$, the following limits exist,
	\[
		a_{ij} = \lim_{N \to \infty} 
		\frac{\log (\abs{E_i-E_j} \vee \ell(E_i))}{-\log N} 
		\qquad \text{and} \qquad
		b_{ij} = \lim_{N \to \infty} 
		\frac{\log \Bigl( 
		    \frac{\abs{E_i-E_j} \vee \ell(E_i)}
		    {\kappa(E_i)}
		    \wedge 1
		    \Bigr)}{-\log N},
	\]
	and denote $\ba=(a_{ij})_{1 \leq i,j \leq m}$ and $\bb =(b_{ij})_{1 \leq i,j \leq m}$.
	Then, the following convergence holds in distribution:
	\[
		\sqrt{\frac{\beta}{\log N}}	
		\pa{\re L_N(E_1) - \delta_1, \dots, \re L_N(E_m) - \delta_m,
		\im L_N(E_1), \dots, \im L_N(E_m})
		\xrightarrow[N\to\infty]{}
		\mathscr{N} \left( 0, \left( \begin{matrix}
	\ba & 0 \\ 0 & \bb
		\end{matrix} \right)\right).
	\] 
\end{theorem}

Note that by our definition of $\log$, for $x \in \R^*$, $\im \log(x) = \pi \1_{x < 0}$
and therefore $\frac{1}{\pi} \sum_{k=1}^N \im \log\pa{E-\lambda_k}$ counts the number of eigenvalues greater than $E$.
We can therefore answer (\ref{eqn:question}), extending to generic $V$ and any $\beta$,
Gustavsson's famous central limit theorem \cite{Gus2005} for the GUE, and its analogue for the GOE \cite{ORo2010}. For the statement, consider the normalized eigenvalue displacements, for $1 \leq n \leq N$,
\[
		Y_N(n)=
		\pi N
		\sqrt{\frac{\beta}{\log N}}\varrho_V(\gamma_n)
		(\lambda_n - \gamma_n).
	\]

\begin{corollary}
	\label{corr:eig_fluct}
For fixed $m\geq 1$, let $n_1, \dots, n_m \in \llbracket 1,N\rrbracket$ possibly depending on $N$.
	Assume that for each $1 \leq i,j \leq m$, the limit
	\[
		b_{ij} = \lim_{N \to \infty} 
		\frac{\log \Bigl( 
		    \frac{\abs{\gamma_{n_i}-\gamma_{n_j}} \vee \ell(\gamma_{n_i})}
		    {\kappa(\gamma_{n_i})}
		    \wedge 1
		    \Bigr)}{-\log N}
	\]
exists, and denote  $\bb = (b_{ij})_{1\leq i, j \leq m}$. Then, the following convergence holds in distribution:
    \[
        \pa{Y_N(n_1), \dots, Y_N(n_m)} \xrightarrow[N\to\infty]{} \mathscr{N}(0,\bb).
    \]
\end{corollary}

\begin{remark}
If  $n_1\wedge (N-n_1)=\OO(1)$, then $b_{11}=0$: This corollary does not identify fluctuations at the edge of the spectrum, as expected, e.g. for $n_1=1$ from the convergence of 
$N^{2/3} (A-\lambda_1)$
to the Tracy-Widom distribution.
\end{remark}

Finally, 
     \cite[Theorem 1.5]{BouMod2019} states that the real part of the log characteristic polynomial 
    of a (real or complex)
    Wigner matrix is log-correlated (in the bulk of the spectrum) in the limit of large
    dimension, conditional on the same being true in the GOE
    and GUE cases. The proof of this theorem applies
    equally to the imaginary part of the log-characteristic 
    polynomial, and therefore
Theorem \ref{thm:log_corr_field} also holds in the Wigner case.

\begin{corollary} Let
$\lambda_1\leq \dots\leq \lambda_N$ be the eigenvalues
of a real (resp. complex) Wigner matrix as defined in \cite{BouMod2019}. Let $\kappa>0$ and $E_1,\dots,E_m\in[-2+\kappa,2-\kappa]$ satisfy the hypothesis  of Theorem \ref{thm:log_corr_field}.
Then the conclusion of Theorem \ref{thm:log_corr_field} holds with $\beta=1$ (resp. $\beta=2$).
\end{corollary}

\subsection{Related Works. } 
\label{sec:History}
We now describe part of the rich literature on rigidity and central limit theorems in contexts related to the measure (\ref{eq:beta_ensembles_density}).
For many other facets of the 1d log-gas we refer to Peter Forrester's book \cite{For2010}.\\

\vspace{-0.2cm}

\noindent {\it Local law and rigidity.} Typically, (weak) local laws refer to the number of particles in any mesoscopic ball behaving as predicted by the macroscopic equilibrium measure, while rigidity (sometimes called strong local law) means the fluctuations are smaller than for independent particles. Rigidity often 
appears in the context of long-range, repulsive interactions and is an important step in many  proofs of universality in random matrix theory \cite{ErdYau2017}. These notions and the following results are meant to hold with overwhelming probability, $1-\OO(N^{-D})$ for any $D>0$.

For  $\beta$-ensembles, \cite{BouErdYau2014a,BouErdYau2012,BouErdYau2014b} provided the first rigidity bounds (see also \cite{SosWon2012} for the  G$\beta$E), with optimal polynomial scale, i.e.\@ Corollary \ref{cor:rigidity} with $N^{\oo(1)}$ in place of $C\log N$. 
For particles in the bulk, \cite{Li2016} extended these results to the multicut case.
For systems out of equilibrium, a dynamic approach  provided rigidity bounds which are of order $(\log N)^C/N$ in the bulk \cite{HuaLan2019}, and hold up to  the edge \cite{AdhHua2020}.
For discrete $\beta$-ensembles, \cite{GuiHua2019} obtained rigidity bounds similar to \cite{BouErdYau2014a,BouErdYau2012,BouErdYau2014b}, also by combining loop equations with a multiscale analysis of local Gibbs measures.
For the circular $\beta$-ensembles, \cite{Lam2019b} proved
rigidity on the scale $C(\log N)/N$ based on Selberg's integrals, and 
\cite{NajVir2020} identified the correct order of the variance for the number of particles in intervals for the G$\beta$E.

We expect that the optimal rigidity scale $(\log N)/N$ from Corollary \ref{cor:rigidity} holds for Wigner-type  matrices. The first rigidity bounds for (generalized) Wigner matrices were proved in \cite{ErdYauYin2012}, of order $e^{c(\log\log N)^2}/N$. For Wigner matrices, \cite{TaoVu2013} gave the extreme fluctuations $\OO((\log N)^C/N)$, and the current best explicit $C=2$
follows from results in  \cite{GotNauTik2018}.

In dimension two, the log-gas corresponds to the Coulomb interaction. Local laws for general temperature appeared in \cite{Leb2017,BauBouNikYau2017}, together with rigidity in \cite{BauBouNikYau2017}, in the sense of  $\OO(N^\e)$ fluctuations for smooth enough linear statistics, on any mesoscopic scale. For the Coulomb gas in greater dimension, \cite{ArmSer2021} recently obtained local laws and \cite{Ser2020} proved rigidity bounds. Rigidity of the number of particles in domains with smooth boundary is still open  for the Coulomb gas in dimension $d\geq 2$. One expects the variance to be proportional to the boundary's surface \cite{MarYal1980}.  For advances on this conjecture in the case of the hierarchical Coulomb gas (resp. determinantal point processes), see \cite{Cha2019,GanSar2020} (resp. \cite{FenLam2020}).

Theorem \ref{thm:log_corr_field} identifies the exact fluctuations for the number of particles in intervals, for general $\beta$-ensembles. Its proof exploits the full strength of
Theorem~\ref{th:local_law_bulk}, in particular its validity up to the microscopic scale, and Gaussian decay reflected by the factor $q^{q/2}$ (see Section \ref{sec:outline}).\\
	
\vspace{-0.2cm}

\noindent {\it Fluctuations of singular linear statistics.} 
Johansson's method \cite{Johansson1998} has inspired many works on anomalously small  Gaussian fluctuations, for smooth enough linear statistics of particles distributed as in (\ref{eq:beta_ensembles_density}). 
Approaches also related to  a renormalized energy  \cite{BekLebSer2018}, resp. Stein's method \cite{LamLedWeb2019}, have allowed  the possibility of a critical external potential $V$, resp. quantitative such CLTs.
However much less was known for test functions with poor regularity, such as indicators.
Charge and potential fluctuations  were predicted in \cite[Sections 14.5.1, 14.5.2]{For2010} to be Gaussian with logarithmic variance, but all rigorous results were restricted to eigenvalues densities which either are integrable (e.g. determinantal), or admit a sparse random matrix model.

On the integrable side, Theorem \ref{thm:log_corr_field} unifies and extends previous results about the classical invariant  ensembles.
As previously mentioned, Gustavsson proved a joint central limit theorem for $\im L_N$ for the GUE \cite{Gus2005}, based on a general technique by Costin and Lebowitz \cite{CosLeb1995}, and  O'Rourke proved analogous results for GOE and GSE \cite{ORo2010}. 
These results also hold for general external potential $V$ in the bulk \cite{LanSos2020} and $\beta=1,2,4$, thanks to a comparison technique which reduces the result for general $\beta$ to the case of quadratic $V$. 
Concerning $\re L_N$, joint Gaussian fluctuations were proved  for $\beta = 2$, quadratic $V$ and energy levels  independent of $N$ \cite{Kra2007}. 
More was known for random unitary matrices:  $\re L_N$ and $\im L_N$ evaluated at  one point are asymptotically Gaussian and independent \cite{KeaSna2000}, as in Theorem \ref{thm:log_corr_field}, and $\re L_N,\,\im L_N$ evaluated at multiple points convergence to a log-correlated field  \cite{Bou2010}.

On the sparse matrix model  side, Augeri, Butez and Zeitouni have recently proved the one dimensional central limit theorem for $\re L_N(E)$ in the bulk, and quadratic $V$ \cite{AugButZei2020}.
 Their method is completely different from ours and applies to a wide class of Jacobi matrices: The characteristic polynomial of the G$\beta$E satisfies a recurrence, a consequence of the Dumitriu-Edelman tridiagonal model \cite{DumEde2002}.
Using this approach, Tao and Vu \cite{TaoVu2012} proved a CLT for $\re L_N$ at $E=0$, for GOE and GUE (see also \cite{MehNor1998,DelLeC2000}); their method applies to any $\beta$, and from this point of comparison they extended the result to some Wigner matrices (see below).
The study of the recurrence when $E \neq 0$ \cite{AugButZei2020} is considerably harder. This recurrence was also analyzed in \cite{LamPaq20201}, resp.\@ \cite{LamPaq20202}, for $E$ in the upper half plane, resp.\@ at the edge of the spectrum. In particular, \cite{LamPaq20202} proves the CLT for $\re L_N(B+\lambda N^{-2/3})$, with convergence to the same Gaussian random variable for any $\lambda=\OO(1)$.
Concerning $\im \log L_N$, Gaussianity of the number of points in one large interval was obtained for the limiting ${\rm sine}_\beta$ process, again based on the inductive analysis of a related random Schr\" odinger operator \cite{KriValVir2012}.

Such results are universal in the class of Wigner matrices.
Tao and Vu generalized Gustavsson's theorem on $\im L_N$ to Hermitian Wigner matrices under a four moment matching condition, in the bulk \cite{TaoVu2011}. This CLT and its real symmetric analogue \cite{ORo2010} were then extended to  Wigner matrices with finite moments \cite{BouMod2019,LanSos2020}. At the edge, joint fluctuations of $\im L_N$ are known for generalized Wigner matrices  \cite{BouErdYau2014b,Bou2018}.
Concerning $\re L_N$,   the CLT at $E=0$ was known again under a four moment matching assumption \cite{TaoVu2012}, a condition removed
in \cite{BouMod2019}.\\

\noindent {\it Related topics.} Our work is connected to the following lines of research. First,  central limit theorems for smooth enough linear statistics also
hold in higher dimension, for Coulomb systems, as first proved for the Ginibre ensemble  \cite{RidVir2007}, then general $V$, $\beta=2$ \cite{AmeHedMak2011}, any temperature for $d=2$ \cite{LebSer2018,BauBouNikYau2019} and $d=3$ \cite{Ser2020}. This raises the question of upgrading these results to an analogue of Theorem \ref{thm:log_corr_field}, i.e. fluctuations of the electric potential and the charges.

Second, Theorem \ref{thm:log_corr_field} states that $L_N$ belongs to the universality class of log-correlated fields, see \cite{Arg2017} for a survey on these fields and their connections with branching random walks, the Gaussian free field, random matrices,  and analytic number theory.
This suggests the following  asymptotic behavior for the maximum of $\re L_N$ (or $\im L_N$): 
\begin{align}
	\Big( \max_{E \in [A,B]} \re L_N(E) \Big)
	 - \sqrt{\frac{2}{\beta}}
	 \Big( \log N - \frac{3}{4} \log \log N \Big)
	 \xrightarrow[N\to\infty]{\text{(d)}} Z,\ \ \ \ \ \ \ \ \ 
\mbox{$Z$ a randomly shifted Gumbel},
\end{align}
see \cite{FyoSim2016} for a precise conjecture in the case of the GUE.
Fyodorov, Hiary and Keating  initiated such predictions, both on macroscopic and mesoscopic intervals, motivated by the analogous question for $\zeta$ \cite{FyoHiaKea2012,FyoKea2014}.
Parts of their program are proved, including, in the context of log-gases: the first order for the CUE \cite{ArgBelBou2017}, the Ginibre ensemble \cite{Lam2019a} and unitarily invariant Hermitian ensembles \cite{LamPaq2019};
the second order  for the CUE \cite{PaqZei2018}; tightness of the third order for the more general C$\beta$E ensemble \cite{ChhMadNaj2018}.

Another common property of log-correlated fields is the convergence to Gaussian multiplicative chaos of the measure obtained by exponentiating the field.
For the measure (\ref{eq:beta_ensembles_density}), we expect that
\begin{align}
	\frac{e^{\gamma \re L_N(x)}}{\E[e^{\gamma \re L_N(x)}]} \diff x
\end{align}
should converge in distribution with respect to the weak topology to a Gaussian multiplicative chaos measure for any $\gamma \in (0,\gamma_c)$,  and to zero for $\gamma \geq \gamma_c$, with $\gamma_c = \sqrt{2\beta}$. 
A similar result should hold for  $\im L_N$.
Such a convergence has been shown for the CUE \cite{Web2015,NikSakWeb2020}, unitarily invariant Hermitian random matrices \cite{BerWebWon2018,ClaFahLamWeb2019},
classical compact groups \cite{ForKea2020}, and the GOE, GSE \cite{Kiv2021}.

In a different direction, the method presented in this article is based solely on the loop equations. This provides 
an example  where such a hierarchy alone
implies convergence of the point process along subsequences, and precise fluctuations of the individual particles and the potential, despite non summable decay of correlations.
For more on (generalized) BBGKY hierarchies and their consequences on charge and potential fluctuations, see the review \cite{Mar1988}.

\subsection{Outline of the paper. }\label{sec:outline}
We briefly describe the next sections and the ideas of the proofs, which are based on loop equations. This hierarchy was instrumental in obtaining 
partition function expansions in \cite{BorGui2013}. Combined with the rigidity from \cite{BouErdYau2014b}, it also provides
a CLT on mesoscopic scales \cite{BekLod2018}. We show it  gives information up to the microscopic scale.

Section \ref{sec:local_law} contains the proof of local law in Theorem \ref{th:local_law_bulk}. 
We encourage the reader to first consider the case of quadratic $V$, so that the technical Section \ref{subsection:dealing_with_Delta} can be skipped.
The main novelty is algebraic, with a pertinent combination of loop equations. These  are first simply written in terms of moments of  Stieltjes transforms in Proposition \ref{uncentered loop eq moments}, presenting a crucial combinatorial gain compared to the expression in terms of cumulants \cite{BorGui2013}. Then
Section \ref{subsec:combin} shows combinations of this hierarchy up to order $4q-1$ control the $2q$-th moments of the centered Stieltjes transform, see Equation (\ref{eq:loop_eq_combined}). The possibility of this combination was inspired by  Lee and Schnelli  \cite{LeeSch2018}, who 
introduced a new method for the proof of local laws for matrices with independent entries, based on recursive moment estimates. 
For non-quadratic $V$, Section \ref{subsection:dealing_with_Delta} bounds a critical term based on complex analysis, explaining our Assumption \ref{assumption:analytic}\footnote{We believe our results hold for smooth $V$, by replacing the use of Cauchy's formula by Green's theorem for high order quasi-analytic extensions of $V$. We do not pursue this direction, for the sake of simplicity.}.
Section \ref{subsec:between} concludes the proof of Theorem \ref{th:local_law_bulk}, appealing to an appropriate stability lemma from Appendix \ref{section:stability_lemma}.

Section \ref{sec:consequences} proves some consequences of the local law. Proposition \ref{prop:density} gives Wegner estimates, a result 
essential for the proof of Theorem \ref{thm:log_corr_field} and of independent interest; the key input is the Gaussian tail of $s_N-m_V$.
Section \ref{section:rigidity_past_the_edge} proves Corollary \ref{cor:rigidity_past_the_edge}, based on an extension of the local law outside the trapezoidal region,  obtained in Section \ref{sec:partial_local_law_past_the_edge}. The derivation of Corollary \ref{cor:rigidity} in Section \ref{subsec:rigidity} is more subtle: 
A direct application of the Helffer-Sj{\H o}strand formula with the local law as input would give fluctuations $(\log N)^{3/2}/N$ in the bulk. To reach 
$(\log N)/N$, we combine the local law and Johansson's method \cite{Johansson1998}; this relies on rigidity on scale $(\log N)^C/N$ for some biased measures, which is obtained thanks to the Gaussian decay in Theorem \ref{th:local_law_bulk}, again.  In Section \ref{subsec:smoothed}, the local law and the Wegner estimate  reduce the proof of Theorem \ref{thm:log_corr_field} to Gaussian fluctuations of $L_N(z)$
with $\im z$ slightly above the microscopic scale.
 
This central limit theorem for  $L_N(z)$ is proved in Section \ref{section:CLT_above_the_axis}. For these smooth linear statistics, Johansson's classical strategy  applies. We follow an effective implementation of this method on mesoscopic scales, from \cite{BouErdYauYin2016}.\\

\noindent{\bf Acknowledgement.} We thank Gaultier Lambert for his careful reading which improved the paper. We are also grateful to an anonymous referee for pointing out a mistake in an older version of Lemma \ref{lem:delta(f)_and_sigma(f)}.\\

\noindent{\bf Notation.} In this paper, the large (resp. small) constant $C$ (resp. $c$) may vary from line to line, and only depends on the fixed $\beta>0$ and $V$ satisfying the assumptions \ref{assumption:analytic}, \ref{assumption:V'_at_infinity}, \ref{assumption:off-criticality}, \ref{assumption:large_dev}.

\section{Proof of the local law}
\label{sec:local_law}

In this section, our goal is to prove Theorem \ref{th:local_law_bulk}, the local law in a trapezoidal region above the segment $[A,B]$.
In Section \ref{sec:partial_local_law_past_the_edge}, we apply the same strategy to obtain a partial local law outside of the trapezoid.

\subsection{Preliminaries. }  
\label{subsec:preliminaries}

We present in this section several known results concerning the equilibrium measure, $\mu_V$, see for example \cite{AlbPasShc2001}, Proposition 1 and Equation (2.22).
The equilibrium measure $\mu_V(\diff t) = \varrho_V(t) \diff t$ is supported on a single interval $[A,B]$ and satisfies, for any $x \in (A,B)$,
\begin{equation}
	\frac{1}{2} V'(x) = {\rm p.v.} \int_A^B \frac{\varrho_V(t) \diff t}{x-t}.
\end{equation}
Recall that $V$ is analytic in $\R$ and let $\Omega$ denote a simply connected open set of the complex plane, containing $\R$ such that we can extend $V$ analytically in $\Omega$.
For any $t \in [A,B]$, we can write the equilibrium density as
\begin{equation}\label{eqn:rhoV}
	\varrho_V(t) = \frac{1}{\pi} r(t) \sqrt{(t-A)(B-t)},
\end{equation}
where 
\begin{equation} \label{eq:def_r}
	r(z) \coloneqq \frac{1}{2 \pi} \int_A^B \frac{V'(z) - V'(t)}{z-t}
	\frac{\diff t}{\sqrt{(t-A)(B-t)}}
\end{equation}
is analytic in $\Omega$.
Moreover,  Assumption \ref{assumption:off-criticality} means that
the function $r$ has no zero in $[A,B]$.

Recall that for any $z \in \CC \setminus [A,B]$, 
we define the Stieljes transform of the equilibrium measure as
\begin{equation}
	m_V(z) = \int_A^B \frac{\varrho_V(t) }{t-z}\diff t.
\end{equation}
Then, for any $z \in \Omega \setminus [A,B]$, we have
\begin{equation} \label{eq:link_2m+V'_and_r}
	2 m_V(z) + V'(z) = 2 r(z) b(z),
\end{equation}
where 
\begin{equation} \label{eq:def_b}
	b(z) \coloneqq \sqrt{z-A} \sqrt{z-B}
\end{equation}
and we recall that we always use the principal branch of the square root, extended to negative real numbers by $\sqrt{-x} = \ii \sqrt{x}$ for $x>0$.
Note that $b$ is an analytic function in $\CC \setminus [A,B]$,
and satisfies $b(z) \sim z$ at infinity.

For any $z \in \Omega$, introduce
\begin{equation} \label{eq:def_function_h}
	h(z) \coloneqq \int_A^B 
	\frac{V'(\lambda)-V'(z)}{\lambda-z} \varrho_V(\lambda) \diff \lambda,
\end{equation}
which defines an analytic function in $\Omega$.
Then, for any $z \in \Omega \setminus [A,B]$, we have
\begin{equation} \label{eq:fixed_point_equation}
	m_V(z)^2 + V'(z) m_V(z) + h(z) = 0,
\end{equation}
which we refer to as the fixed point equation for $m_V(z)$.
The main strategy for the proof of the local law is to show that the empirical Stieljes transform $s_N(z)$ satisfies approximately the same fixed point equation.
Properties of this quadratic equation are discussed in Appendix \ref{section:stability_lemma}.

Finally, we state two estimates for the distribution of particles that we will use. 
The first one is a rough rigidity result. For any $\varepsilon > 0$, there exist constants $c,C > 0$ such that
\begin{align}
	\P (\exists k \in \llbracket 1,N \rrbracket, \abs{\lambda_k - \gamma_k} \geq \varepsilon) 
	& \leq C e^{-c N}. 
	\label{eq:large_deviation_rigidity}
\end{align}
This is a consequence of the large deviation principle with speed $N^2$ for the empirical measure, see \cite[Theorem 2.6.1]{AndGuiZei10}, combined with the large deviation principle with speed $N$ for the extreme eigenvalues, \cite[Theorem 2.6.6]{AndGuiZei10} which holds up to a condition on the partition function that follows from \cite[Theorem 1 (iii)]{Shc2011}. The large deviation principle with speed $N$ for the extreme eigenvalues can also be found in \cite[Proposition 2.1]{BorGui2013}. Note that we need Assumption \ref{assumption:large_dev} here to guarantee that the rate functions appearing in the large deviation principle for the extreme eigenvalues do not vanish outside of $[A,B]$.

To state the second estimate we use, 
let $\varrho_1^{(N)}(s)$ denote the 1-point function for the eigenvalues under $\mu_N$, which satisfies, for any continuous bounded function $g$,
\[
	\int_\R g(s) \varrho_1^{(N)}(s) \diff s 
	= \E \left[ \frac{1}{N} \sum_{j=1}^N g(\lambda_j) \right].
\]
Then, there are constants $M_1,c>0$ depending only on $V$ and $\beta$, such that for any $\abs{s} \geq M_1$,
\begin{align} \label{eq:bound_density_at_infinity}
	\varrho_1^{(N)}(s) \leq e^{-c NV(s)}.
\end{align}
For a proof, see Pastur and Shcherbina \cite[Theorem 2.2 (i)]{PasShc2008}.

\subsection{Combining loop equations. }\label{subsec:combin}

In order to prove that the empirical Stieljes transform $s(z)$ approximately satisfies the fixed point equation \eqref{eq:fixed_point_equation}, 
for any $z \in \Omega \setminus \R$, we introduce the random function, 
\begin{align*}
	P(z) = s(z)^2+V'(z)s(z)+h(z),
\end{align*}
where for brevity we have written and will continue to write $s(z) = s_N(z)$.
In this section, aiming at a result analogous to \cite[Equation (3.3)]{LeeSch2018}, we combine the loop equations (see Appendix \ref{section:stability_lemma}) to express them in terms of $P(z)$.

Recalling the definition of the function $h$ in \eqref{eq:def_function_h}, for any $z \in \Omega$, we introduce the random variable, 
\begin{align} \label{eq:def_Delta}
	\Delta(z) 
	\coloneqq 
	\left( \frac{1}{N} \sum_{k=1}^N \frac{V'(\lambda_k)-V'(z)}{\lambda_k-z} \right) - h(z),
\end{align}
where the dependence in $N$ is again kept implicit in the notation for brevity.
Furthermore, for any $z,w \in \Omega \setminus \R$, we set 
\begin{equation} \label{eq:def_f(z,w)}
	f(z,w) \coloneqq 
	\begin{cases}
	\partial_w \left( \frac{s(z)-s(w)}{z-w} \right) & \text{if } w \neq z, \\
	\frac{1}{2} s''(z) & \text{if } w = z.
	\end{cases}
\end{equation}
Then, for any $n \geq 1$ and $z,z_1,\dots,z_{n-1} \in \Omega \setminus \R$, 
we write the loop equations for moments from Proposition \ref{uncentered loop eq moments} as
\begin{multline*}
	 \E\left[ (s(z)^2+V'(z)s(z)+h(z)) \prod_{i=1}^{n-1} s(z_i) \right]
	+ \frac{1}{N} \pa{\frac{2}{\beta}-1} 
		\E\left[ s'(z) \prod_{i=1}^{n-1} s(z_i) \right] \\
	 + \frac{2}{N^2\beta} \sum_{j=1}^{n-1}
		\E\left[ \frac{ \prod_{i=1}^{n-1} s(z_i)}{s(z_j)} f(z,z_j) \right]
	+ \E\left[ \Delta(z) \prod_{i=1}^{n-1} s(z_i) \right]
	= 0.
\end{multline*}
Fix some $z,w \in \Omega \setminus \R$.
For any integers $u,v\geq 0$, take $n=u+v+1$, $z_1,\dots, z_u = w$ and $z_{u+1}, \dots, z_{n-1} = \overline{w}$.
Then, noting that $s(\overline{w}) = \overline{s}(w)$, the loop equation becomes
\begin{align*}
	& \E\left[ (s(z)^2+V'(z)s(z)+h(z)) s(w)^u \overline{s}(w)^v \right]
	+ \frac{1}{N} \pa{\frac{2}{\beta}-1} 
	\E\left[ s'(z) s(w)^u \overline{s}(w)^v \right] \\ 
	& + \frac{2}{N^2\beta} 
	\E\left[ u s(w)^{u-1} \overline{s}(w)^v f(z,w)
		+ v s(w)^u \overline{s}(w)^{v-1} f(z,\overline{w}) \right]
	+ \E\left[ \Delta(z) s(w)^u \overline{s}(w)^v \right]
	= 0,
\end{align*}
and recalling the definition of $P(z)$, we have 
\begin{align*}
	\E\left[ P(z) s(w)^u \overline{s}(w)^v \right]
	& = \frac{1}{N} \pa{1-\frac{2}{\beta}} 
	\E\left[ s'(z) s(w)^u \overline{s}(w)^v \right]
	- \E\left[ \Delta(z) s(w)^u \overline{s}(w)^v \right] \\ 
	& \quad {} - \frac{2}{N^2\beta} 
	\E\left[ f(z,w) \partial_s \Bigl( s(w)^u \overline{s}(w)^v \Bigr)
		+ f(z,\overline{w}) \partial_{\overline{s}} \Bigl( s(w)^u \overline{s}(w)^v \Bigr) \right].
\end{align*}
For any integer $q \geq 1$, we compute
\begin{align*}
	& \E\left[ P(z) P(w)^{q-1} \overline{P}(w)^q \right] \\
	& = \sum_{\substack{a_1+b_1+c_1=q-1\\ a_2+b_2+c_2= q}} \frac{(q-1)!}{a_1!b_1!c_1!} \frac{q!}{a_2!b_2!c_2!} 
	V'(w)^{b_1} h(w)^{c_1}
	\overline{V'}(w)^{b_2} \overline{h}(w)^{c_2}
	\E\left[
	P(z) s(w)^{2a_1+b_1} \overline{s}(w)^{2a_2+b_2} 
	\right] \\
	& = \frac{1}{N} \pa{1-\frac{2}{\beta}} 
	\E\left[ s'(z) P(w)^{q-1} \overline{P}(w)^q \right]
	- \E\left[ \Delta(z) P(w)^{q-1} \overline{P}(w)^q \right] \\ 
	& \quad {} - \frac{2}{N^2\beta} 
	\E\left[ f(z,w) \partial_s\left( P(w)^{q-1} \overline{P}(w)^q \right)
		+ f(z,\overline{w}) 
		\partial_{\overline{s}} \left( P(w)^{q-1} \overline{P}(w)^q \right) \right].
\end{align*}
Hence we have that for any $z,w \in \Omega \setminus \R$,
\begin{align}
	\begin{split} \label{eq:loop_eq_combined}
		\E\left[ P(z) P(w)^{q-1} \overline{P}(w)^q \right]
		& = \E\left[ 
		\left( \frac{1}{N} \pa{1-\frac{2}{\beta}} s'(z) - \Delta(z) \right)
		P(w)^{q-1} \overline{P}(w)^q \right] \\
		& \quad {} - \frac{2(q-1)}{N^2\beta} 
		\E\left[ f(z,w) (2s(w) + V'(w)) P(w)^{q-2} \overline{P}(w)^q \right] \\
		& \quad {} - \frac{2q}{N^2\beta} 
		\E\left[ f(z,\overline{w}) 
			(2\overline{s}(w) + \overline{V}'(w)) \abs{P(w)}^{2q-2} \right]. 
	\end{split}
\end{align}
Finally note that under the measure $\mu_N^{[a,b]}$ defined in \eqref{eq:beta_ensembles_density_confined}, for which particles are confined to the interval $[a,b]$, by combining the loop equations from Proposition \ref{prop:loop_equations_confined}, we can prove in a similar way, for any $z,w \in \Omega \setminus [a,b]$,
\begin{align}
	\begin{split} \label{eq:loop_eq_confined_combined}
		\E^{[a,b]} \left[ P(z) P(w)^{q-1} \overline{P}(w)^q \right]
		& = \E^{[a,b]}\left[ 
		\left( \frac{1}{N} \pa{1-\frac{2}{\beta}} s'(z) - \Delta(z) \right)
		P(w)^{q-1} \overline{P}(w)^q \right] \\
		& \quad {} - \frac{2(q-1)}{N^2\beta} 
		\E^{[a,b]} \left[ f(z,w) (2s(w) + V'(w)) P(w)^{q-2} \overline{P}(w)^q \right] \\
		& \quad {} - \frac{2q}{N^2\beta} 
		\E^{[a,b]} \left[ f(z,\overline{w}) 
			(2\overline{s}(w) + \overline{V}'(w)) \abs{P(w)}^{2q-2} \right] \\
		& \quad {} + \frac{2}{ N^2\beta} \left( 
			\frac{\partial_a \E^{[a,b]}[P(w)^{q-1} \overline{P}(w)^q]}{z-a} 
			+ \frac{\partial_b \E^{[a,b]}[P(w)^{q-1} \overline{P}(w)^q]}{z-b} \right) \\
		& \quad {} + \frac{2}{\beta N^2} \left( \frac{\partial_a \ln Z^{[a,b]}_N}{z-a} 
			+ \frac{\partial_b \ln Z^{[a,b]}_N}{z-b} \right)
		\E^{[a,b]} \left[ P(w)^{q-1} \overline{P}(w)^q \right].
	\end{split}
\end{align}

\subsection{Bound on $\Delta$. }
\label{subsection:dealing_with_Delta}

In this section, our goal is to prove the following bound on the quantity $\Delta(z)$ defined in \eqref{eq:def_Delta}.
\begin{lemma} \label{lem:moments_Delta}
For any compact set $K \subset \Omega$, there exist $C >0$ such that for any $z \in K$ and any $q,N \geq 1$,
\[
	\E\left[ \abs{\Delta(z)}^{2q} \right]
	\leq \frac{(Cq)^{2q}}{N^{2q}}.
\]
\end{lemma}
\begin{remark}\label{rem:gaussian}
For quadratic $V$, $\Delta(z)=0$, so the above bound is trivial and this section can be skipped. Non-vanishing $\Delta$ is 
responsible for the second error term in Theorem \ref{th:local_law_bulk} (see the chain of inequalities from (\ref{eq:bound_after_Young}) to (\ref{eq:second_case}) below), so that the local law is improved for Gaussian $\beta$-ensembles as mentioned in Remark \ref{rem:comment_local_law2}
\end{remark}
Recall that the function $r$ (see \eqref{eq:def_r}) 
is analytic in $\Omega$ and, by Assumption \ref{assumption:off-criticality}, $r$ has no zero in $[A,B]$. 
We fix some constant $\delta \in (0,1]$ such that the region
\[
\{ x+\ii y : x \in [A-6\delta,B+6\delta], y \in [-6\delta,6\delta] \}
\]
is included in $\Omega$ and does not contain any zero of $r$.
In order to prove Lemma~\ref{lem:moments_Delta}, we will work under the confined distribution of particles
\begin{equation*}
	\diff \mu_N^{[A-\delta,B+\delta]} (\lambda_1,\dots,\lambda_N)
	\coloneqq \frac{1}{Z_N^{[A-\delta,B+\delta]}}
	\cdot
	\prod_{1 \leq k < l \leq N} \absa{\lambda_k - \lambda_l}^\beta 
	\cdot
	\prod_{k=1}^N e^{- \frac{\beta N}{2} V(\lambda_k)}  \1_{\lambda_k \in [A-\delta,B+\delta]} 
	\rd \lambda_k.
\end{equation*}
We denote by $\E^{[A-\delta,B+\delta]}$ the integral with respect to $\mu_N^{[A-\delta,B+\delta]}$.

For any $k \geq 1$, consider the rectangle with vertices $A-k\delta+\ii k\delta$, $B+k\delta+\ii k\delta$, $B+k\delta-\ii k\delta$, $A-k\delta-\ii k\delta$, and denote by $\cR_k$ the corresponding closed contour with positive orientation.
We will first prove the following preliminary lemma, establishing a local law on the rectangle $\cR_3$ under the law $\P^{[A-\delta,B+\delta]}$. We will then write $\Delta(z)$ as a contour integral on $\cR_3$ in terms of $s - m_V$ in order to prove Lemma \ref{lem:moments_Delta}.
\begin{lemma} \label{lem:first_estimate_stieljes_transform}
There exist constants $C,c > 0$ such that, for any $w \in \cR_3$ and $q,N \geq 1$,
\begin{align*}
	\E^{[A-\delta,B+\delta]} \left[ \abs{s(w) - m_V(w)}^{2q} \right] 
	\leq \frac{(Cq)^q}{N^{2q}} + C^q e^{-c N}.
\end{align*}
\end{lemma}
The proof of this lemma makes use of ideas of Bourgade, Erdös and Yau, \cite[Lemma 2.2]{BouErdYau2012} and \cite[Lemma 3.3]{BouErdYau2014a},  for convex and non-convex potentials, respectively. 
The idea is to integrate the combined loop equation \eqref{eq:loop_eq_confined_combined} with respect to $z$ on a contour around the spectrum to get rid of the term involving $\Delta(z)$, which is analytic inside this contour.
We use this to get bounds on moments of $P(w)$ in terms of moments of $P(z)$ at other points $z$ and we conclude using the maximum principle. See Figure \ref{fig:my_label}.
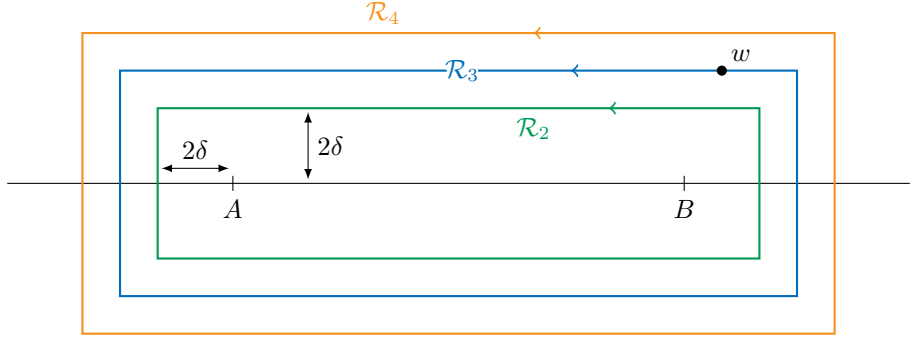
\begin{figure}
    \centering
    {\color{black}
    \begin{tikzpicture}
    \draw (-6,0) -- (6,0);
    \draw (-3,0.1) -- (-3,-0.1) node[below]{$A$};
    \draw (3,0.1) -- (3,-0.1) node[below]{$B$};
    \draw[thick,ForestGreen] (-4,-1) rectangle (4,1);
    \draw[ForestGreen] (1,1) node[below]{$\cR_2$};
    \draw[<->,>=latex] (-2,0.05) -- (-2,0.95); 
    \draw (-2,0.5) node[right]{$2\delta$};
    \draw[<->,>=latex] (-3.95,0.2) -- (-3.05,0.2); 
    \draw (-3.5,0.2) node[above]{$2\delta$};
    \draw[thick,NavyBlue] (-4.5,-1.5) rectangle (4.5,1.5);
    \draw[fill,white] (-0.15,1.4) rectangle (0.25,1.6);
    \draw[NavyBlue] (0.05,1.5) node{$\cR_3$};
    \draw[fill] (3.5,1.5) circle (0.06);
    \draw (3.5,1.5) node[above right]{$w$};
    \draw[thick,BurntOrange] (-5,-2) rectangle (5,2);
    \draw[BurntOrange] (-1,2) node[above]{$\cR_4$};
    \draw[->,thick,ForestGreen] (2.01,1) -- (2,1); 
    \draw[->,thick,NavyBlue] (1.51,1.5) -- (1.5,1.5); 
    \draw[->,thick,BurntOrange] (1.01,2) -- (1,2); 
    \end{tikzpicture}
    }
    \caption{Illustration of the contours used in the proof of Lemma \ref{lem:first_estimate_stieljes_transform}. The point $w$ is fixed on $\cR_3$. The point $z$ will be successively taken in $\cR_2$, then in $\cR_4$ and finally in $\cR_3$}
    \label{fig:my_label}
\end{figure}
\begin{proof}
In this proof we only work under the measure $\mu_N^{[A-\delta,B+\delta]}$, so we write $\E$ instead of $\E^{[A-\delta,B+\delta]}$ for brevity.
For any $z,w \in \Omega \setminus [A-\delta,B+\delta]$, we set 
\begin{align*}
	\varphi_w(z) 
	& \coloneqq 
	P(z) P(w)^{q-1} \overline{P}(w)^q
	- \frac{1}{N} \pa{1-\frac{2}{\beta}} s'(z) P(w)^{q-1} \overline{P}(w)^q \\
	& \quad {} + \frac{2(q-1)}{N^2\beta} 
		f(z,w) (2s(w) + V'(w)) P(w)^{q-2} \overline{P}(w)^q 
		+ \frac{2q}{N^2\beta} f(z,\overline{w}) 
			(2\overline{s}(w) + \overline{V}'(w)) \abs{P(w)}^{2q-2}. 
\end{align*}
Then, combining \eqref{eq:loop_eq_confined_combined} with Lemma \ref{lem:terms_derivative_wrt_a}, there exist constants $c_1 > 0$ such that, for any $w \in \cR_3$ and $z \in \cR_2$,
\begin{align} \label{eq:loop_eq_rewritten}
	\E[\varphi_w(z) ]
	+ \E\left[ \Delta(z) P(w)^{q-1} \overline{P}(w)^q \right] 
	= \OO \left( C^q e^{-c_1 N} \right),
\end{align}
where the error term is uniform in $z$ and $w$.
From \eqref{eq:loop_eq_rewritten} and recalling that $r$ is uniformly bounded away from 0 on $\cR_2$ because its zeroes are at distance at least $6 \delta$ from $[A,B]$, we have
\begin{align*}
	\frac{1}{2\ii\pi}
	\int_{\cR_2}
	\frac{\E[\varphi_w(z)] + \E[\Delta(z) P(w)^{q-1} \overline{P}(w)^q]}
	{r(z) (z-w)}
	\diff z
	= \OO \left( C^q e^{-c_1 N} \right),
\end{align*}
uniformly in $w \in \cR_3$.
The function 
\[
z \mapsto \frac{\E[\Delta(z) P(w)^{q-1} \overline{P}(w)^q]}{r(z) (z-w)}
\]
is holomorphic inside the rectangle $\cR_3$ so its contour integral over $\cR_2$ is zero and we deduce
\begin{align} \label{eq:integral_over_R_2}
	\frac{1}{2\ii\pi}
	\int_{\cR_2}
	\frac{\E[\varphi_w(z)]}{r(z) (z-w)}
	\diff z
	= \OO \left( C^q e^{-c_1 N} \right),
\end{align}
uniformly in $w \in \cR_3$.
The function $z \mapsto \frac{\E[\varphi_w(z)]}{r(z)}$ is holomorphic on $\Omega \setminus [A-\delta,B+\delta]$ so the Cauchy integral formula yields
\begin{align} \label{eq:from_R_2_to_R_4}
	\frac{1}{2\ii\pi}
	\int_{\cR_4}
	\frac{\E[\varphi_w(z)]}{r(z) (z-w)}
	\diff z
	= \frac{\E[\varphi_w(w)]}{r(w)}
	+ \frac{1}{2\ii\pi}
	\int_{\cR_2}
	\frac{\E[\varphi_w(z)]}{r(z) (z-w)}
	\diff z.
\end{align}
On the other hand, recalling the definition of $\varphi_w(z)$, uniformly in $z \in \cR_3 \cup \cR_4$ and $w \in \cR_3$, we have
\begin{align} \label{eq:varphi_w(z)-P(z)}
	\absa{ \varphi_w(z) - P(z) P(w)^{q-1} \overline{P}(w)^q }
	& \leq \frac{C \abs{P(w)}^{2q-1}}{N} + \frac{C q \abs{P(w)}^{2q-2}}{N^2}. 
\end{align}
Combining this with \eqref{eq:integral_over_R_2} and \eqref{eq:from_R_2_to_R_4} and recalling that $r$ is uniformly bounded away from 0 on $\cR_3$, we get
\begin{align*}
	\E \left[ \abs{P(w)}^{2q} \right]
	& \leq \E \left[ \frac{C \abs{P(w)}^{2q-1}}{N} + \frac{Cq \abs{P(w)}^{2q-2}}{N^2} \right]
	+ C \absa{\int_{\cR_4}
		\frac{\E[P(z) P(w)^{q-1} \overline{P}(w)^q]}{r(z) (z-w)}
		\diff z}
	+ C^q e^{-c_1 N}. 
\end{align*}
Recall Young's inequality says that if $x,y\geq 0$ and $a,b > 1$ 
are such that $a^{-1} + b^{-1} = 1$, then $xy \leq \frac{x^a}{a} + \frac{y^b}{b}$.
Choosing $x= 4C/N$, $y = \abs{P(w)}^{2q-1}/4$, $a=2q$, $b=2q/(2q-1)$ we have
\begin{align*}
	\frac{C \abs{P(w)}^{2q-1}}{N}
	& \leq \frac{1}{2q} \frac{(4C)^{2q}}{N^{2q}} 
	+ \frac{2q-1}{2q} \frac{\abs{P(w)}^{2q}}{4^{2q/(2q-1)}}
	\leq \frac{(4C)^{2q}}{N^{2q}} 
	+ \frac{\abs{P(w)}^{2q}}{4}.
\end{align*}
Similarly, we have $\frac{Cq \abs{P(w)}^{2q-2}}{N^2} \leq \frac{(4Cq)^q}{N^{2q}} + \frac{\abs{P(w)}^{2q}}{4}$ 
and therefore, 
up to a modification of the constant~$C$, we have
\begin{align} \label{eq:first_bound_for_P(w)}
	\E \left[ \abs{P(w)}^{2q} \right]
	& \leq \frac{(Cq)^q}{N^{2q}}
	+ C \absa{\int_{\cR_4}
		\frac{\E[P(z) P(w)^{q-1} \overline{P}(w)^q]}{r(z) (z-w)}
		\diff z}
	+ C^q e^{-c_1 N}, 
\end{align}
for any $w \in \cR_3$.
 
We now bound the integral over $\cR_4$.
Recall $P = s^2+V's+h$ and $m_V^2 + V' m_V + h = 0$. It follows that
\begin{align} \label{eq:P_rewritten}
	P = s^2 - m_V^2 + V' (s - m_V) = (s-m_V)^2 + (2m_V+V') (s-m_V)
\end{align}
and therefore
\begin{align*}
	& \int_{\cR_4}
	\frac{\E[P(z) P(w)^{q-1} \overline{P}(w)^q]}{r(z) (z-w)}
	\diff z \\
	& = \E \left[ P(w)^{q-1} \overline{P}(w)^q 
	\int_{\cR_4} \frac{(s(z)-m_V(z))^2 + (2m_V(z)+V'(z)) (s(z)-m_V(z))}{r(z) (z-w)}
	\diff z\right].
\end{align*}
Moreover, recall from \eqref{eq:link_2m+V'_and_r} that, for any $z \in \Omega \setminus [A,B]$,
$2 m_V(z) + V'(z) = 2 r(z) b(z)$, where $b$ is an analytic function in $\CC \setminus [A,B]$, which is always a square root of $(z-A)(z-B)$ and such that $b(z) \sim z$ at infinity.
Therefore, we have 
\begin{align*}
	\int_{\cR_4} \frac{(2m_V(z)+V'(z)) (s(z)-m_V(z))}{r(z) (z-w)} \diff z
	= \int_{\cR_4} \frac{b(z) (s(z)-m_V(z))}{z-w} \diff z.
\end{align*}
Since $s$ and $m_V$ are Stieljes transform of compactly supported probability measures, they both satisfy $s(z) = z^{-1} + \OO(z^{-2})$ and $m_V(z) = z^{-1} + \OO(z^{-2})$ as $\abs{z} \to \infty$. 
Hence, the function $z \mapsto \frac{b(z) (s(z)-m_V(z))}{z-w}$ is holomorphic on and outside $\cR_3$ and behaves as $\OO(z^{-2})$ as $\abs{z} \to \infty$. Therefore, by the Cauchy integral formula with residue at infinity, we get
\begin{align*}
	\int_{\cR_4} \frac{b(z) (s(z)-m_V(z))}{z-w} \diff z = 0.
\end{align*}
Hence, we have
\begin{align}
	\absa{ \int_{\cR_4}
	\frac{\E[P(z) P(w)^{q-1} \overline{P}(w)^q]}{r(z) (z-w)}
	\diff z }
	& = \absa{ \int_{\cR_4} 
	\frac{\E[(s(z)-m_V(z))^2 P(w)^{q-1} \overline{P}(w)^q]}{r(z) (z-w)}
	\diff z } \nonumber \\
	& \leq 
	C \max_{z \in \cR_4} 
	\absa{
	\E \left[ (s(z)-m_V(z))^2 P(w)^{q-1} \overline{P}(w)^q \right] 
	} \nonumber \\
	& \leq C \max_{z \in \cR_3} 
	\absa{
	\E \left[ (s(z)-m_V(z))^2 P(w)^{q-1} \overline{P}(w)^q \right] 
	}, \label{eq:bound_integral_over_R_4}
\end{align}
where we applied the maximum principle to the function $z \mapsto \E[(s(z)-m_V(z))^2 P(w)^{q-1} \overline{P}(w)^q]$, which is analytic outside the contour $\cR_3$ and tends to 0 at infinity.
Applying Young's inequality again, for some $\lambda > 0$, with $x=\lambda \abs{s(z)-m_V(z)}^2$, $y = \frac{1}{\lambda} \abs{P(w)}^{2q}$, $a=2q$, $b=2q/(2q-1)$, we have
\begin{align}
	\absa{
		\E \left[ (s(z)-m_V(z))^2 P(w)^{q-1} \overline{P}(w)^q \right] 
	}
	& \leq
	\frac{\lambda^{2q}}{2q} 
	\E \left[ \abs{s(z)-m_V(z)}^{4q} \right] 
	+ \frac{1}{\lambda} 
	\E \left[ \abs{P(w)}^{2q} \right]. 
	\label{eq:bound_Young}
\end{align}
Coming back to \eqref{eq:first_bound_for_P(w)}, applying \eqref{eq:bound_integral_over_R_4} and \eqref{eq:bound_Young} and choosing $\lambda$ large enough depending on the other constants, we get that, for some constant $C>1$,
\begin{align} \label{eq:second_bound_for_P(w)}
	\E \left[ \abs{P(w)}^{2q} \right]
	& \leq \frac{(Cq)^q}{N^{2q}}
	+ C^q \max_{z \in \cR_3} 
	\E \left[ \abs{s(z)-m_V(z)}^{4q} \right] 
	+ C^q e^{-c_1 N}, 
\end{align}
for any $w \in \cR_3$.

Using $2 m_V + V' = 2 rb$ and our definition of $\delta$, we see there is a constant $c_2 >0$ such that $\abs{2 m_V + V'}(z) \geq 2 c_2$ for any $z \in \cR_3$.
Therefore, if $\abs{s(z)-m_V(z)} \leq c_2$, then it follows from \eqref{eq:P_rewritten} that $\abs{s(z)-m_V(z)} \leq \abs{P(z)}/c_2$.
We define the event $E_z \coloneqq \{ \abs{s(z)-m_V(z)} \leq \varepsilon \}$.
Then, if $\varepsilon \leq c_2$, we get
\begin{align} \label{eq:bound_using_E_z}
	\E \left[ \abs{s(z)-m_V(z)}^{4q} \right] 
	& \leq 
	\varepsilon^{2q} \cdot \frac{1}{c_2^{2q}}
	\E \left[ \abs{P(z)}^{2q} \right] 
	+ C^{4q} \P(E_z^c)
\end{align}
where on event $E_z^c$ we simply used that $\abs{s(z)-m_V(z)} \leq C$.
We choose $\varepsilon \coloneqq c_2/(2 \sqrt{C})$ where $C$ is the constant appearing in \eqref{eq:second_bound_for_P(w)}.
Then, it follows from \eqref{eq:large_deviation_rigidity}%
\footnote{Recall that we are working under $\P^{[A-\delta,B+\delta]}$.
Note that the function $\lambda \in [A-\delta,B+\delta] \mapsto \frac{1}{\lambda-z}$ is bounded, as well as its derivative, uniformly in $z \in \cR_3$.
Therefore, for any $\varepsilon > 0$, there exists $\varepsilon' > 0$ such that, for any $z \in \cR_3$, $\P^{[A-\delta,B+\delta]}(E_z^c) \leq \P^{[A-\delta,B+\delta]}(F)$ where $F \coloneqq \{ \exists k \in \llbracket 1,N \rrbracket, \abs{\lambda_k - \gamma_k} \geq \varepsilon' \}$.
Then, we have 
$\P^{[A-\delta,B+\delta]}(F) \leq  \P (F) / \P(\forall k, \lambda_k \in [A-\delta,B+\delta])$, 
and we can apply \eqref{eq:large_deviation_rigidity} to both probabilities to get the desired result.}, 
that there exist constants $c_3,C_3 > 0$ such that for any $z \in \cR_3$, $\P(E_z^c) \leq C_3 e^{-c_3 N}$.
Coming back to \eqref{eq:second_bound_for_P(w)}, we get, for any $w \in \cR_3$,
\begin{align*}
	\E \left[ \abs{P(w)}^{2q} \right]
	& \leq \frac{(Cq)^q}{N^{2q}}
	+ \frac{1}{2^q} \max_{z \in \cR_3} 
	\E \left[ \abs{P(z)}^{2q} \right] 
	+ C^{4q} C_3 e^{-c_3 N}
	+ C^q e^{-c_1 N}. 
\end{align*}
We take the maximum over $w \in \cR_3$ and bring the term $\frac{1}{2^q} \max_{z \in \cR_3} \E[ \abs{P(z)}^{2q}]$ to the left-hand side to get, up to a modification of the constant $C$,
\begin{align} \label{eq:almost_final_bound}
	\max_{w \in \cR_3} \E \left[ \abs{P(w)}^{2q} \right]
	& \leq \frac{(Cq)^q}{N^{2q}}
	+ C^q e^{-(c_1 \vee c_3) N}. 
\end{align}
Finally, proceeding as in \eqref{eq:bound_using_E_z}, for any $w \in \cR_3$, we have
\begin{align*}
	\E \left[ \abs{s(w)-m_V(w)}^{2q} \right] 
	& \leq \frac{1}{c_2^{2q}} \E \left[ \abs{P(w)}^{2q} \right] 
	+ C^{2q} \cdot C_3 e^{-c_3 N},
\end{align*}
which, combined with \eqref{eq:almost_final_bound}, concludes the proof.
\end{proof}
\begin{proof}[Proof of Lemma \ref{lem:moments_Delta}]
Fix a compact set $K \subset \Omega$.
We will first prove there exist $C,c >0$ such that for any $z \in K$ and any $q \geq 1$,
\begin{align} \label{eq:moments_of_Delta_confined}
	\E^{[A-\delta,B+\delta]} \left[ \abs{\Delta(z)}^{2q} \right]
	\leq \frac{(Cq)^q}{N^{2q}} + C^q e^{-cN}.
\end{align}
For this we write $\Delta(z)$ as a contour integral on $\cR_3$: $\P^{[A-\delta,B+\delta]}$-a.s., we have
\begin{align*}
	\Delta(z) = \frac{1}{2\ii \pi} \int_{\cR_3} \frac{V'(w) - V'(z)}{w-z} \left( s(w) - m_V(w) \right) \diff w,
\end{align*}
because the function $w \mapsto \frac{V'(w) - V'(z)}{w-z}$ is analytic in $\Omega$.
Since $\frac{V'(w) - V'(z)}{w-z}$ is uniformly bounded for $w \in \cR_3$ and $z \in K$, we get, using Jensen's inequality,
\begin{align*}
	\abs{\Delta(z)}^{2q}
	\leq C^q \int_{\cR_3} \absa{s(w) - m_V(w)}^{2q}  \mathrm{Leb}(\diff w),
\end{align*}
where $\mathrm{Leb}$ denotes the Lebesgue measure on $\cR_3$.
Taking the expectation $\E^{[A-\delta,B+\delta]}$ and applying Lemma \ref{lem:first_estimate_stieljes_transform}, we get \eqref{eq:moments_of_Delta_confined}.

Now we want to replace $\E^{[A-\delta,B+\delta]}$ by $\E$.
First note that
\begin{align} \label{eq:Delta_bound_1}
	\E \left[ \abs{\Delta(z)}^{2q} 
	\1_{\forall k, \lambda_k \in [A-\delta,B+\delta]} \right]
	\leq \E^{[A-\delta,B+\delta]} \left[ \abs{\Delta(z)}^{2q} \right]
	\leq \frac{(Cq)^q}{N^{2q}} + C^q e^{-cN}.
\end{align}
by \eqref{eq:moments_of_Delta_confined}.
Let $M_0$ and $M_1$ be the constants given by Assumption \ref{assumption:V'_at_infinity} and \eqref{eq:bound_density_at_infinity} respectively.
We fix some $M \geq \max(M_0,M_1)$ such that $[A-\delta,B+\delta] \subseteq [-M,M]$.
On the event $\{ \forall k, \lambda_k \in [-M,M] \}$, there exists a constant $C>0$ such that $\abs{\Delta(z)} \leq C$ for any $z \in K$. Therefore, using \eqref{eq:large_deviation_rigidity}, we have
\begin{align} \label{eq:Delta_bound_2}
	\E \left[ \abs{\Delta(z)}^{2q} 
	\1_{\exists k, \lambda_k \notin [A-\delta,B+\delta]} 
	\1_{\forall k, \lambda_k \in [-M,M]} \right]
	& \leq C^{2q} \P(\exists k, \lambda_k \notin [A-\delta,B+\delta])
	\leq C^{2q} e^{-cN}.
\end{align}
It remains to bound 
\begin{align} \label{eq:Delta_remaining_term}
	\E \left[ \abs{\Delta(z)}^{2q} 
	\1_{\exists k, \lambda_k \notin [-M,M]} \right]
	& = \E \left[ \abs{\Delta(z)}^{2q} 
		\1_{\lambda_{\min} \leq -M \text{ or } \lambda_{\max} \geq M} \right],
\end{align}
where $\lambda_{\min} \coloneqq \min_{1 \leq k \leq N} \lambda_k$ and $\lambda_{\max} \coloneqq \max_{1 \leq k \leq N} \lambda_k$.
By definition of $\Delta(z)$, we have, uniformly in $z \in K$,
\begin{align*} 
	\abs{\Delta(z)}
	\leq C + \frac{1}{N} \sum_{k=1}^N 
	C \left( \frac{\abs{V'(\lambda_k)}}{\abs{\lambda_k}+1} +1 \right)
	\leq C + C \sup_{y \in [\lambda_{\min},\lambda_{\max}]} \frac{\abs{V'(y)}}{\abs{y}+1}
	\leq C \left(1 + V(\lambda_{\min}) + V(\lambda_{\max}) \right),
\end{align*}
where, in the last inequality, we used Assumption \ref{assumption:V'_at_infinity}.
Therefore, we get
\begin{align*} 
	\E \left[ \abs{\Delta(z)}^{2q} 
	\1_{\lambda_{\min} \leq -M \text{ or } \lambda_{\max} \geq M} \right]
	\leq C^q \E \left[
	\left(1 + V(\lambda_{\min})^{2q} + V(\lambda_{\max})^{2q} \right)
	\1_{\lambda_{\min} \leq -M \text{ or } \lambda_{\max} \geq M} \right].
\end{align*}
The first term in this expectation gives $\P(\lambda_{\min} \leq -M \text{ or } \lambda_{\max} \geq M) \leq C e^{-cN}$ by \eqref{eq:large_deviation_rigidity}, 
and two other terms can be treated similarly. We now study
\begin{align*} 
	\E \left[ V(\lambda_{\max})^{2q} 
	\1_{\lambda_{\min} \leq -M \text{ or } \lambda_{\max} \geq M} \right]
	& = \E \left[ V(\lambda_{\max})^{2q} 
	\1_{\lambda_{\min} \leq -M \text{ and } \lambda_{\max} < M} \right]
	+ \E \left[ V(\lambda_{\max})^{2q} 
	\1_{\lambda_{\max} \geq M} \right] \\
	& \leq \E \left[ \left( V(\lambda_{\min})^{2q} + C \right)
	\1_{\lambda_{\min} \leq -M} \right]
	+ \E \left[ V(\lambda_{\max})^{2q} 
	\1_{\lambda_{\max} \geq M} \right].
\end{align*}
We postpone momentarily the proof of the bounds 
\begin{align} \label{eq:Delta_bound_3}
	\E \left[ V(\lambda_{\min})^{2q} \1_{\lambda_{\min} \leq -M} \right]
	\leq \frac{(Cq)^{2q}}{N^{2q}}
	\qquad \text{and} \qquad
	\E \left[ V(\lambda_{\max})^{2q} \1_{\lambda_{\max} \geq M} \right]
	\leq \frac{(Cq)^{2q}}{N^{2q}}.
\end{align}
Assuming them, we have
\begin{align*} 
	\E \left[ \abs{\Delta(z)}^{2q} 
	\1_{\lambda_{\min} \leq -M \text{ or } \lambda_{\max} \geq M} \right]
	\leq \frac{(Cq)^{2q}}{N^{2q}} + C^q e^{-cN},
\end{align*}
which, combined with \eqref{eq:moments_of_Delta_confined}, \eqref{eq:Delta_bound_1} and \eqref{eq:Delta_bound_2} proves that
\begin{align}  \label{eq:almost_final_bound_Delta}
	\E\left[ \abs{\Delta(z)}^{2q} \right]
	\leq \frac{(Cq)^{2q}}{N^{2q}} + C^q e^{-cN}.
\end{align}
Using that $x \log x \geq -1/e$ for any $x \geq 0$, we have that $\frac{(Cq)^{2q}}{N^{2q}} \geq \exp(-2N/(Ce))$, so up to a modification of the constant $C$, the second term on the right-hand side of \eqref{eq:almost_final_bound_Delta} can be neglected and the result follows.

Finally, we prove \eqref{eq:Delta_bound_3}. The proofs of both cases are identical, so we focus on the case $\lambda_{\max}$.
Recall $\varrho_1^{(N)}(s)$ denotes the 1-point function for the eigenvalues under $\mu_N$ and note that, for any interval $[a,b]$, we have $\P(\lambda_{\max} \in [a,b]) \leq \int_a^b N \varrho_1^{(N)}(s) \diff s$. Hence, 
\begin{align*} 
	\E \left[ V(\lambda_{\max})^{2q} 
	\1_{\lambda_{\max} \geq M} \right]
	& \leq \int_M^\infty V(s)^{2q} N \varrho_1^{(N)}(s) \diff s
	\leq N \int_M^\infty V(s)^{2q} e^{-c_1 NV(s)} \diff s,
\end{align*}
using \eqref{eq:bound_density_at_infinity}. By Assumption \ref{assumption:V'_at_infinity}, we have $V'(s) \geq c$ for $s \geq M$, so we can use the change of variable $x = V(s)$ and induction on $q$ to find
\begin{align*} 
	\int_M^\infty V(s)^{2q} e^{-c_1 NV(s)} \diff s
	& \leq \int_{V(M)}^\infty x^{2q} e^{-c_1 Nx} \frac{\diff s}{c}
	\leq C \int_0^\infty x^{2q} e^{-c_1 Nx} \diff s
	= C \frac{(2q)!}{(c_1N)^{2q+1}},
\end{align*}
which proves the second part of \eqref{eq:Delta_bound_3}.
\end{proof}

\begin{remark} \label{rem:alternative_assumption}
The proof of Lemma \ref{lem:moments_Delta} is the only place where we use Assumption \ref{assumption:V'_at_infinity}.
As mentioned in Remark~\ref{rem:comment_local_law3}, this assumption can be replaced 
to include slower divergence. For example, 
\begin{equation} \label{eq:alternative_assumption}
	\liminf_{x \to \pm \infty} \frac{V(x)}{2 \ln \abs{x}} > 1
	\qquad \text{and} \qquad 
	\limsup_{x \to \pm \infty} \frac{\abs{V'(x)}}{\abs{x}} < \infty
\end{equation}
works.
The first part of this assumption is the usual assumption that ensures that the measure in \eqref{eq:beta_ensembles_density} is finite.
The second part of \eqref{eq:alternative_assumption} implies that $\abs{\Delta(z)}$ is uniformly bounded for $z$ in a compact set, so we can directly bound the expectations in \eqref{eq:Delta_bound_2} and \eqref{eq:Delta_remaining_term} by $C^q e^{-cN}$.
Therefore, under the assumption (\ref{eq:alternative_assumption}), Lemma \ref{lem:moments_Delta} becomes 
\[
	\E\left[ \abs{\Delta(z)}^{2q} \right]
	\leq \frac{(Cq)^q}{N^{2q}} + C^q e^{-cN},
\]
which is slightly better,
and the local law in Theorem \ref{th:local_law_bulk} becomes
\[
	\E\left[ \absa{ s(z) - m_V(z) }^q \right]
	\leq \frac{(Cq)^{q/2}}{(N\eta)^q} 
		+ \frac{C^q e^{-cN}}{\abs{z-A}^{q/2} \abs{z-B}^{q/2}},
\]
by following the same proof as below.
\end{remark}

\subsection{Proof of the local law between the edges. }\label{subsec:between}

In this section we prove the local law in a trapezoid above $[A,B]$, that is Theorem \ref{th:local_law_bulk}.
With the loop equations combined as in \eqref{eq:loop_eq_combined}, we are able to show $P(z) = s(z)^2 + V'(z)s(z) + h(z)$ is small.
Since $m_V(z)$ is solution of the equation $m_V(z)^2 + V'(z)m_V(z) + h(z) = 0$, this means that $m_V(z)$ and $s(z)$ are close (see Lemma \ref{lem:fixed_point_equation} for the precise statement).
\begin{proof}[Proof of Theorem \ref{th:local_law_bulk}] 
We consider the constant $\tilde{\eta} > 0$ given by Lemma \ref{lem:preliminaries_fixed_point_equation}, and work with $z = E + \ii \eta$ such that $\eta \in (0,\tilde{\eta}]$ and $E \in [A-\tilde{\eta},B+\tilde{\eta}]$.
Recall that by the choice of $\tilde{\eta}$, we have in particular $z \in \Omega$, that is $V$ is analytic at $z$.
We start from the loop equation \eqref{eq:loop_eq_combined} in which we take $w=z$ and apply the triangle inequality:
\begin{align} 
	\begin{split} \label{eq:first_bound_loop_eq}
	\E\left[ \abs{P(z)}^{2q} \right]
	& \leq \E\left[ 
	\absa{\frac{1}{N} \pa{1-\frac{2}{\beta}} s'(z)} 
	 \cdot \abs{P(z)}^{2q-1} \right] 
	+ \E\left[ \absa{ \Delta(z) } \cdot \abs{P(z)}^{2q-1} \right] \\
	& \quad {} + 
	\E\left[
	\absa{ \frac{2(q-1)}{N^2\beta} f(z,z) (2s(z) + V'(z))} 
	\cdot \abs{P(z)}^{2q-2} \right] \\
	& \quad {} +  
	\E\left[ 
	\absa{ \frac{2q}{N^2\beta} f(z,\overline{z}) (2s(z) + V'(z))}
	\cdot \abs{P(z)}^{2q-2} \right]. 
	\end{split}
\end{align}
Recall Young's inequality: if $x,y\geq 0$ and $a,b > 1$ such that $a^{-1} + b^{-1} = 1$, then $xy \leq \frac{x^a}{a} + \frac{x^b}{b}$.
We apply it to each term, introducing artificially factors $\lambda$ and $\frac{1}{\lambda}$ for some $\lambda >0$ and taking $a=2q$, $b=2q/(2q-1)$ for the two first terms and $a=q$, $b=q/(q-1)$ for the two last terms. It follows that the right-hand side of \eqref{eq:first_bound_loop_eq} is smaller than
\begin{align*}
	& \frac{\lambda^{2q}}{2q} \E\left[ 
	\absa{\frac{1}{N} \pa{1-\frac{2}{\beta}} s'(z)}^{2q}
	\right] 
	+ \frac{\lambda^{2q}}{2q} \E\left[ \absa{ \Delta(z) }^{2q} \right] 
	+ \frac{\lambda^q}{q} \E\left[
	\absa{ \frac{2(q-1)}{N^2\beta} f(z,z) (2s(z) + V'(z))}^q 
	\right] \\
	& {} +  
	\frac{\lambda^q}{q} \E\left[ 
	\absa{ \frac{2q}{N^2\beta} f(z,\overline{z}) (2s(z) + V'(z))}^q
	\right]
	+ \left( 2 \cdot \frac{2q-1}{2q \lambda^{2q/(2q-1)}}
	+ 2 \cdot \frac{q-1}{q \lambda^{q/(q-1)}} \right)
	\E\left[ \abs{P(z)}^{2q} \right].
\end{align*}
Taking $\lambda = 8$, the last term is smaller than $\frac{1}{2} \E[\abs{P(z)}^{2q}]$, so we can bring it to the left-hand side and get
\begin{align*}
	\E\left[ \abs{P(z)}^{2q} \right]
	& \leq \frac{C^q}{N^{2q}} \E\left[\absa{s'(z)}^{2q}\right] 
	+ C^q \E\left[ \absa{ \Delta(z) }^{2q} \right] 
	+ \frac{(Cq)^q}{N^{2q}} 
	\E\left[ \left( \absa{f(z,z)}^q + \absa{f(z,\overline{z})}^q \right)  \absa{2s(z) + V'(z)}^q \right].
\end{align*}
Recall that we write $z = E + \ii \eta$.
Then, we have
\begin{align*}
	\absa{s'(z)}
	= \absa{ \frac{1}{N} \sum_{k=1}^N \frac{1}{(\lambda_k-z)^2} }
	\leq \frac{1}{N} \sum_{k=1}^N \frac{1}{\abs{\lambda_k-z}^2}
	= \frac{\im s(z)}{\eta}
\end{align*}
and, recalling the definition of $f$ in \eqref{eq:def_f(z,w)} and proceeding similarly,
\begin{align*}
	\absa{f(z,z)} 
	= \frac{1}{2} \abs{s''(z)} 
	\leq \frac{\im s(z)}{\eta^2}
	\qquad \text{and} \qquad
	\absa{f(z,\overline{z})}
	= \absa{ \frac{s(z)-\overline{s(z)}}{(z-\overline{z})^2} - \frac{\overline{s'(z)}}{z-\overline{z}} }
	\leq 2 \frac{\im s(z)}{\eta^2}.
\end{align*}
Hence, applying also Lemma \ref{lem:moments_Delta} to bound $\E[ \absa{ \Delta(z) }^{2q}]$, we get
\begin{align} \label{eq:bound_after_Young}
	\E\left[ \abs{P(z)}^{2q} \right]
	& \leq \frac{(Cq)^q}{(N\eta)^{2q}} 
	\E\left[ (\im s(z))^{2q} 
	+ (\im s(z))^q \cdot \absa{2s(z) + V'(z)}^q \right] 
	+ \frac{(Cq)^{2q}}{N^{2q}},
\end{align}
uniformly in $z$ such that $\eta \in (0,\tilde{\eta}]$ and $E \in [A-\tilde{\eta},B+\tilde{\eta}]$.

We now assume additionally that $E \in [A-\eta,B+\eta]$.
We can therefore apply bound \eqref{eq:stability_lemma_bound_3} of Lemma \ref{lem:fixed_point_equation}, noting that $\im s(z) > 0$ to get
\begin{align} \label{eq:consequence_stability}
	\absa{s(z)-m_V(z)} 
	\leq C \left( \frac{\abs{P(z)}}{\abs{b(z)}} \wedge \abs{P(z)}^{1/2} \right).
\end{align}
Let $\Lambda \coloneqq (\im s(z)) \vee \abs{2s(z)+V'(z)}$.
Note that
\begin{align*}
	\Lambda 
	\leq 2 \absa{s(z)-m_V(z)} 
	+ \left(\im m_V(z) \vee \abs{2m_V(z)+V'(z)}\right)
	\leq C(\abs{P(z)}^{1/2} + \abs{b(z)}),
\end{align*}
where we applied \eqref{eq:consequence_stability} for the first term, and, for the second term, we used that $2m_V(z)+V'(z) = 2r(z)b(z)$ by \eqref{eq:link_2m+V'_and_r} and $\im m_V(z) \leq \frac{3}{2} \im(r(z)b(z))$ by \eqref{eq:preliminary_lemma_bound_2}.
Hence, we get
\begin{align*}
	\E \left[ \absa{\Lambda}^{2q} \right] 
	& \leq C^q \left( \E \left[ \abs{P(z)}^q \right] 
	+ \abs{b(z)}^{2q} \right)
	\leq \frac{(Cq)^{q/2}}{(N\eta)^q} 
	\E \left[ \absa{\Lambda}^q \right] 
	+ \frac{(Cq)^{q}}{N^{q}}
	+ C^q \abs{b(z)}^{2q},
\end{align*}
by \eqref{eq:bound_after_Young}.
Using that $x \leq a \sqrt{x} + b$ for some $a,b,x > 0$ implies $x \leq a^2 + b$, we get
\begin{align} \label{eq:bound_Lambda}
	\E \left[ \absa{\Lambda}^{2q} \right] 
	\leq \frac{(Cq)^{q}}{(N\eta)^{2q}} 
	+ \frac{(Cq)^{q}}{N^{q}} 
	+ C^q \abs{b(z)}^{2q}.
\end{align}
Plugging this in \eqref{eq:bound_after_Young}, we get
\begin{align} \label{eq:final_bound_P(z)}
	\E\left[ \abs{P(z)}^{2q} \right]
	& \leq \frac{(Cq)^{2q}}{(N\eta)^{4q}} 
	+ \frac{(Cq)^q \abs{b(z)}^{2q}}{(N\eta)^{2q}} 
	+ \frac{(Cq)^{2q}}{N^{2q}},
\end{align}
noting that the crossed term $\frac{(Cq)^{2q}}{(N\eta)^{2q} N^q}$ can be neglected.

We now distinguish two cases. First assume that $\abs{b(z)} \geq \sqrt{q}/(N\eta)$. Then we apply successively the first bound in \eqref{eq:consequence_stability} and \eqref{eq:final_bound_P(z)} to get
\begin{align*}
	\E\left[ \abs{s(z)-m_V(z)}^{2q} \right]
	\leq \frac{C^q}{\abs{b(z)}^{2q}} \E\left[ \abs{P(z)}^{2q} \right]
	\leq \frac{(Cq)^q}{(N\eta)^{2q}} 
	+ \frac{(Cq)^{2q}}{N^{2q} \abs{b(z)}^{2q}},
\end{align*}
where we used the assumption $\abs{b(z)} \geq \sqrt{q}/(N\eta)$ to get rid of one term.
Hence the result is proved in this first case.
We now assume that $\abs{b(z)} < \sqrt{q}/(N\eta)$.
In this case, we apply the second bound in \eqref{eq:consequence_stability} and get
\begin{align} \label{eq:second_case}
	\E\left[ \abs{s(z)-m_V(z)}^{2q} \right]
		\leq C^q \E\left[ \abs{P(z)}^q \right]
	\leq C^q \E\left[ \abs{P(z)}^{2q} \right]^{1/2}
	\leq \frac{(Cq)^q}{(N\eta)^{2q}} 
	+ \frac{(Cq)^{q/2} \abs{b(z)}^q}{(N\eta)^q} 
	+ \frac{(Cq)^q}{N^q}.
\end{align}
Using $\abs{b(z)} < \sqrt{q}/(N\eta)$, the second term in the right-hand side of \eqref{eq:second_case} is smaller than the first one. So we only have to prove that
\begin{align*} 
	\frac{(Cq)^q}{N^q} \leq \frac{(Cq)^q}{(N\eta)^{2q}} 
		+ \frac{(Cq)^{2q}}{N^{2q} \abs{b(z)}^{2q}}.
\end{align*}
If $\eta \leq N^{-1/2}$, we simply bound $N^{-q} \leq (N\eta)^{-2q}$. If $\eta > N^{-1/2}$, we get $\abs{b(z)} < \sqrt{q}/(N\eta) < \sqrt{q/N}$ and therefore $N^{-q} \leq q^q N^{-2q} \abs{b(z)}^{-2q}$.
So the result is proved in the second case and this concludes the proof.
\end{proof}

\subsection{A partial local law past the edge. }
\label{sec:partial_local_law_past_the_edge}

The local law in Theorem \ref{th:local_law_bulk} has only been established in a trapezoid region above the interval $[A,B]$, that is at a point $z = E + \ii \eta$ such that $A-\eta \leq E \leq B+\eta$. 
In this section we show how the method used to prove this local law can be extended past the edges.
Recall that the constraint $A-\eta \leq E \leq B+\eta$ was required in order to use bound \eqref{eq:stability_lemma_bound_3} of Lemma~\ref{lem:fixed_point_equation} concerning the stability of the fixed point equation satisfied by $m_V$.
However, if this constraint is not satisfied we can still apply the other bounds of Lemma~\ref{lem:fixed_point_equation} to control the Stieljes transform $s(z)$, with $z$ beyond the edge.

As in Appendix \ref{section:stability_lemma}, we define 
\[
	\widetilde{m}_V(z) \coloneqq -\frac{V'(z)}{2} - r(z) b(z).
\]
which is the other root of the equation $u^2 + V'(z) u + h(z) = 0$ satisfied by $m_V(z)$. 
The following proposition shows that $s(z)$ has to be close to $m_V(z)$ or $\widetilde{m}_V(z)$, and a similar bound on $\im(s(z)-m_V(z))$ follows from \eqref{eq:stability_lemma_bound_2}.
\begin{proposition} \label{prop:bound_s-m_1_s-m_2_past_the_edge}
Let $\tilde{\eta} > 0$ be the constant given by Lemma \ref{lem:preliminaries_fixed_point_equation} and recall the definition of $\kappa$ (\ref{def:l(E)_and_kappa(E)}).
There exist constants $C,C' >0$ 
such that for any $q \geq 1$, $N \geq 1$ and any $z = E + \ii \eta$ 
with $0 < \eta \leq \tilde{\eta}$ and $E \notin [A-\eta, B+\eta]$, we have, if $\eta \geq (C'q)^{1/2}/ (N \sqrt{\kappa})$,
\begin{align} \label{eq:bound_1_s-m_1_s-m_2}
	\E \left[ \abs{s(z)-m_V(z)}^{2q} \wedge \abs{s(z)-\widetilde{m}_V(z)}^{2q} \right]
	\leq 
	\dfrac{(Cq)^{2q}}{(N\eta)^{4q} \kappa^q}
	+ \dfrac{(Cq)^q}{N^{2q} \eta^q \kappa^q} 
	+ \dfrac{(Cq)^{2q}}{N^{2q} \kappa^q} 
\end{align}
and, if $\eta \leq (C'q)^{1/2}/ (N \sqrt{\kappa})$,
\begin{align} \label{eq:bound_2_s-m_1_s-m_2}
	\E \left[ \abs{s(z)-m_V(z)}^{2q} \wedge \abs{s(z)-\widetilde{m}_V(z)}^{2q} \right]
	\leq \dfrac{(C q)^q}{(N\eta)^{2q}} + \frac{(Cq)^q}{N^q}.
\end{align}
Moreover, the same bounds hold for $\E[ \abs{\im(s(z)-m_V(z))}^{2q} ]$.
\end{proposition}
\begin{proof}
Recall we proved in \eqref{eq:bound_after_Young} that
\begin{align} \label{eq:bound_after_Young_bis}
	\E\left[ \abs{P(z)}^{2q} \right]
	& \leq \frac{(Cq)^q}{(N\eta)^{2q}} 
	\E\left[ (\im s(z))^{2q} 
	+ (\im s(z))^q \cdot \absa{2s(z) + V'(z)}^q \right] 
	+ \frac{(Cq)^{2q}}{N^{2q}},
\end{align}
uniformly in $z = E + \ii \eta$ such that $\eta \in (0,\tilde{\eta}]$ and $E \in [A-\tilde{\eta},B+\tilde{\eta}]$.
Now we assume moreover that $E \notin [A-\eta,B+\eta]$.
For brevity, let $\Upsilon \coloneqq \absa{s(z)-m_V(z)} \wedge \absa{s(z)-\widetilde{m}_V(z)}$.
Then, on the one hand, we have
\begin{align*}
	\im s(z) 
	& \leq \Upsilon + (\abs{\im m_V(z)} \vee \abs{\im \widetilde{m}_V(z)})
	\leq \Upsilon + \frac{C \eta}{\abs{b(z)}},
\end{align*}
using $  \abs{\im m_V(z)}\vee \abs{\im \widetilde{m}_V(z)} \leq \frac{C \eta}{\abs{b(z)}}$ by (\ref{eq:preliminary_lemma_bound_4}). On the other hand, we have
\begin{align*}
	\abs{2s(z) + V'(z)} 
	& \leq 2\Upsilon + (\abs{2 m_V(z) + V'(z)} \vee \abs{2 \widetilde{m}_V(z) + V'(z)})
	= 2\Upsilon + 2 \abs{r(z) b(z)}
	\leq 2\Upsilon + C \abs{b(z)}.
\end{align*}
Moreover, we have $\abs{b(z)} \leq C \sqrt{\kappa}$ because $\eta \leq \kappa$.
Therefore, \eqref{eq:bound_after_Young_bis} becomes
\begin{align} \label{eq:bound_P_Upsilon}
	\E\left[ \abs{P(z)}^{2q} \right]
	& \leq \frac{(Cq)^q}{(N\eta)^{2q}} 
	\left( 
		\E[\Upsilon^{2q}] 
		+ \kappa^{q/2} \E[\Upsilon^{q}] 
		+ \eta^q
	\right)
	+ \frac{(Cq)^{2q}}{N^{2q}},
\end{align}
where we used $\eta^q/|b|^{2q}\leq C$.
It follows from bound \eqref{eq:stability_lemma_bound_1} of Lemma \ref{lem:fixed_point_equation} combined with $\abs{b(z)} \geq c \sqrt{\kappa}$ that
\begin{equation} 
	\Upsilon 
	= \absa{s(z)-m_V(z)} \wedge \absa{s(z)-\widetilde{m}_V(z)}
	\leq C \left( \frac{\abs{P(z)}}{\sqrt{\kappa}} \wedge \abs{P(z)}^{1/2} \right).
\end{equation}
Using the bound $\Upsilon \leq C \abs{P(z)}/\sqrt{\kappa}$ and \eqref{eq:bound_P_Upsilon}, we get
\begin{equation} \label{eq:first_bound_Upsilon}
	\E[\Upsilon^{2q}]
	\leq \frac{(Cq)^q}{(N\eta)^{2q} \kappa^q} 
	\left( 
		\E[\Upsilon^{2q}] 
		+ \kappa^{q/2} \E[\Upsilon^{q}] 
		+ \eta^q
	\right)
	+ \frac{(Cq)^{2q}}{N^{2q} \kappa^q}.
\end{equation}
We fix $C' \coloneqq 2C$ where $C$ is the constant appearing in the last equation.
We distinguish cases.
If $\eta \geq (C'q)^{1/2}/(N \sqrt{\kappa})$, the factor $\frac{(Cq)^q}{(N\eta)^{2q} \kappa^q}$ in \eqref{eq:first_bound_Upsilon} is smaller than $1/2$ so with can bring the term involving $\E[\Upsilon^{2q}]$ to the left-hand side and get
\begin{align} 
	\E[\Upsilon^{2q}]
	\leq 
	\frac{(Cq)^q}{(N\eta)^{2q} \kappa^{q/2}} \E[\Upsilon^{q}] 
	+ \frac{(Cq)^q}{N^{2q} \eta^q \kappa^q} 
	+ \frac{(Cq)^{2q}}{N^{2q} \kappa^q}.
\end{align}
Using that $x \leq a x^{1/2} + b$ for some $a,b,x > 0$ implies $x \leq a^2 + b$, it proves \eqref{eq:bound_1_s-m_1_s-m_2}.
We now consider the second case: if $\eta \leq (C'q)^{1/2}/(N \sqrt{\kappa})$, we use $\Upsilon \leq C \abs{P(z)}^{1/2}$ and \eqref{eq:bound_P_Upsilon} to get
\[
	\E[\Upsilon^{2q}]
	\leq C^q \cdot \E\left[ \abs{P(z)}^q \right]
	\leq \frac{(Cq)^{q/2}}{(N\eta)^q} 
	\left( 
	\E[\Upsilon^{q}]
	+ \kappa^{p/4} \E[\Upsilon^{q/2}]
	+ \eta^{q/2}
	\right)
	+ \frac{(Cq)^{q}}{N^{q}},
\]
Using that $x \leq a x^{1/2} + b x^{1/4} +c$ for some $a,b,c,x > 0$ implies $x \leq a^2 + b^{4/3} + c$, we get
\begin{equation} \label{eq:second_bound_Upsilon}
	\E[\Upsilon^{2q}]
	\leq \frac{(Cq)^q}{(N\eta)^{2q}} 
	+ \frac{(Cq)^{2q/3}}{(N\eta)^{4q/3}} \kappa^{q/3}
	+ \frac{(Cq)^{q/2}}{N^q \eta^{q/2}}
	+ \frac{(Cq)^q}{N^q}.
\end{equation}
Then, using $\eta \leq (C'q)^{1/2}/(N \sqrt{\kappa})$ (note that, since $\eta \leq \kappa$, it implies $\eta \leq (C'q)^{1/3} N^{-2/3}$), we observe that the second and third term in the right-hand side of \eqref{eq:second_bound_Upsilon} can be bounded by the first one.
This proves \eqref{eq:bound_2_s-m_1_s-m_2}.
Finally, it follows from \eqref{eq:stability_lemma_bound_2} that the bounds \eqref{eq:bound_1_s-m_1_s-m_2} and \eqref{eq:bound_2_s-m_1_s-m_2} hold for $\E[ \abs{\im(s(z)-m_V(z))}^{2q} ]$ instead of $\E[\Upsilon^{2q}]$.
\end{proof}

\section{Consequences of the local law}
\label{sec:consequences}

In this section, we apply the local law to establish various results: negligible expected number of particles in a submicroscopic interval, rigidity at the edge, extension of the local law beyond the trapezoid region and  rigidity in the bulk.

\subsection{Wegner estimate. }\label{subsec:number}
The following estimate will be used in Section \ref{section:CLT_above_the_axis} to prove that we can regularize the logarithm in the proof of the central limit theorem for the logarithm of the characteristic polynomial.
\begin{proposition} \label{prop:density} 
    Let $\tilde{\eta} >0$ be the constant given by Lemma \ref{lem:preliminaries_fixed_point_equation}
    and recall the definition of $\ell(E)$ in \eqref{def:l(E)_and_kappa(E)}.
    Let $I = [E-\delta_N\ell(E),E+\delta_N \ell(E)]$ for some $E \in [A-\tilde{\eta},B+\tilde{\eta}]$ and $\delta_N \to 0$. 
    Let $\cN(I) \coloneqq \abs{\{ k : \lambda_k \in I \}}$ be the number of particles in $I$. 
    Then $\E[\cN(I)] \to 0$ as $N \to \infty$ uniformly in $E$.
\end{proposition}
Note that $\ell(E)$ is exactly the microscopic scale (that is the typical spacing between particles) at a point $E$ between the edges, but it is larger than the microscopic scale past the edges.

\begin{remark} 
Instead of Proposition \ref{prop:density}, we could prove the following quantitative Wegner estimate: There exists $C > 0$ such that for any $N \geq 1$ and any interval $I \subseteq \R$, $\E[\cN(I)] \leq CN \abs{I}$.
The proof of this claim relies on the observation that, for $\varepsilon \ll 1$, a translation of all particles by $\varepsilon/N$ to the left or to the right, depending on the sign of $\sum_{k=1}^N V'(\lambda_k)$ in the current configuration, results in a negligible change on the density of particles \eqref{eq:density_of_particles}.

However, this explicit bound only catches the size order of $\cN(I)$ in the bulk, and we need estimates up to the edges.
That is why we have to adopt a different strategy for the proof of Proposition \ref{prop:density}: in the proof of Lemma \ref{lem:inductive_step_density}, we translate particles only locally.
This results in a smaller change in the part $\exp(- \frac{\beta N}{2} \sum_{k=1}^N V(\lambda_k))$ of the density, but requires to deal with changes in the Vandermonde determinant part.
\end{remark}

To prove Proposition \ref{prop:density}, we will use bounds on the imaginary part of the Stieljes transform $s(z)$ which follow from the local law established in the previous section: Note that
if $I = [E-\eta,E+\eta]$, for some $E \in \R$ and $\eta > 0$, then, considering $z \coloneqq E + \ii \eta$, we have
\begin{align} \label{eq:bound_number_of_particles_in_terms_of_im_s}
	\cN(I) 
	= \sum_{j = 1}^N \1_{\abs{\lambda_j-E} \leq \eta}
	\leq \sum_{j = 1}^N \frac{2 \eta^2}{\abs{z-\lambda_j}^2}
	= 2 \eta N \cdot \im s(z).
\end{align}
The following lemma  proves that the number of particles in an interval of length $\ell(E)$ centered at $E$ has bounded moments. 
\begin{lemma} \label{lem:first_step_density}
Let $\tilde{\eta}$ be the constant given by Lemma \ref{lem:preliminaries_fixed_point_equation}.
There exists a constant $C > 0$ such that, for any $E \in [A-\tilde{\eta},B+\tilde{\eta}]$ and any $N,q \geq 1$, letting $I \coloneqq [E-\ell(E),E+\ell(E)]$,
\begin{align} \label{eq:all_moments_N(I)}
\E \left[ \cN(I)^q \right] \leq (Cq)^{q/2}.
\end{align}
In particular, for any $t \geq 0$, there exists a constant $C_t > 0$ such that, for any $E \in [A-\tilde{\eta},B+\tilde{\eta}]$ and any $N \geq 1$,
\[
\E \left[ e^{t \cN(I)} \right] \leq C_t.
\]
\end{lemma}
\begin{proof} 
We apply \eqref{eq:bound_number_of_particles_in_terms_of_im_s} with $\eta = \ell(E)$ to get $\cN(I) \leq 2 \eta N \cdot \im s(z)$, where $z \coloneqq E + \ii \eta$.
In order to bound moments of $\im s(z)$, we will distinguish cases depending on if $E$ is in the trapezoid region or not.
If $E \in [A-\eta,B+\eta]$, then it follows from \eqref{eq:bound_Lambda} that
\begin{align} \label{eq:bound_im_s_explicit}
	\E \left[ \absa{\im s(z)}^q \right] 
	\leq \frac{(Cq)^{q/2}}{(N\eta)^q} 
	+ \frac{(Cq)^{q/2}}{N^{q/2}} 
	+ C^q \abs{b(z)}^{q}.
\end{align}
Therefore, we get
\begin{align*} 
	\E \left[ \cN(I)^q \right] 
	\leq (2 \eta N)^q \cdot \E \left[ \absa{\im s(z)}^q \right] 
	\leq (Cq)^{q/2}
	+ (Cq)^{q/2} \eta^q N^{q/2} 
	+ C^q (N\eta)^q \abs{b(z)}^q.
\end{align*}
Note that $\eta = \ell(E) \leq N^{-2/3}$ so $\eta^q N^{q/2} \leq 1$.
Moreover, distinguishing between the cases $\kappa(E) \leq N^{-2/3}$ and $\kappa(E) > N^{-2/3}$, we have $\abs{b(z)} \leq C (N\eta)^{-1}$ in both cases.
Therefore, we get $(N\eta)^q \abs{b(z)}^q \leq C^q$ and this proves \eqref{eq:all_moments_N(I)} in this case. 
Now assume that $E \notin [A-\eta,B+\eta]$, so in particular $\eta = \ell(E) = N^{-2/3}$.
Then, it follows from Proposition \ref{prop:bound_s-m_1_s-m_2_past_the_edge} that
\begin{align} \label{eq:bound_im_s_explicit_2}
	\E \left[ \absa{\im (s(z)-m_V(z))}^q \right] 
	\leq \begin{cases}
	\frac{(Cq)^q}{(N\eta)^{2q} \kappa^{q/2}}
	+ \frac{(Cq)^{q/2}}{N^q \eta^{q/2} \kappa^{q/2}} 
	+ \frac{(Cq)^q}{N^q \kappa^{q/2}} 
	& \text{if } \eta \geq (C'q)^{1/2}/ (N \sqrt{\kappa}), \\
	\frac{(C q)^{q/2}}{(N\eta)^q} + \frac{(Cq)^{q/2}}{N^{q/2}} 
	& \text{if } \eta \leq (C'q)^{1/2}/ (N \sqrt{\kappa}),
	\end{cases}
\end{align}
where $\kappa = \kappa(E)$. 
Using that $(C'q)^{1/2} \leq  N \eta \sqrt{\kappa}$ in the first case of \eqref{eq:bound_im_s_explicit_2} and that $\eta = N^{-2/3}$, we get
\[
	\E \left[ \absa{\im (s(z)-m_V(z))}^q \right] 
	\leq \frac{(C q)^{q/2}}{(N\eta)^q}.
\]
Since $\absa{\im m_V(z)} \leq C \eta/ \sqrt{\kappa} \leq C/(N\eta)$ by (\ref{eq:preliminary_lemma_bound_4}) and $\cN(I) \leq 2 \eta N \cdot \im s(z)$, this proves \eqref{eq:all_moments_N(I)} in this case. The second case of \eqref{eq:bound_im_s_explicit_2} 
is clear.
The second bound of the lemma follows by writing the exponential as a series.
\end{proof}

We will prove Proposition \ref{prop:density} by induction. The previous lemma is our base case and the inductive step is based on the following lemma, which shows that, if we reduce sufficiently the length of an interval, the expected number of particles will be reduced at least by a factor $\frac{3}{4}$, say.
\begin{lemma} \label{lem:inductive_step_density}
Let $\tilde{\eta}$ be the constant given by Lemma \ref{lem:preliminaries_fixed_point_equation} and $E \in [A-\tilde{\eta},B+\tilde{\eta}]$.
Let $I' \subseteq I$ be intervals centered at $E$ such that $\abs{I} \leq 2 \ell(E)$ and let 
\[
	\theta \coloneqq \E \left[ e^{\cN(I)}-1 \right].
\]
There exists a constant $C_0 > 0$, depending only on $V$ and $\beta$, such that, if $\abs{I'} = \abs{I}/M$ with $M \geq C_0 \theta^{-1/2}$ and if $N \geq C_0 \theta^{-3}$, then
\[
	\E \left[ e^{\cN(I')}-1 \right] 
	\leq \frac{3 \theta}{4}.
\]
\end{lemma}
\begin{proof}
For the sake of contradiction, we assume that $\E[ e^{\cN(I')}-1] > \frac{3 \theta}{4}$.

\textit{First step: restrict ourselves to the case where $\cN(I \setminus I') = 0$.}
We have
$$
	1 + \theta  = \E \left[ e^{\cN(I')+\cN(I \setminus I')}\right] 
	 \geq \E \left[ e^{\cN(I')} \1_{\cN(I \setminus I') = 0} \right]
	+ \E \left[ e^{\cN(I') + 1} \1_{\cN(I \setminus I') > 0} \right] 
	 = (1-e) \cdot \E \left[ e^{\cN(I')} \1_{\cN(I \setminus I') = 0} \right]
	+ e \cdot \E \left[ e^{\cN(I')} \right].
$$
Using our assumption $\E[ e^{\cN(I')}] > 1+ \frac{3}{4} \theta$, we get
\begin{align*}
	\E \left[ e^{\cN(I')} \1_{\cN(I \setminus I') = 0} \right]
	> \frac{e (1+\frac{3}{4} \theta) - 1 -\theta}{e-1}
	= 1 + \frac{\frac{3e}{4} -1}{e - 1} \cdot \theta 
	> 1 + \frac{\theta}{2}. 
\end{align*}
This implies that
\begin{align}
	\E \left[ (e^{\cN(I')}-1) \1_{\cN(I \setminus I') = 0} \right]
	> \frac{\theta}{2}. \label{eq:result_step_1}
\end{align}

\textit{Second step: translate eigenvalues in $I'$.} Let $\delta \coloneqq \abs{I'}/2$, so that $I' = [E-\delta,E+\delta]$.
Our goal is to translate eigenvalues in $I'$ by $r\delta$ for some  $r \in \{ -6,-4,-2,2,4,6 \}$.
Recall $I = [E-M\delta,E+M\delta]$ and we can assume $M \geq 20$, therefore the translated interval $I'+r\delta$ is still included in $I$ and far from the endpoints of $I$.
Let $h\colon \R \to \R$ denote a $C^1$-diffeomorphism such that $h(x) = x$ for $x \notin I$ and $h(x) = x + r\delta$ if $x \in I'$. 
Let 
\begin{equation} \label{eq:density_of_particles}
	f(\lambda_1,\dots,\lambda_N) 
	\coloneqq \frac{1}{Z_N}
	e^{- \frac{\beta N}{2} \sum_{k=1}^N V(\lambda_k)}
	\left(
	\prod_{1 \leq j < k \leq N} \absa{\lambda_k - \lambda_j}^\beta 
	\right)
\end{equation}
denote the density of the particles.
Note that if $\lambda_j \notin I$ and $\lambda_k \in I'$, then 
\begin{align*}
	\absa{ \frac{h(\lambda_k) - h(\lambda_j)}{\lambda_k - \lambda_j} }
	= \frac{\lambda_k+r\delta - \lambda_j}{\lambda_k - \lambda_j} 
	= 1 + \frac{r\delta}{\lambda_k-\lambda_j} 
	= 1 + \frac{r\delta}{E-\lambda_j} 
	+ \OO \left( \frac{\delta^2}{\abs{E-\lambda_j}^2} \right),
\end{align*}
where we used that $\abs{\lambda_k-E} \leq \delta \leq \abs{\lambda_j-E}/M$ with $M \geq 20$, so that the $\OO(\dots)$ terms do not depend on any parameter (recall $\abs{r} \leq 6$).
We have, on $\{ \cN(I \setminus I') = 0 \}$,
\begin{align*}
	& \frac{f(h(\lambda_1),\dots,h(\lambda_N))}{f(\lambda_1,\dots,\lambda_N)} \\
	& =	e^{- \frac{\beta N}{2} \sum_{k=1}^N (V(h(\lambda_k)) - V(\lambda_k))}
	\prod_{\lambda_j \notin I,  \lambda_k \in I'} 
	\left( 
	1 + \frac{r\delta}{E-\lambda_j} 
	+ \OO \left( \frac{\delta^2}{\abs{E-\lambda_j}^2} \right)
	\right)^\beta \\
	& = \exp \left( - \frac{\beta N}{2} 
	\sum_{k=1}^N \left( r \delta V'(E) + \OO(\delta^2) \right) 
	\1_{\lambda_k \in I'} \right)
	\prod_{\lambda_j \notin I}
	\left( 
	1 + \frac{r\delta}{E-\lambda_j} 
	+ \OO \left( \frac{\delta^2}{\abs{E-\lambda_j}^2} \right)
	\right)^{\beta \cN(I')} \\
	& = \exp \left( \beta \cN(I') r \delta 
	\left( 
	- \frac{N V'(E)}{2} 
	+ \OO(\delta)
	+ \sum_{\lambda_j \notin I} 
	\frac{1}{E-\lambda_j} 
	+ \OO \left( \frac{\delta}{\abs{E-\lambda_j}^2} \right)
	 \right) 
	\right),
\end{align*}
where the $\OO(\dots)$ terms depend only on $V$ (we used in particular that $V''$ was bounded on $I' \subseteq [A-1,B+1]$).
We consider the event
\[
	A \coloneqq \left\{ \sum_{\lambda_j \notin I} 
	\frac{1}{E-\lambda_j} \geq \frac{NV'(E)}{2} \right\}.
\]
On the event $A$, we choose $r > 0$ and, on the event $A^c$, we choose $r<0$.
Therefore, in both cases, we have, on the event $A \cap \{ \cN(I \setminus I') = 0 \}$ or $A^c \cap \{ \cN(I \setminus I') = 0 \}$,
\begin{align}
	\frac{f(h(\lambda_1),\dots,h(\lambda_N))}{f(\lambda_1,\dots,\lambda_N)}
	& \geq \exp \left( -C \cN(I') \delta^2 N
	\left( 
	1 + \frac{1}{N} \sum_{\lambda_j \notin I} \frac{1}{\abs{E-\lambda_j}^2}
	\right) 
	\right), \label{eq:lower_bound_change_density}
\end{align}
where $C$ depends only on $V$ and $\beta$.
Letting $z \coloneqq E + \ii M \delta$, note that
\begin{align*}
	\frac{1}{N} \sum_{\lambda_j \notin I} \frac{1}{\abs{E-\lambda_j}^2}
	\leq \frac{1}{N} \sum_{j=1}^N \frac{2}{\abs{z-\lambda_j}^2}
	= \frac{2}{M \delta} \im s(z).
\end{align*}
Hence, setting $Y \coloneqq C \cN(I') \delta^2 N
(1 + \frac{1}{M \delta} \im s(z))$, the right-hand side of \eqref{eq:lower_bound_change_density} is larger than $e^{-Y}$.
Therefore, in the case $r>0$, we have
\begin{align*}
	\E \left[ (e^{\cN(I')}-1) e^{-Y} 
	\1_{A \cap \{ \cN(I \setminus I') = 0 \}} 
	\right] 
	& \leq \int_{\R^N}
	(e^{\cN(I')}-1) 
	\1_{\cN(I \setminus I') = 0} 
	f(h(\lambda_1),\dots,h(\lambda_N))
	\diff \lambda_1 \dots \diff \lambda_N \\
	& = \int_{\R^N} 
	(e^{\cN(h(I'))} -1)
	\1_{\cN(I \setminus h(I')) = 0}
	f(\lambda_1,\dots,\lambda_N)
	\diff \lambda_1 \dots \diff \lambda_N,
\end{align*}
where we applied a change of variable, replacing $h(\lambda_j)$ by $\lambda_j$, using that $h$ fixes $I$.
Therefore, we proved, for $r \in \{2,4,6\}$,
\begin{align*} 
	\E \left[ (e^{\cN(I'+r \delta)} -1)
	\1_{\cN(I \setminus (I'+r \delta)) = 0}
	\right] 
	\geq \E \left[ (e^{\cN(I')}-1) e^{-Y} 
	\1_{A \cap \{ \cN(I \setminus I') = 0 \}} 
	\right] 
\end{align*}
and the same inequality holds for $r \in \{-6,-4,-2\}$ by replacing $A$ by $A^c$.
Therefore, we get for any $r \in \{2,4,6\}$,
\begin{align} 
	& \E \left[ (e^{\cN(I'+r \delta)} -1)
	\1_{\cN(I \setminus (I'+r \delta)) = 0} \right] 
	+ \E \left[ (e^{\cN(I'-r \delta)} -1)
	\1_{\cN(I \setminus (I'-r \delta)) = 0} \right] 
	\nonumber \\
	& \geq \E \left[ (e^{\cN(I')}-1) e^{-Y} 
	\1_{\cN(I \setminus I') = 0} \right]. 
	\label{eq:result_step_2}
\end{align}

\textit{Third step: getting a lower bound for $\E[ (e^{\cN(I')}-1) e^{-Y} 
	\1_{\cN(I \setminus I') = 0}]$.} 
Recall from \eqref{eq:result_step_1} that we proved $\E[ (e^{\cN(I')}-1) \1_{\cN(I \setminus I') = 0}]
> \frac{\theta}{2}$.
Using that $1-e^{-y} \leq y$ for $y \geq 0$, we have
\begin{align*} 
	\E \left[ (e^{\cN(I')}-1) \1_{\cN(I \setminus I') = 0} \right]
	- \E \left[ (e^{\cN(I')}-1) e^{-Y} \1_{\cN(I \setminus I') = 0} \right] 
	\leq 
	\E \left[ Y e^{\cN(I')} \right].
\end{align*}
Recall $z \coloneqq E + \ii M \delta$.
By \eqref{eq:bound_im_s_explicit} and Lemma \ref{lem:first_step_density} (we use here the assumption $M \delta = \abs{I} \leq 2 \varepsilon$),
we have
\[
	\E \left[ (\im s(z))^4 \right] 
	\leq \frac{C}{(NM\delta)^4},
	\qquad
	\E \left[ \cN(I')^4 \right] \leq C
	\qquad \text{and} \qquad
	\E \left[ e^{2\cN(I')} \right] \leq C.
\]
Recalling $Y = C \cN(I') (\delta^2 N + \frac{1}{M^2} N M \delta \im s(z))$, it follows that (we use $\delta\leq 2N^{-2/3}$)
\[
	\E \left[ Y^2 \right] 
	\leq C \left( (\delta^2 N)^2 + \frac{1}{M^4} \right)
	\leq C \left( N^{-2/3} + \frac{1}{M^4} \right).
\]
Therefore, applying Cauchy--Schwarz inequality, we get
\[
	\E \left[ Y e^{\cN(I')} \right]
	\leq C \left( N^{-1/3} + \frac{1}{M^2} \right)
	\leq \frac{\theta}{6},
\]
for $M \geq C_0 \cdot \theta^{-1/2}$ and for $N \geq C_0 \cdot \theta^{-3}$ where $C_0$ is a well-chosen constant depending only on $\beta$ and $V$.
We assume from now on that $M$ and $N$ satisfy these inequalities.
Therefore, we have shown
\begin{align} 
	\E \left[ (e^{\cN(I')}-1) e^{-Y} \1_{\cN(I \setminus I') = 0} \right] 
	\geq \frac{\theta}{3}.  \label{eq:result_step_3}
\end{align}

\textit{Fourth step: conclusion.} 
Note that, since the events $\{ \cN(I \setminus (I'+r \delta)) = 0 \} \cap \{ \cN(I'+r \delta) > 0\}$ are pairwise disjoint,
\begin{align*} 
	\theta = \E \left[ e^{\cN(I)}-1 \right] 
	& \geq \sum_{r \in \{ -6,-4,-2,0,2,4,6 \}}
	\E \left[ (e^{\cN(I'+r \delta)} -1)
	\1_{\cN(I \setminus (I'+r \delta)) = 0} \right] \\
	& \geq \E \left[ (e^{\cN(I')}-1) \1_{\cN(I \setminus I') = 0} \right]
	+ 3 \E \left[ (e^{\cN(I')}-1) e^{-Y} \1_{\cN(I \setminus I') = 0} \right],
\end{align*}
applying \eqref{eq:result_step_2}.
Using \eqref{eq:result_step_1} and \eqref{eq:result_step_3},
this shows $\theta > \frac{3 \theta}{2}$, which is our contradiction.
\end{proof}
We finally conclude the section by proving Proposition \ref{prop:density}.
\begin{proof}[Proof of Proposition \ref{prop:density}]
The result is proved by induction, using the second bound of Lemma \ref{lem:first_step_density} as base case and Lemma \ref{lem:inductive_step_density} for the inductive step.
\end{proof}

\subsection{Rigidity past the edge. }
\label{section:rigidity_past_the_edge}

The goal of this section is to use the bounds obtained in Proposition \ref{prop:bound_s-m_1_s-m_2_past_the_edge} to prove the rigidity estimate for the leftmost and the rightmost particles stated in Corollary \ref{cor:rigidity_past_the_edge}.
\begin{proof}[Proof of Corollary \ref{cor:rigidity_past_the_edge}]
By symmetry between the leftmost and the rightmost particles, we can focus on bounding $\P(\lambda_N > B + K N^{-2/3})$.
Let $\tilde{\eta} > 0$ be the constant given by Lemma \ref{lem:preliminaries_fixed_point_equation}.
It follows from \eqref{eq:large_deviation_rigidity} that 
\[
\P(\lambda_N > B + \tilde{\eta}) \leq C e^{-cN},
\]
so it is enough to prove that there are constants $c,C>0$ such that, for any $N \geq 1$ and $K \in [1,\tilde{\eta} N^{2/3}]$,
\begin{align} \label{eq:goal_rigidity_past_the_edge}
\P(\lambda_N \in (B+K N^{-2/3}, B + \tilde{\eta}]) \leq C \exp \left( -c K^{3/4} \right).
\end{align}
First, we consider an interval $I = [E,E+\eta]$ for some $E=B+\kappa$ and $\eta,\kappa > 0$ such that $\eta \leq \kappa \leq \tilde{\eta}$. 
Set $z = E + \ii \eta$, and recall from \eqref{eq:bound_number_of_particles_in_terms_of_im_s} that $\cN(I) \leq 2 \eta N \im s(z)$.
Moreover by (\ref{eq:preliminary_lemma_bound_4}), there is a constant $C_0 > 1$ such that $\im m_{V}(z) \leq (C_0^2 \eta)/(2 \sqrt{\kappa})$ for any $z$ with $\eta \leq \kappa \leq \tilde{\eta}$.
Hence, if we assume that $\eta \leq \kappa^{1/4}/(C_0 \sqrt{N})$, we have $\im m_{V}(z) \leq 1/(4N\eta)$. Using that $\cN(I)$ is a nonnegative integer, it follows that
\[
	\cN(I) \leq (4 \eta N) \abs{\im (s(z) -m_{V}(z))}.
\]
Assume that $\eta \geq (C'q)^{1/3} N^{-2/3}$ for some integer $q \geq 0$ and $C'$ the constant given by Proposition \ref{prop:bound_s-m_1_s-m_2_past_the_edge}.
In particular, we have $\eta \geq (C'q)^{1/2}/ (N \sqrt{\kappa})$ and therefore the first bound of Proposition \ref{prop:bound_s-m_1_s-m_2_past_the_edge} and 
$\eta \geq (C'q)^{1/3} N^{-2/3}$ give
\begin{align*} 
	\E \left[ \abs{\im (s(z) -m_{V}(z))}^q \right]
	\leq 
	\frac{(Cq)^q}{(N\eta)^{2q} \kappa^{q/2}}
	+ \frac{(Cq)^{q/2}}{N^q \eta^{q/2} \kappa^{q/2}} 
	+ \frac{(Cq)^q}{N^q \kappa^{q/2}}	
	\leq 
	\frac{(Cq)^{q/2}}{N^q \eta^{q/2} \kappa^{q/2}} + \frac{(Cq)^q}{N^q \kappa^{q/2}}.
\end{align*}
Hence, we proved that, if $(C'q)^{1/3} N^{-2/3} \leq \eta \leq \kappa^{1/4}/(C_0 \sqrt{N})$, then with $C$ depending on $\tilde{\eta}$,
since $\kappa \leq \tilde{\eta}$, we have
\begin{align} \label{eq:bound_N(I)}
	\E [\cN(I)^q ]
	\leq \left( \frac{C q \eta}{\kappa} \right)^{q/2}
	+ \frac{(C q \eta)^q}{\kappa^{q/2}}
	\leq \left( \frac{C q \eta}{\kappa} \right)^{q/2}
	+ \left( \frac{C q \eta}{\kappa} \right)^q.
\end{align}
We will apply this inequality to a well-chosen sequence of intervals.
Let $L \geq C_0$ be a parameter that we fix subsequently,
and define recursively the sequence $a_0 \coloneqq K$ and $a_{j+1} \coloneqq a_j + a_j^{1/4}/L$. 
Then, for any $j \geq 0$, let
\begin{align} \label{eq:def_I_j}
	\kappa_j \coloneqq a_j N^{-2/3},
	\qquad
	E_j \coloneqq B+\kappa_j,
	\qquad
	\eta_j \coloneqq \frac{a_j^{1/4}}{L} N^{-2/3}
	\qquad \text{and} \qquad 
	q_j \coloneqq 
	\left\lfloor \frac{a_j^{3/4}}{L^3 C'} \right\rfloor.
\end{align} 
Since $(C' q_j)^{1/3} N^{-2/3} \leq \eta_j \leq \kappa_j^{1/4}/(C_0 \sqrt{N})$, 
choosing $L$ large enough depending on $C$ and $C'$,
\eqref{eq:bound_N(I)} applied to the interval $I_j \coloneqq (E_j,E_j+\eta_j]$ with exponent $q_j$ gives
\begin{align} \label{eq:bound_N(I_j)}
	\E [\cN(I_j)^{q_j} ]
	\leq \left( \frac{C q_j \eta_j}{\kappa_j} \right)^{q_j/2}
	+ \left( \frac{C q_j \eta_j}{\kappa_j} \right)^{q_j}
	\leq \left( \frac{C}{L^4 C'} \right)^{q_j/2}
	+ \left( \frac{C}{L^4 C'} \right)^{q_j}
	\leq e^{-q_j}.
\end{align} 
Noting that $(B+KN^{-2/3},B+\tilde{\eta}] \subseteq \bigcup_{j=1}^{j_{\max}} I_j$ with $j_{\max} \coloneqq \max \{ j \geq 1: \kappa_j < \tilde{\eta} \}$, we find
\begin{align*}
	\P(\lambda_N \in (B+K N^{-2/3}, B + \tilde{\eta}])
	& \leq \sum_{j=1}^{j_{\max}} \P(\cN(I_j) \geq 1)
	\leq \sum_{j=1}^{j_{\max}} \E [ \cN(I_j)^{q_j} ]
	\leq \sum_{j=1}^{j_{\max}} e^{-q_j}.
\end{align*}
Since \eqref{eq:goal_rigidity_past_the_edge} follows directly when $K < L$, we now assume $K \geq L$. 
Using the definition of $q_j$ and that $a_j \geq K+jK^{1/4}/L$, using the change of variables $y = (K+xK^{1/4}/L)^{3/4}/(L^3 C')$, we get
\begin{multline}
	\P(\lambda_N \in (B+K N^{-2/3}, B + \tilde{\eta}]) \leq \sum_{j=1}^\infty \exp \left( -\frac{(K+jK^{1/4}/L)^{3/4}}{L^3 C'} +  1 \right)  
	 \leq \int_0^\infty \exp \left( -\frac{(K+xK^{1/4}/L)^{3/4}}{L^3 C'} +  1 \right) \diff x \\
	 = \frac{4e L^5 (C')^{4/3}}{3K^{1/4}} \int_{K^{3/4}/(L^3 C')}^\infty y^{1/3} e^{-y} \diff y.
	\label{eq:final_bound_sum_N(I_j)}
\end{multline}
Finally, $\int_b^\infty y^{1/3} e^{-y} \diff y \leq C(b^{1/3}+1) e^{-b}$ for $b\geq 0$ proves
 \eqref{eq:goal_rigidity_past_the_edge} and concludes the proof.
\end{proof}

\subsection{Extending the local law past the edge. }

In this section, we 
prove Proposition \ref{prop:local_law_past_the_edge} which
extends the local law outside the trapezoid where Theorem \ref{th:local_law_bulk} holds.
Note that the dependence of the constant in $q$ is worse in this result than in Theorem \ref{th:local_law_bulk}.
\begin{proposition} \label{prop:local_law_past_the_edge}
	Let $\beta >0$.
    There exist $\tilde{\eta}>0$ and $C>0$ such that for any $q \geq 1$, $N \geq 1$ and any $z = E + \ii \eta$ 
    with $0 < \eta \leq \tilde{\eta}$ and $A-\tilde{\eta} \leq E \leq B+\tilde{\eta}$, we have
	\[
		\E\left[ \absa{ s_N(z) - m_V(z) }^q \right]
		\leq \frac{(Cq)^{2q}}{(N\eta)^q}.
	\]
\end{proposition}
In order to prove Proposition \ref{prop:local_law_past_the_edge}, we 
first prove Lemma \ref{lem:number_of_eigenvalues_past_the_edge} which
bounds the number of particles past the edge.
We do not seek an optimal bound. Indeed, 
with a better treatment of the first part of the interval in the proof,
we could improve the bound below to $(Cq)^{5q/4}$, which would improve Proposition \ref{prop:local_law_past_the_edge} as well.
\begin{lemma} \label{lem:number_of_eigenvalues_past_the_edge}
Let $\beta >0$. There is a constant $C >0$ such that, for any $q \geq 1$,
\begin{align*} 
\E \left[ \cN((B,\infty))^q \right] \leq (Cq)^{2q}.
\end{align*}
\end{lemma}
\begin{proof}
We consider the constant $C'\geq 1$ from Proposition \ref{prop:bound_s-m_1_s-m_2_past_the_edge} and the constant $L\geq 1$ appearing in the proof of Corollary \ref{cor:rigidity_past_the_edge}.
We set $K \coloneqq (L^3C'q)^{4/3}$ and divide the interval $(B,\infty)$ into three parts,
\[
(B,\infty) = (B,B+K N^{-2/3}] \cup I \cup (B+\tilde{\eta},\infty),
\]
where $I \coloneqq (B+K N^{-2/3},B+\tilde{\eta}]$.

\textit{First part: interval $(B,B+K N^{-2/3}]$.} 
We set $\eta \coloneqq K N^{-2/3}$ and $z \coloneqq B+\eta+\ii \eta$. By \eqref{eq:bound_number_of_particles_in_terms_of_im_s}, we have
\begin{align*}
	\E \left[ \cN((B,B+K N^{-2/3}])^q \right] 
	\leq (2 \eta N)^q \cdot \E \left[ \absa{\im s(z)}^q \right].
\end{align*}
With our choice of $K$ and since we are in the case $\kappa = \eta$, we have $\eta \geq (C'q)^{1/2}/ (N \sqrt{\kappa})$. Hence, applying the first bound of \eqref{eq:bound_im_s_explicit_2}, combined with the fact that $\absa{\im m_V(z)} \leq C \eta/ \sqrt{\kappa} = C \sqrt{\eta}$, we have
\begin{align*}
	\E \left[ \cN((B,B+K N^{-2/3}])^q \right] 
	\leq (2 \eta N)^q \cdot 
	\left( \frac{(Cq)^q}{N^{2q} \eta^{5q/2}}
		+ \frac{(Cq)^{q/2}}{(N\eta)^q} 
		+ \frac{(Cq)^q}{N^q \eta^{q/2}}
		+C^q \eta^{q/2}\right)
	\leq (Cq)^{2q}.
\end{align*}

\textit{Second part: interval $I$.} We follow the proof of Corollary \ref{cor:rigidity_past_the_edge}, taking $K = (L^3C'q)^{4/3}$, and using the same notation and intervals $I_j$ to divide the interval $I$ (see \eqref{eq:def_I_j}).
Note in particular that $q_j \geq q_0 = q$ by our choice of $K$.
Applying Minkowski's inequality and then that $\cN(I_j)^q \leq \cN(I_j)^{q_j}$ since $\cN(I_j)$ takes only nonnegative integer values, 
applying \eqref{eq:bound_N(I_j)}, we get
\begin{align*}
	\E \left[ \cN(I)^q \right]^{1/q}
	\leq \sum_{j=1}^{j_{\max}} \E \left[ \cN(I_j)^q \right]^{1/q}
	\leq \sum_{j=1}^{j_{\max}} \E \left[ \cN(I_j)^{q_j} \right]^{1/q}
	\leq \sum_{j=1}^{j_{\max}} e^{-q_j/q}.
\end{align*}
Then, proceeding as in \eqref{eq:final_bound_sum_N(I_j)},
using that $q_j \geq a_j^{3/4}/(L^3C')-1$ and $a_j \geq K+jK^{1/4}/L$, we get
\begin{align*}
	\E \left[ \cN(I)^q \right]^{1/q}
	\leq \sum_{j=1}^\infty \exp \left( -\frac{(K+jK^{1/4}/L)^{3/4}}{L^3 C' q} + \frac{1}{q} \right)
	\leq \frac{C q^{4/3}}{K^{1/4}}.
\end{align*}
Recalling the definition of $K$, we proved that $\E \left[ \cN(I)^q \right] \leq (Cq)^q$.

\textit{Third part: interval $(B+\tilde{\eta},\infty)$.} Using the large deviation estimate \eqref{eq:large_deviation_rigidity}, we have
\begin{align*}
	\E \left[ \cN((B+\tilde{\eta},\infty))^q \right]
	\leq N^q \cdot \P( \lambda_N \geq B+\tilde{\eta} )
	\leq N^q e^{-cN} 
	\leq (Cq)^q,
\end{align*}
where the last inequality follows from the same argument as the one following \eqref{eq:almost_final_bound_Delta}. This concludes the proof.
\end{proof}
\begin{proof}[Proof of Proposition \ref{prop:local_law_past_the_edge}]
Note that in the trapezoid, $A-\eta \leq E \leq B+\eta$, the result follows directly from Theorem \ref{th:local_law_bulk}.
Therefore, we focus on the case $B+\eta < E \leq B+\tilde{\eta}$, the case $A-\tilde{\eta} \leq E \leq A-\eta$ being treated similarly.

For any $r \geq 0$, set $z_r \coloneqq B+\eta+r+\ii\eta$ and let $t \coloneqq E-B-\eta$ so that $z_t = z$.
Note that $\re(z_0) = B + \im (z_0)$ so we can apply Theorem \ref{th:local_law_bulk} at the point~$z_0$.
Note further that
\begin{align*}
	\absa{s(z) - \widetilde{m}_V(z)} 
	&
	\geq \re (s(z) - s(z_0))
	+ \re (s(z_0) - m_V(z_0)) 
	+ \re (m_V(z_0) - \widetilde{m}_V(z))
\end{align*}
and, by \eqref{eq:preliminary_lemma_bound_5},
\begin{align*}
	\re (m_V(z_0) - \widetilde{m}_V(z)) 
	& \geq c \abs{r(z) b(z)}
	= c \abs{m_V(z)-\widetilde{m}_V(z)}.
\end{align*}
Hence, we have
\begin{align*}
	\abs{m_V(z)-\widetilde{m}_V(z)} \leq C \left( \absa{s(z) - \widetilde{m}_V(z)} + \re (s(z_0) - s(z)) + \abs{s(z_0) - m_V(z_0)} \right).
\end{align*}
Distinguishing between the cases $\abs{s-m_V} \leq \abs{s-\widetilde{m}_V}$ and $\abs{s-\widetilde{m}_V} \leq \abs{s-m_V}$ and, in the second case, using $\abs{s-m_V} \leq \abs{s-\widetilde{m}_V} + \abs{\widetilde{m}_V-m_V}$ and the previous bound, we conclude that
\begin{align*}
	\abs{s(z)-m_V(z)} 
	\leq C \left( 
	\abs{s(z)-m_V(z)} \wedge \abs{s(z) - \widetilde{m}_V(z)} 
	+ \re (s(z_0) - s(z)) 
	+ \abs{s(z_0) - m_V(z_0)} \right).
\end{align*}
Applying Proposition \ref{prop:bound_s-m_1_s-m_2_past_the_edge} for the first term on the right-hand side, and Theorem \ref{th:local_law_bulk} for the third term, we get
\begin{align} \label{eq:bound_s-m_V_first_step}
	\E \left[ \abs{s(z)-m_V(z)}^q \right]
	\leq \frac{(Cq)^{2q}}{(N\eta)^q}
	+ C \E \left[ ([\re (s(z_0) - s(z))]_+)^q \right],
\end{align}
where we set $x_+ \coloneqq \max (x,0)$.
We now bound $[\re (s(z_0) - s(z))]_+$.
Since $\abs{\lambda_j-z_r} \geq \eta$, we have
\begin{align*}
	- \frac{\diff}{\diff r} \re s(z_r) 
	& = \frac{1}{N} \sum_{j=1}^N
	\frac{\eta^2 - (B+\eta+r-\lambda_j)^2}{\abs{\lambda_j-z_r}^4}
	\leq \frac{1}{N} \sum_{j=1}^N
	\frac{1}{\eta^2}
	\1_{\abs{B+\eta+r-\lambda_j} < \eta},
\end{align*}
and since $\int_0^t \1_{\abs{B+\eta+r-\lambda_j}< \eta} \diff r \leq 2 \eta \1_{B < \lambda_j < E+\eta}$, we get
\begin{align*}
	\re (s(z_0) - s(z))
	& = - \int_0^t	\frac{\diff}{\diff r} \re s(z_r) \diff r
	\leq \frac{2}{N\eta} \sum_{j=1}^N \1_{\lambda_j > B}
	= \frac{2}{N\eta} \cN((B,\infty)).
\end{align*}
Inserting this in \eqref{eq:bound_s-m_V_first_step} and
applying Lemma \ref{lem:number_of_eigenvalues_past_the_edge} proves the result.
\end{proof}

\subsection{Rigidity in the bulk. }\label{subsec:rigidity}
We now prove Corollary \ref{cor:rigidity}.
The following lemma is classical and relies on the Helffer-Sj{\H o}strand formula, first used in random matrix theory in \cite{ErdRamSchYau2010}.

\begin{lemma}\label{lem:HS}
There exists $C>0$ such that, for any $0 < \eta < \gamma$ and $M>0$, 
for any real function~$f$ compactly supported in $[A,B]$, 
on the event $\{ \forall x \in [A,B], \forall y \in (0,\gamma], \abs{(s-m_V)(x+iy)} < M/(Ny) \}$, we have
$$
\absa{ \sum_{k=1}^N f(\lambda_k)-N\int f\rd\mu_V }
\leq CM\left(\frac{\|f\|_1}{\gamma}+\eta\|f''\|_1+\log(\gamma/\eta)\|f'\|_1\right).
$$
\end{lemma}

\begin{proof}
Let $\chi(x)=1$ on $[0,\gamma]$, $\chi(x)=0$ on $[2\gamma,\infty)$ and $\|\chi'\|_\infty<100/\gamma$.
From \cite[(B.13)]{ErdRamSchYau2010} and an integration by parts as in \cite[(B.17)]{ErdRamSchYau2010}, for some universal constant $C$, 
we have
\begin{align}
\absa{ \sum_{k=1}^N f(\lambda_k)-N\int f\rd\mu_V }
&\leq\ C((\rm I)+(II)+(III)+(IV)),\notag\\
(\rm I)&=N\iint_{y>0}(|f(x)|+y|f'(x)|)|\chi'(y)| |(s-m_V)(x+\ii y)|\rd x\rd y,\notag\\
(\rm II)&=N\iint_{0<y<\eta} y|f''(x)||(s-m_V)(x+\ii y)|\rd x\rd y,\notag\\
(\rm III)&=N\iint_{\eta<y} |\partial_y(y\chi(y))||f'(x)||(s-m_V)(x+\ii y)|\rd x\rd y,\notag\\
(\rm IV)&=N \int \eta \abs{f'(x)} \abs{(s-m_V)(x+\ii \eta)} \rd x \label{eqn:HSExpand}.
\end{align}
The result follows easily.
\end{proof}

We now prove the following direct consequence of Theorem \ref{th:local_law_bulk}, which gives an a priori rigidity estimate
with suboptimal logarithmic power.

\begin{lemma}\label{lem:rigid}
Recall that $\hat k=\min(k,N+1-k)$.
For any $r\geq 2$ there exists $N_0$ such that for any  $N>N_0$, we have 
$$
\mathbb{P}\left( \cap_{1\leq k\leq N}
\{|\lambda_k-\gamma_k|<(\log N)^{3r+2}N^{-\frac{2}{3}}(\hat k)^{-\frac{1}{3}}\}\right)
\geq 1-e^{-(\log N)^{3r/4}}.
$$
\end{lemma}

\begin{proof}
Consider $\mathscr{D}=\{A<\re z<B,0<\im z<(\log N)^{-2r}\}$, and define the event
\begin{equation}\label{eqn:set}
\mathscr{A}=\cap_{z\in \mathscr{D}}\left\{|s(z)-m_V(z)|\leq \frac{(\log N)^r}{N\eta}\right\}.
\end{equation}
We will first prove a lower bound for $\mathbb{P}(\mathscr{A})$.
First note that for $z=E+\ii\eta$ we have both
$$
\left|\frac{\rd}{\rd\eta}\re s(z)\right|\leq \left|\frac{\rd}{\rd\eta}s(z)\right|\leq \frac{1}{\eta}\im s(z),\ {\rm and}\ \im s(E+\ii\eta)\leq \frac{\eta_0}{\eta} \im s(E+\ii\eta_0),\ 0<\eta<\eta_0,
$$
so $\abs{\frac{\rd}{\rd\eta}\re s(z)} \leq \frac{\eta_0}{\eta^2} \im s(E+\ii\eta_0)$ for any $0<\eta<\eta_0$. 
This implies $|\re s(z)| < |\re s(E+\ii\eta_0)| 
+\frac{\eta_0}{\eta} |\im s(E+\ii\eta_0)|$. 
For $\eta_0=1/N$ we therefore obtain, assuming
$|(s-m_V)(E+\frac{\ii}{N})|<B$ and noting $|m_V|<C'=C'(V)$ uniformly in $\mathbb{C}$,
$$
|(s-m_V)(E+\ii\eta)|\leq \frac{5(C'+B)}{N\eta},\ 0<\eta<\frac{1}{N}.
$$
As a consequence, for $N$ large enough depending on $C'$, we have $\tilde{\mathscr{A}} \subset \mathscr{A}$, where we set
\begin{align}
\tilde{\mathscr{A}} := \cap_{z\in\tilde{\mathscr{D}}}\left\{|s(z)-m_V(z)|\leq \frac{(\log N)^r}{10 N\eta}\right\}, 
\quad \text{with} \quad 
\tilde{\mathscr{D}} := \mathscr{D}\cap \{\im z>1/N\}. \nonumber 
\end{align}
So we are now aiming for a lower bound for $\mathbb{P}(\tilde{\mathscr{A}})$.
For any $z\in \tilde{\mathscr{D}}$ and $q\geq 1$, Theorem \ref{th:local_law_bulk} and Markov's inequality imply 
$$
\mathbb{P}\left(|s(z)-m_V(z)|>\frac{u}{N\eta}\right)\leq u^{-q}
\left((Cq)^{q/2} 
			+ \frac{(Cq)^q\eta^q}{\abs{z-A}^{q/2} \abs{z-B}^{q/2}}\right).
$$
Choosing $q=\lfloor u^2/(Ce)\rfloor$, for some $\theta=\theta(C)$ we obtain, for any $u>1$,
$$
\mathbb{P}\left(|s(z)-m_V(z)|>\frac{u}{N\eta}\right)\leq \theta^{-1}e^{-\theta u^2}(1+(Cq\eta)^{q/2}).
$$
If we assume further that  $u^2\eta<1$, we obtain
\begin{equation}\label{eqn:GausTail}
\mathbb{P}\left(|s(z)-m_V(z)|>\frac{u}{N\eta}\right)\leq 2\theta^{-1}e^{-\theta u^2}.
\end{equation}
In particular, $u=(\log N)^r/20$ and $0<\eta<(\log N)^{-2r}$ give
$$
\mathbb{P}\left(|s(z)-m_V(z)|>\frac{(\log N)^r}{20 N\eta}\right)\leq 2\theta^{-1}e^{-\theta (\log N)^{2r}/400}.
$$
Note that $s$ and $m_V$ are $N^2$-Lipschitz on $\tilde{\mathscr{D}}$, so that $|s-m_V|\leq  \frac{(\log N)^r}{20N\eta}$ on $\tilde{\mathscr{D}}\cap N^{-3}\mathbb{Z}^2$ implies 
$|s-m_V|\leq  \frac{(\log N)^r}{10N\eta}$ on $\tilde{\mathscr{D}}$. This observation and a union bound yield
$$
\mathbb{P}\left(\tilde{\mathscr{A}}\right)
= 1 - \mathbb{P}\left(\cup_{z\in\tilde{\mathscr{D}}\cap N^{-3}\mathbb{Z}^2}
|s(z)-m_V(z)| \leq \frac{(\log N)^r}{10 N\eta}\right)
\geq 1-\theta^{-1}N^{6}e^{-\theta(\log N)^{2r}/4},
$$
which implies that
\begin{equation} \label{eqn:AP}
\mathbb{P}\left(\mathscr{A}\right)
\geq 1-\theta^{-1}N^{6}e^{-\theta(\log N)^{2r}/400}
\geq 1-e^{-(\log N)^{2r-1}},
\end{equation}
for any $N\geq N_0(C,C')$, using that $2r-1 >1$ in the second inequality.

Let $f_{E_1,E_2}$ denote an approximation of $\1_{[E_1,E_2]}$ on scale $N^{-1}$, i.e. $f(x)=0$ on $(-\infty,E_1]\cup[E_2,\infty)$, $1$ on $[E_1+N^{-1},E_2-N^{-1}]$, and $\|f^{(k)}\|_\infty<100 N^{k}$, $k=1,2$.
We consider
$$
\mathscr{B}=\cap_{A<E_1<E_2<B} \left\{ \absa{ \sum_i f_{E_1,E_2}(\lambda_i)-N\int f_{E_1,E_2}\rd\mu_V} < (\log N)^{3r+1} \right\}.
$$
From  Lemma \ref{lem:HS} with $M=(\log N)^r$, $\gamma=(\log N)^{-2r}$, $\eta=1/N$, we have $ \mathscr{A}\subset\mathscr{B}$ for $N$ large enough, so that
\begin{equation}\label{eqn:BP}
\mathbb{P}\left(\mathscr{B}\right)>1-e^{-(\log N)^{2r-1}}.
\end{equation}
Moreover, by Corollary \ref{cor:rigidity_past_the_edge}, for large enough $N$ we have
\begin{align} 
\P\left( \mathscr{C}\right) 
&\geq 1- \frac{1}{2} e^{-(\log N)^{3r/4}},\label{eqn:CP}\\
\mathscr{C} 
&\coloneqq \{\lambda_1>A-(\log N)^{r+1}N^{-2/3}\}
\cap\{\lambda_N<B+(\log N)^{r+1}N^{-2/3}\} \nonumber.
\end{align}
Moreover, observe that $\mathscr{B}\cap\mathscr{C} \subset \cap_{1\leq k\leq N} \{|\lambda_k-\gamma_k| <(\log N)^{3r+2} N^{-\frac{2}{3}}(\hat k)^{-\frac{1}{3}}\}$ for $N$ large enough, so the result follows from \eqref{eqn:BP} and \eqref{eqn:CP}.
\end{proof}


We now start improving on the $(\log N)^C/N$ rigidity from the previous lemma.
For any $E\in[A+N^{-2/3},B-N^{-2/3}]$, we denote $\eta_0=e^{(\log N)^{1/4}}/(N\sqrt{\kappa})$ and $\eta_1=1/(N\sqrt{\kappa})$ with $\kappa = \kappa(E) = \abs{E-A} \wedge \abs{E-B}$.
We define the function $f=f_E$ as follows, with the notation $\tilde\eta$ from Theorem \ref{th:local_law_bulk}:
\begin{align*}
& f=0 \text{ on } (-\infty,E]\cup[B+\tilde\eta/2,+\infty), 
\quad 
f=1 \text{ on } [E+\eta_1,B+\tilde\eta/4], \\
& \|f^{(k)}\|_{L^\infty(-\infty,B]}\leq 100 \cdot \eta_1^{-k},
\quad \|f^{(k)}\|_{L^\infty[B,\infty)}\leq 100 \cdot \tilde{\eta}^{-k},
\quad k=1,2
\end{align*}
Then, we introduce a similar function $f_0=(f_0)_E$, but which is smoothed at scale $\eta_0$ (instead of scale $\eta_1$ for $f$) close to $E$:
Letting $E_0 \coloneqq E-\eta_0$ if $E \geq \frac{A+B}{2}$ and $E_0 \coloneqq E$ if $E < \frac{A+B}{2}$, we choose
\begin{align*}
& f_0=0 \text{ on } (-\infty,E_0]\cup[B+\tilde\eta/2,+\infty), 
\quad 
f_0=1 \text{ on } [E_0+\eta_0,B+\tilde\eta/4], \\
& \|f_0^{(k)}\|_{L^\infty(-\infty,B]}\leq 100 \cdot \eta_0^{-k},
\quad \|f_0^{(k)}\|_{L^\infty[B,\infty)}\leq 100 \cdot \tilde{\eta}^{-k},
\quad k=1,2
\end{align*}
Moreover, we assume that $f_0 = f$ on $[B+\tilde{\eta}/4,B+\tilde{\eta}/2]$.
Therefore, the function $g \coloneqq f-f_0$ is non zero only on an interval 
of length at most $\eta_0+\eta_1$ and included in $[A,B]$, for any $E\in[A+N^{-2/3},B-N^{-2/3}]$ (as a consequence of our choice of $E_0$).
We first bound fluctuations of the linear statistics of $g:=f-f_0$.

\begin{lemma}\label{lem:smallscale}
For any $D>0$, there is a $N_0$ such that, for any $N\geq N_0$ and $E\in[A+N^{-2/3},B-N^{-2/3}]$
$$
\mathbb{P}\left(\absa{\sum_{i=1}^N g(\lambda_i)-N\int g\rd\mu_V} > (\log N)^{9/10}\right)\leq N^{-D}. 
$$
\end{lemma}

\begin{proof}
We define the event $\mathscr{A}$ as in \eqref{eqn:set} with $r = 3/5$.
Then it follows from \eqref{eqn:AP} that, for some constant $\theta > 0$ and for $N$ large enough,
\[
\mathbb{P}\left(\mathscr{A}\right)
\geq 1-\theta^{-1}N^{6}e^{-\theta(\log N)^{6/5}/400}
\geq 1-N^{-D}.
\]
On the other hand, from Lemma \ref{lem:HS} with $\gamma=\eta_0$, $\eta=\eta_1$, we have
$\mathscr{A} \subset \{|\sum g(\lambda_i)-\int g\rd\mu_V|<(\log N)^{9/10}\}$ for $N$ large enough, so the result follows.
\end{proof}

\begin{proof}[Proof of Corollary \ref{cor:rigidity}]
Recall that $f$, $f_0$, $g$, whose definitions precede Lemma \ref{lem:smallscale}, depend on $E$. Fix some $K>0$.
As $a_N/\log N\to\infty$, for any $k\in\llbracket a_N,N-a_N\rrbracket$, we have $E:=\gamma_k+\frac{K\log N}{N^{2/3}{\hat k}^{1/3}}\in[A+N^{-2/3},B-N^{-2/3}]$ for $N$ large enough.
Moreover, there are constants $C,c>0$ (independent of $K$) such that
\[
\P \left(\lambda_k-\gamma_k>\frac{K\log N}{N^{2/3}{\hat k}^{1/3}}\right)
\leq \P \left(\sum f_0(\lambda_i)-N\int f_0\rd\mu_V>\frac{K\log N}{C}\right)
+ \P \left(\sum g(\lambda_i)-N\int g\rd\mu_V>\frac{K\log N}{C}\right)
+e^{-cN},
\]
where the exponentially small term accounts for  $f\neq 1$ beyond $B+\tilde\eta/4$, and relies on (\ref{eq:large_deviation_rigidity}).
The second probability on the right-hand side is bounded thanks to Lemma \ref{lem:smallscale}.
Applying a Chernoff bound, by Lemma \ref{lem:loop}, there exists $t_1$ such that
$$
\mathbb{P}\left(\absa{\sum f_0(\lambda_i) -N\int f_0 \rd\mu_V} > \frac{K\log N}{C}\right)
\leq e^{-t_1 K (\log N)/C}
\E \left[ e^{t_1 \absa{\sum f_0(\lambda_i)-N\int f_0\rd\mu_V}} \right]
\leq e^{(A-K/C)t\log N}
$$
\sloppy
for some absolute constant $A$. Choosing $K$ large enough for a fixed $D$ therefore concludes the proof that $\P(\lambda_k-\gamma_k>\frac{C\log N}{N^{2/3}{\hat k}^{1/3}})\leq N^{-D}$. The probability of the event $\lambda_k-\gamma_k<-\frac{C\log N}{N^{2/3}{\hat k}^{1/3}}$ is similarly bounded.
\end{proof}
%

\subsection{Tightness. }\label{subsec:tightness}
We now give the brief proof of Corollary \ref{cor:tightness}. 
For any interval $J=[E_N-a\ell(E_N),E_N+a\ell(E_N)]$, $a>1$, the desired inequality $\lim_{t\to\infty}\limsup_{N\to\infty}\mathbb{P}(\mathcal{N}(J)>t)=0$
follows directly from Lemma \ref{eq:all_moments_N(I)} applied to intervals of type $I=[E'-\ell(E'),E'+\ell(E')]$ covering $J$.

\subsection{Smoothed log-correlated field. }\label{subsec:smoothed}
The goal of this section is a key step for the proof of Theorem \ref{thm:log_corr_field}: regularization of the $\log$ by replacing the point $E$ in the linear statistics $L_N(E)$ by the point $E+\ii\eta(E)$, with
\begin{equation} \label{eq:def_eta_0}
	\kappa(E) = \absa{E-A} \wedge \absa{E-B}
	\qquad \text{ and } \qquad 
	\eta(E) \coloneqq \frac{\exp\pa{(\log N)^{1/4}}}{N(\sqrt{\kappa(E)} \vee N^{-1/3})}
	= \exp\pa{(\log N)^{1/4}} \cdot \ell(E),
\end{equation}
for $E \in [A,B]$, (see \eqref{def:l(E)_and_kappa(E)} and \eqref{def:L_N} for the definitions of $\ell(E)$ and $L(E)$). 
This means we can regularize the logarithm at the scale slightly larger than the microscopic scale.
Note that we approach $E$ by a point above the real axis which is consistent with our choice to extend the logarithm to the negative real axis by continuity from above.
\begin{proposition} \label{prop:smoothing}
	Let $z = E + \ii \eta(E)$ with $E \in [A,B]$.
	Then
	$(L_N(z) - L_N(E))/\sqrt{\log N}$ converges to 0 in probability,
	uniformly in $E \in [A,B]$. 
\end{proposition}

\begin{proof}
Let $E \in [A,B]$.
For brevity, we write $\eta = \eta(E)$, so that $z = E + \ii \eta$.
Writing $\log(z - \lambda) - \log(E - \lambda) = \int_0^\eta \frac{\ii \diff u}{E+\ii u - \lambda}$, we have
\[
	L_N(z) - L_N(E)
	= \ii N \int_0^{\eta} (m_V (E+\ii u) - s_N(E+\ii u)) \diff u.
\]
Recall we want to prove this quantity is $o(\sqrt{\log N})$ in probability.
Let
\[
	\eta' \coloneqq \frac{(\log N)^{-1/4}}{ N (\sqrt{\kappa(E)} \vee N^{-1/3})}=(\log N)^{-1/4} \cdot \ell(E).
\]
Then by Theorem \ref{th:local_law_bulk}, we have
\begin{align}
    \E \left[ 
        \absa{ N \int_{\eta'}^{\eta} (m_V (E+\ii u) - s_N(E+\ii u)) \diff u }
    \right]
    \leq \int_{\eta'}^{\eta} \frac{C}{u} \diff u
    \leq C(\log N)^{1/4}.
    \label{eq:from_eta'_to_eta}
\end{align}
Now let $z' = E + \ii \eta'$. By the triangle inequality,
\begin{align}
    & \absa{ N \int_0^{\eta'} (m_V (E+\ii u) - s_N(E+\ii u)) \diff u }
    \nonumber \\
    & \leq N \eta' \absa{m_V (z') - s_N(z')}
    + N \int_0^{\eta'} \absa{m_V (E+\ii u) - m_V(z')} \diff u
    + N \int_0^{\eta'} \absa{s_N (E+\ii u) - s_N(z')} \diff u.
    \label{eq:decompo}
\end{align}
The first term on the right-hand side of \eqref{eq:decompo} is bounded in $L^1$ by Theorem \ref{th:local_law_bulk}. 
Using that $\abs{m_V'(w)} \leq C/\abs{b(w)}$ for $w$ in a compact subset of $\CC$, the second term is bounded by $C N (\eta')^{3/2} \leq C$.
For the third term, we introduce the event $A \coloneqq \{ \cN([E-\eta',E+\eta']) =0 \}$, on which there are no particles at distance less than $\eta'$ from $E$.
By Proposition \ref{prop:density}, we have $\P(A) \to 1$ uniformly in $E$,
and on the event $A$, we have
\begin{align*}
    \absa{s_N (E+\ii u) - s_N(z')} 
    = \absa{\frac{1}{N} \sum_{k=1}^N \frac{E+\ii u-z'}{(\lambda_k-E-\ii u)(\lambda_k-z')}} 
    \leq \frac{1}{N} \sum_{k=1}^N \frac{\eta' \sqrt{2}}{\abs{\lambda_k-z'}^2}
    = \sqrt{2} \im s_N(z').
\end{align*}
Moreover, with our specific value of $\eta'$, it follows from \eqref{eq:bound_Lambda} that 
$\E[\im s_N(z')] \leq C( (N\eta')^{-1} + N^{-1/2} + \abs{b(z')} ) \leq C (N\eta')^{-1}$.
Therefore, we have 
\begin{align*}
    \E \left[ \1_A \cdot
    N \int_0^{\eta'} \absa{s_N (E+\ii u) - s_N(z')} \diff u 
    \right]
    \leq \sqrt{2} N \eta' \cdot \E \left[\im s_N(z')\right]
    \leq C.
\end{align*}
Thus, the right-hand side of \eqref{eq:decompo} is $\OO(1)$ in probability.
Combined with \eqref{eq:from_eta'_to_eta}, this concludes the proof.
\end{proof}

\begin{proof}[Proof of Theorem \ref{thm:log_corr_field}]
The proof of Lemma \ref{cor:momentsBound} in the next section requires a function supported on the domain where our strong local law holds.
Therefore, let $\phi$ be a fixed, smooth cutoff function on scale $1$,
\begin{equation}\label{eqn:phi}
\phi(x)=1\ {\rm on}\  [A-\tilde{\eta}/4,B+\tilde{\eta}/4],\ 0\ {\rm on}\ [A-\tilde{\eta}/2,B+\tilde{\eta}/2]^{\rm c},
\end{equation}
and define 	
\begin{equation}
	\label{def:L_Ntilde}
	\widetilde L_N(z) 
	\coloneqq \sum_{j=1}^N \log\pa{z-\lambda_j}\phi(\lambda_i)
	- N \int \log\pa{z-x} \diff \mu_V(x).
\end{equation}	
Recall the definition of $\ba$, $\bb$ and $\delta_\ell$ in Theorem \ref{thm:log_corr_field}.
Then Proposition \ref{prop:CLT_above_the_axis}, which will be proved in the next section, shows that (distinguishing cases to simplify~$\Gamma$)
\begin{align} \label{eq:rewriting_CLT_above_the_axis}
	\sqrt{\frac{\beta}{\log N}}	
	\pa{\re \widetilde L_N(z_1) - \delta_1, \dots, \re \widetilde  L_N(z_k) - \delta_k,
	\im  \widetilde L_N(z_1), \dots, \im \widetilde L_N(z_k)}
	\xrightarrow[N\to\infty]{\text{(d)}}
	\cN \left( 0, \left( \begin{matrix}
	\ba & 0 \\ 0 & \bb
	\end{matrix} \right)\right).
\end{align}
Moreover, it follows from \eqref{eq:large_deviation_rigidity} that $\widetilde{L}_N(z)- L_N(z)$ converges to $0$ in probability
and Proposition \ref{prop:smoothing} states that $(L_N(z) - L_N(E))/\sqrt{\log N}$ converges to 0 in probability.
Hence, by Slutsky's theorem, we can replace $\widetilde{L}_N(z_\ell)$ by $L_N(E_\ell)$ in \eqref{eq:rewriting_CLT_above_the_axis} which concludes the proof.
\end{proof}

\begin{proof}[Proof of Corollary \ref{corr:eig_fluct}]
Using that $\lambda_k \leq E$ if and only if $\abs{\left\{ \lambda_j \leq E \right\}} \geq k$, we have, for any $\xi \in \R$,
\[
    Y_N(n) \leq \xi 
    \quad \Longleftrightarrow \quad
    \im L_N \left( 
        \gamma_n + \frac{\xi}{\pi N \varrho_V(\gamma_n)} \sqrt{ \frac{\log N}{\beta}}
        \right)
    \leq N \pi 
    \int_{\gamma_n}^{\gamma_n + \frac{\xi}{\pi N\varrho_V(\gamma_n)} \sqrt{ \frac{\log N}{\beta}}} \varrho_{V}(x) \diff x.
\]
Therefore, the result follows from Theorem \ref{thm:log_corr_field}.
See \cite[Lemma B.2]{BouMod2019} for more details.
\end{proof}

\section{Central limit theorem above the axis}

\label{section:CLT_above_the_axis}
In this section, we prove the following central limit theorem  at level 
$\eta(E)$, see \eqref{eq:def_eta_0}. 
Recall the definition of $\widetilde L_N(z)$ in \eqref{def:L_Ntilde}.

\begin{proposition} \label{prop:CLT_above_the_axis} 
For any $m \geq 1$, uniformly in  
$\xi_1,\dots,\xi_m,\zeta_1,\dots,\zeta_m \in [-(\log N)^{1/4},(\log N)^{1/4}]$ and $E_1, \dots,E_m \in [A,B]$, the following holds.
Setting $z_\ell = E_\ell + \ii \eta(E_\ell)$,  we have 
\begin{align*}
	\E \left[ \exp \left( \sqrt{\frac{\beta}{\log N}}  
		\sum_{\ell=1}^m
        	\left( \xi_\ell \re \widetilde L_N(z_\ell)
        		+ \zeta_\ell \im \widetilde L_N(z_\ell) \right)
		\right) \right]
	& = \exp \left(
		\Gamma
    	+ \OO \left((\log N)^{-1/4}\right) \right),
\end{align*}
where the error terms depend only on $\beta$, $V$, $m$
and
\begin{align*}
	\Gamma 
	&=-\frac{1}{2\log N} \sum_{\ell,j =1}^m \Bigg(
	\xi_\ell \xi_j 
		\log\pa{ \absa{E_\ell - E_j} \vee \pa{\eta_\ell + \eta_j}    } + 
	\zeta_\ell \zeta_j 
		 \log\pa{
			\frac{ \absa{E_j - E_\ell} \vee \pa{\eta_\ell + \eta_j}      }{ \pa{ \absa{E_\ell - A} \vee \eta_\ell}  \pa{\absa{B-E_j} \vee \eta_j  }  } \wedge 1
		}
	\Bigg) \\
	& \hphantom{{}={}} +  
	 \sqrt{\frac{\beta}{\log N}} \frac{1}{4} \pa{\frac{2}{\beta}-1}
	\sum_{\ell=1}^m \xi_\ell \log (\kappa_\ell \vee \eta_\ell),
\end{align*}	
with $\eta_\ell = \eta(E_\ell)$ and $\kappa_\ell = \kappa(E_\ell)$.
\end{proposition}

\subsection{Strategy of the proof. }
We first explain the classical strategy, which relies on a loop equation argument that dates back to Johansson \cite{Johansson1998} to obtain the Laplace transform of linear statistics.
The essential ingredient to adapt Johansson's method to the scale $\eta(E)$ is the rigidity estimate under biased measures. This rigidity is proved in Subsection \ref{subsection:rigidity_biased} and relies on Theorem 
\ref{th:local_law_bulk}.

\begin{remark}
Johansson's method can be adapted to asymptotics of the Fourier transform of linear statistics instead of the Laplace transform, as in  Section 5 in \cite{BouErdYauYin2016} and Section 6 in \cite{LanSosYau2019}.
Natural benefits are a quantitative convergence, and no need of rigidity estimates under biased measures. However in this paper we choose to prove asymptotics of the Laplace transform, as this is needed for the tail estimate, Lemma \ref{lem:loop}.
\end{remark}

Let $m \geq 1$ and $E_1,\dots,E_m \in [A,B]$, and let $a_1,\dots,a_m,b_1,\dots,b_m \in [-1,1]$.
These $a_i$'s and $b_i$'s will play the role of rescaled versions of $\xi_1,\dots,\xi_m,\zeta_1,\dots,\zeta_m$ in the statement of Proposition \ref{prop:CLT_above_the_axis}.
Let us emphasize here that, throughout Section \ref{section:CLT_above_the_axis}, constants $C$ and $\OO(\dots)$ only depend on $\beta$, $V$ and $m$.

For any $1 \leq \ell \leq m$, we set $\eta_\ell = \eta(E_\ell)$, $z_\ell \coloneqq E_\ell + \ii \eta_\ell$ and $\kappa_\ell \coloneqq \kappa(E_\ell)$ (see \eqref{eq:def_eta_0} for the definition of $\eta(E)$ and $\kappa(E)$). 
Moreover, for any $s \in \R$, let
$f_\ell(s) \coloneqq \phi(s)\re \left( (a_\ell- \ii b_\ell) \log (z_\ell-s) \right)$ and
\begin{equation} \label{eq:def_function_f}
	f(s) 
	\coloneqq \phi(s)\sum_{\ell=1}^m
	\left( a_\ell \re \log (z_\ell-s)
		+ b_\ell \im \log (z_\ell-s) \right)
	= \sum_{\ell = 1}^m f_\ell(s),
\end{equation}
where $\phi$ is defined in (\ref{eqn:phi}).
We are interested in the following centered linear statistics
\[
	S_N(f) \coloneqq \sum_{j=1}^N f(\lambda_j) - N \int f\rd \mu_V.
\]
Recall $\mu_N$ is the probability distribution of the particles defined in \eqref{eq:beta_ensembles_density}.
We define a measure $\mu_N^t$, for any $t \in \R$, by
\begin{equation} \label{eq:def_Z(theta)}
	\diff \mu_N^t (\lambda_1,\dots,\lambda_N) 
	\coloneqq \frac{e^{t S_N(f)}}{Z(t)} 
	\diff \mu_N (\lambda_1,\dots,\lambda_N),
	\qquad 
	\text{with } Z(t) \coloneqq
	\E \left[ e^{t S_N(f)} \right],
\end{equation}
where $\E$ still denotes the expectation under $\mu_N$. The expectation under $\mu_N^t$ is denoted by $\E_{\mu_N^t}$.
Let $\varrho_1^{(N,t)}(s)$ be the 1-point function for the eigenvalues under $\mu_N^t$, which satisfies
\[
	\int_\R h(s) \varrho_1^{(N,t)}(s) \diff s 
	= \E_{\mu_N^t} \left[ \frac{1}{N} \sum_{k=1}^N h(\lambda_k) \right],
\]
for any continuous bounded function $h$.
Let $m_{N,t}(z) \coloneqq \E_{\mu_N^t} \left[ s_N(z) \right]$ be the Stieljes transform of $\varrho_1^{(N,t)}(s)$.
Recall our goal is to estimate the Laplace transform of $S_N(f)$, i.e. $Z(t)$.
For this, we will estimate
\[
	Z'(t) = \E \left[S_N(f) e^{t S_N(f)} \right] 
	= Z(t) \cdot \E_{\mu_N^t} \left[S_N(f) \right].
\]
Using the Helffer--Sj\"{o}strand formula \cite{HelfferSjostrand1989}, we can express $\E_{\mu_N^t} [S_N(f)]$ in terms of $m_{N,t}(z) - m_{V}(z)$, see \eqref{eqn:HSf}.
Hence, we first prove precise estimates for $m_{N,t}(z) - m_{V}(z)$ using the first  loop equation, in Lemma \ref{lem:estimate_varphi}.
Then, we use these estimates combined with Helffer--Sj\"{o}strand formula to compute asymptotics of $Z(t)$ in Lemma  \ref{lem:estimate_Z(theta)}.
This gives the proof of Proposition \ref{prop:CLT_above_the_axis}, up to a rewriting of the limiting variances and shifts 
which we deal with in Section \ref{subsection:covariances}.

The key input for the above proof sketch is the rigidity of the particles under biased measures (Lemma \ref{lem:rigid2}, based on Theorem \ref{th:local_law_bulk}), which we now prove.

\subsection{Rigidity under biased measures. }
\label{subsection:rigidity_biased}
The main result of this section is Lemma \ref{lem:rigid2}. We start with the following key estimate about $f$, defined in (\ref{eq:def_function_f}).

\begin{lemma}\label{cor:momentsBound}
For any $m \geq 1$, there exists a (small) $c_0>0$  and (large) $N_0$  such that for any $E_1, \dots,E_m \in [A,B]$, $a_1,\dots,a_m,b_1,\dots,b_m \in [-1,1]$, $N \geq N_0$ and $|\zeta|<c_0$
we have
\[
\absa{\log \E \left[ e^{\zeta \left( \sum_{k=1}^N f(\lambda_k)-N\int f\rd\mu_V \right)} \right]} \leq (\log N)^5.
\]
\end{lemma}

\begin{proof}
First note that, for any real random variable $X$, we have 
$\abs{\log \E [e^X]} \leq \log \E[ e^{\abs{X}}]$.
Therefore, expanding the exponential, we get
\begin{equation} \label{eq:expanding}
\absa{ \log \E \left[ e^{\zeta \left( \sum_{k=1}^N f(\lambda_k)-N\int f\rd\mu_V \right)} \right] }
\leq \log \sum_{k\geq 1} \frac{\abs{\zeta}^k}{k!} 
\E\left[ \absa{\sum_{k=1}^N f(\lambda_k)-N\int f\rd\mu_V}^k \right].
\end{equation}
Note that, up to changing the above constant $\abs{\zeta}^k$ into $\abs{m\zeta}^k$, it is enough to bound the above right-hand side in the case $m=1$. 
We therefore consider $f=f_1$, with $E=E_1\in[A,B]$.
We now apply Helffer-Sj{\H o}strand formula in a way similar to \eqref{eqn:HSExpand}, but with an $\eta$ which depends on $x$, $\eta(x) = \ell(E) \vee (x-B) \vee (A-x)$.
We choose $\gamma=(\log N)^{-2}$ and consider a cutoff function $\chi$ such that  $\chi(x)=1$ on $[0,\gamma]$, $\chi(x)=0$ on $[2\gamma,\infty)$ and $\|\chi'\|_\infty<100/\gamma$.
Since $\eta$ depends on $x$, another boundary term appears in the integration by parts in \cite[(B.17)]{ErdRamSchYau2010} and we get
\begin{align*}
\absa{ \sum_{k=1}^N f(\lambda_k)-N\int f\rd\mu_V }^k
&\leq C^k ((\mathrm{I})^k+(\mathrm{II})^k+(\mathrm{III})^k+(\mathrm{IV})^k+(\mathrm{V})^k), \\
(\rm I)&=N \absa{\im \iint_{y>0} (f(x)+\ii yf'(x)) \chi'(y) (s-m_V)(x+\ii y) \rd x \rd y}, \\
(\rm II)&=N\iint_{0<y<\eta(x)} y|f''(x)||\im (s-m_V)(x+\ii y)|\rd x\rd y, \\
(\rm III)&=N\iint_{\eta(x)<y} \abs{\partial_y(y\chi(y))}\abs{f'(x)}\abs{(s-m_V)(x+\ii y)} \rd x \rd y, \\
(\rm IV)&=N \int \eta(x) \abs{f'(x)} \abs{(s-m_V)(x+\ii \eta(x))} \rd x \\
(\rm V)&=N \int_{y>\ell(E)} \abs{y\chi(y)} 
\absa{f'(B+y) (s-m_V)(B+y+\ii y) - f'(A-y) (s-m_V)(A-y+\ii y)} \rd y.
\end{align*}
We now bound the $k$-th moment of each of these quantities successively.
For $(\mathrm{I})$, we first rewrite the part involving $f'(x)$ using integration by parts w.r.t.\@ $x$ and then $y$, similar to \cite[(B.17)]{ErdRamSchYau2010},
\begin{align*}
\iint_{y>0} \ii yf'(x) \chi'(y) (s-m_V)(x+\ii y) \rd x \rd y 
& = - \iint_{y>0} \ii yf(x) \chi'(y) (s-m_V)'(x+\ii y) \rd x \rd y \\
& = \iint_{y>0} \partial_y (y\chi'(y)) f(x)  (s-m_V)(x+\ii y) \rd x \rd y.
\end{align*}
Therefore, since $\chi'(y)$ is non-zero only if $y \in [\gamma,2\gamma]$, we get
\[
	(\mathrm{I}) 
	\leq C \iint_{\R \times [\gamma,2\gamma]} 
	\abs{f(x) \cdot N \im(s-m_V)(x+\ii y)} \rd x \rd y
\]
By H{\"o}lder's inequality, we have
\begin{align*}
\E \left[ (\mathrm{I})^k \right]
\leq C^k \int_{(\R \times [\gamma,2\gamma])^k}
\E \bigl[ \absa{ f \cdot N \im(s-m_V) \cdots f\cdot N \im(s-m_V) } \bigr] 
\leq C^k \left(
\int_{\R \times [\gamma,2\gamma]}
\E \bigl[ \absa{ f \cdot N \im(s-m_V)}^k \bigr]^{1/k}
\right)^k.
\end{align*}
Applying Theorem \ref{th:local_law_bulk} in the trapezoid region and Proposition \ref{prop:bound_s-m_1_s-m_2_past_the_edge} outside, note that we have 
$\E[\absa{\im(s-m_V)}^k] \leq (Ck)^k/(Ny)^k$ in both cases.
It follows that $\E[(\mathrm{I})^k] \leq (Ck)^k$.
We now deal with $(\mathrm{II})$.
Again by H{\"o}lder's inequality, we have
\begin{align*}
\E \left[ (\mathrm{II})^k \right]^{1/k}
\leq \int_{\R \times \R_+} \1_{y < \eta(x)} y \abs{f''(x)}
\E \bigl[ \absa{N \im (s-m_V)(x+\ii y)}^k \bigr]^{1/k} \rd x \rd y.
\end{align*}
We cut this integral into two parts, depending on if $(x,y)$ is on the trapezoid region (that is $A-y \leq x \leq B+y$) or not.
For the part corresponding to the trapezoid (in which case $\eta(x) = \ell(E)$), Theorem \ref{th:local_law_bulk} 
and $\abs{f''(x)} \leq C/\abs{x-E+\ii\eta_1}^2$ give the bound
\begin{align*}
\int_\R \int_0^{\ell(E)} y \abs{f''(x)} 
\left( \frac{(Ck)^{1/2}}{y} + \frac{Ck}{\sqrt{\kappa(x+\ii y)}}\right) \rd y \rd x
\leq Ck \ell(E) \int_\R \abs{f''(x)} \rd x
\leq Ck \frac{\ell(E)}{\eta_1}.
\end{align*}
For the other part, we focus on the case $x>B$, the case $x<A$ being similar.
Combining the bounds of Proposition~\ref{prop:bound_s-m_1_s-m_2_past_the_edge}, in this region, we have
\begin{align*}
\E \bigl[ \absa{N \im (s-m_V)(x+\ii y)}^k \bigr]^{1/k} 
\leq \frac{(Ck)^{1/2}}{y}
+ \frac{Ck}{\sqrt{x-B}} 
+ \sqrt{N} (Ck)^{1/2} \1_{y \leq (C'k)^{1/2}/N\sqrt{x-B}}.
\end{align*}
Therefore, we bound the remaining part using
\begin{align*}
& \int_B^\infty \int_0^{x-B} y \abs{f''(x)} 
\E \bigl[ \absa{N \im (s-m_V)(x+\ii y)}^k \bigr]^{1/k} \rd y \rd x \\
& \leq \int_B^\infty \abs{f''(x)} 
\left( (Ck)^{1/2} (x-B) + Ck (x-B)^{3/2} \right) \rd x
+ \int_B^\infty (x-B) \frac{(C'k)^{1/2}}{N\sqrt{x-B}} \abs{f''(x)} \sqrt{N} (Ck)^{1/2} \rd x \\
& \leq (Ck)^{1/2} \log N + Ck + \frac{Ck}{\sqrt{N\eta_1}}.
\end{align*}
Recalling $\eta_1 = \exp((\log N)^{1/4}) \ell(E) \geq \ell(E) \geq c N^{-1}$, it follows that
\[
\E \left[ (\mathrm{II})^k \right]
\leq (Ck)^{k/2} (\log N)^k + (Ck)^k.
\]
The main contribution comes from the third term. Note that thanks to our choice of $\eta(x)$, the domain of integration of $(\mathrm{III})$ is included in the trapezoid region, so we can apply Theorem \ref{th:local_law_bulk} to get
\begin{align*}
\E \left[ (\mathrm{III})^k \right]^{1/k}
&\leq \int_{\R \times \R_+} \1_{\eta(x)<y} \abs{\partial_y(y\chi(y))} \abs{f'(x)}
\E \bigl[ \absa{N (s-m_V)(x+\ii y)}^k \bigr]^{1/k} \rd x \rd y \\
& \leq \int_{\R} \int_{\ell(E)}^{2\gamma} \abs{f'(x)}
\left( \frac{(Ck)^{1/2}}{y} + \frac{Ck}{\sqrt{\kappa(x+\ii y)}} \right) \rd y \rd x
\leq \log N \cdot \left( (Ck)^{1/2} \log N + Ck \gamma^{1/2} \right),
\end{align*}
integrating first w.r.t.\@ $y$ and then w.r.t.\@ $x$.
Hence, since $\gamma = (\log N)^{-2}$,
\begin{align*}
\E \left[ (\mathrm{III})^k \right]
&\leq (Ck)^{k/2} (\log N)^{2k} + (Ck)^k.
\end{align*}
Similarly, with H{\"o}lder's inequality and Theorem \ref{th:local_law_bulk},
\begin{align*}
\E \left[ (\mathrm{IV})^k \right]^{1/k}
\leq \int_{\R} \eta(x) \abs{f'(x)}
\left(\frac{(Ck)^{1/2}}{\eta(x)} + \frac{Ck}{\sqrt{\kappa(x)}} \right) \rd x
\leq (Ck)^{1/2} \log N + Ck,
\end{align*}
where for the second term we simply note that $\eta(x) \abs{f'(x)} \leq \eta(x) C/\abs{x-E+\ii\eta_1} \leq C$.
The term $(\mathrm{V})$ is smaller than~$(\mathrm{IV})$.
Coming back to \eqref{eq:expanding}, for some constant $C_0 > 0$ depending only on $m$, $V$ and $\beta$, we proved that
\[
\log \E \left[ e^{\zeta \left( \sum_{k=1}^N f(\lambda_k)-N\int f\rd\mu_V \right)} \right]
\leq \log \sum_{k\geq 1} \frac{\abs{\zeta}^k}{k!} C_0^k \left( k^{k/2} (\log N)^{2k} + k^k \right).
\]
Choosing $\abs{\zeta} \leq (100 C_0)^{-1}$, this last series is smaller than $\sum_{k\leq (\log N)^4}(\log N)^{2k}+\sum_{k\geq 1}\frac{2k^k}{100 k!}\leq e^{(\log N)^5}$, for $N$ large enough, which concludes the proof.
\end{proof}

\begin{lemma}\label{lem:rigid2}
For any $m \geq 1$, letting $t_0 = c_0/2$ with $c_0$ from Lemma \ref{cor:momentsBound}, there exists $N_0$ such that for any $E_1, \dots,E_m \in [A,B]$, $a_1,\dots,a_m,b_1,\dots,b_m \in [-1,1]$, $N \geq N_0$ and $\abs{t}<t_0$, we have 
\[
\mathbb{P}_{\mu_N^t}\left( \bigcap_{1\leq k\leq N}\{|\lambda_k-\gamma_k|<(\log N)^{100}N^{-\frac{2}{3}}(\hat k)^{-\frac{1}{3}}\}\right)\geq 
1-e^{-(\log N)^2}.
\]
\end{lemma}

\begin{proof}
This is an easy consequence of Lemma \ref{lem:rigid} with $r=30$ and Lemma  \ref{cor:momentsBound}. Indeed, with $A= \bigcap_{1\leq k\leq N}\{|\lambda_k-\gamma_k|<(\log N)^{100}N^{-\frac{2}{3}}(\hat k)^{-\frac{1}{3}}\}$, 
by the Cauchy-Schwarz inequality and then Lemma \ref{cor:momentsBound},
we have
\[
\mathbb{P}_{\mu_N^t} \left(A^c\right)
\leq \frac{\E \Bigl[ e^{2t(\sum f(\lambda_i)-N\int f\rd\mu_V)} \Bigr]^{1/2} \mathbb{P}(A^c)^{1/2}}
	{\E \Bigl[ e^{t(\sum f(\lambda_i)-N\int f\rd\mu_V)} \Bigr]}
\leq e^{\frac{3}{2}(\log N)^5-\frac{1}{2}(\log N)^{3\cdot 30/4}},
\]
which concludes the proof.
\end{proof}

\subsection{Analysis of the first loop equation. }
\label{subsection:first_loop}
Recall $m_{N,t}(z)$ is the Stieljes transform of $\varrho_1^{(N,t)}(s)$.
We introduce 
\begin{equation*} 
	\varphi(z) = \varphi_{N,t}(z) 
	\coloneqq m_{N,t}(z) - m_{V}(z)
\end{equation*}
and for $z \in \Omega \setminus \R$, recalling $\Omega$ is an open set in $\C$ containing $\R$ and such that $V$ is analytic on $\Omega$,
\begin{align} \label{eq:def_psi}
	\psi(z)  
	& \coloneqq \frac{2t}{\beta N} \int_A^B \frac{f'(s)}{s-z} \varrho_{V}(s) \diff s - \frac{1}{N}\pa{\frac{2}{\beta}-1} m'_{V}(z)
	- \int_\R \frac{V'(s) - V'(z)}{s-z} \pa{\varrho_1^{(N,t)}(s) - \varrho_V(s)} \diff s \\
	\label{eq:def_Err}
	\mathrm{Err}(z) 
	& \coloneqq \varphi(z)^2
	- \frac{2t}{\beta N} 
		\int_\R \frac{f'(s)}{s-z}\pa{ \varrho_1^{(N,t)}(s) - \varrho_V(s)} \diff s 
	+ \frac{1}{N} \left( \frac{2}{\beta} - 1 \right) \varphi'(z)
	+ \Var_{\mu_N^t} \left( s_N(z) \right).
\end{align}
Then, following for example \cite[(2.8)]{Shc2011} or replacing $V$ by $V_t = V - \frac{2t}{\beta N} f$ in \eqref{eq:loop_equation_n=1}, 
we have the loop equation, for any $z \in \Omega \setminus \R$,
\begin{equation} \label{eq:loop_equation_mu_theta}
	(2 m_V(z) + V'(z)) \varphi(z) - \psi(z) + \mathrm{Err}(z) = 0,
\end{equation}
where $\mathrm{Err}(z)$ gathers the negligible terms.
We use the loop equation to show $\varphi(z)$ is close to $\widetilde{\varphi}(z)/N$, where we set
\begin{align} \label{eq:def_phi_tilde}
	\widetilde{\varphi}(z)  
	\coloneqq \frac{1}{2\pi b(z)} \left(
	\frac{2t}{\beta} \int_A^B \frac{f'(s)}{s-z} \tau(s) \diff s
	- \pa{\frac{2}{\beta}-1} 
	\left( \pi (b'(z)-1) + \int_A^B \frac{r'(s) \tau(s)}{r(s)(s-z)} \diff s \right)
	\right),
\end{align}
with $\tau(s) \coloneqq \sqrt{(s-A)(B-s)}$ and recalling $b(z) = \sqrt{z-A}\sqrt{z-B}$.
\begin{lemma} \label{lem:estimate_varphi}
Consider $\abs{t}<t_0$, with $t_0$ from Lemma \ref{lem:rigid2}. 
Let $\tilde{\eta} > 0$ be as in Lemma \ref{lem:preliminaries_fixed_point_equation}.
Then, for any $z = E + \ii \eta$ with $N^{-1} < \abs{\eta} \leq \tilde{\eta}/2$ and $A-\tilde{\eta} \leq E \leq B+\tilde{\eta}$, we have
\begin{align*}
	\varphi(z) 
	& = \frac{\widetilde{\varphi}(z)}{N}
	+ \OO \left( \frac{(\log N)^{201}}{(N\eta)^2 \abs{b(z)}} \right)
	+ \OO \left( \frac{(\log N)^{201}}{N^2 \abs{b(z)}} 
		\sum_{\ell=1}^m \frac{1}{\eta_\ell(\eta_{\ell}\vee|z-z_\ell|)} \right),
\end{align*}
where the error terms depend only on $\beta$ and $V$.
\end{lemma}

The proof of this lemma relies on two steps. First we use the rigidity established in Lemma \ref{lem:rigid2} to show $\mathrm{Err}(z)$ is indeed an error term. 
Then, the loop equation implies that $(2 m_V(z) + V'(z)) \varphi(z) \simeq \psi(z)$, but the third term in $\psi(z)$ is still an unknown second order term. 
In order to get rid of it, we use a contour integral argument similar to the one used in Section \ref{subsection:dealing_with_Delta} or more precisely, to the one used by Shcherbina \cite[(2.12) to (2.17)]{Shc2011}.
For this reason, we need to work with confined particles and, for convenience, we actually restrict ourselves to the rigidity event $\mathscr{R} \coloneqq \bigcap_{1\leq k\leq N}\{|\lambda_k-\gamma_k|<(\log N)^{100}N^{-\frac{2}{3}}(\hat k)^{-\frac{1}{3}}\}$, by introducing the new measure
\begin{align*}
	\diff \mu_N^{t,\mathscr{R}}(\lambda_1,\dots,\lambda_N) 
	= \frac{\1_{\mathscr{R}}}{\P_{\mu_N^t}(\mathscr{R})} 
	\diff \mu_N^t(\lambda_1,\dots,\lambda_N).
\end{align*}
Note that $\P_{\mu_N^t}(\mathscr{R}) \geq 1 - e^{-(\log N)^2}$ by Lemma \ref{lem:rigid2}.
Moreover, let  $\varrho_1^{(N,t,\mathscr{R})}(s)$ be the 1-point function under $\mu_N^{t,\mathscr{R}}$,  
\begin{align*}
\varphi^\mathscr{R}(z) \coloneqq \E_{\mu_N^{t,\mathscr{R}}} [s_N(z)] - m_V(z)
\end{align*}
and $\mathrm{Err}^\mathscr{R}(z)$ be defined as $\mathrm{Err}(z)$ but with $\mu_N^{t,\mathscr{R}}$, $\varrho_1^{(N,t,\mathscr{R})}(s)$ and $\varphi^\mathscr{R}(z)$ instead of $\mu_N^t$, $\varrho_1^{(N,t)}(s)$ and $\varphi(z)$.
We tackle the first step of the argument in the following lemma, which bounds the terms appearing in $\mathrm{Err}^\mathscr{R}(z)$.
\begin{lemma} \label{lem:first_estimates_1}
Consider $\abs{t}<t_0$, with $t_0$ from Lemma \ref{lem:rigid2}. 
Let $\tilde{\eta} > 0$ be as in Lemma \ref{lem:preliminaries_fixed_point_equation}.
Then, for any $z = E + \ii \eta$ with $0< \abs{\eta} \leq \tilde{\eta}$ and $A-\tilde{\eta} \leq E \leq B+\tilde{\eta}$, we have
\begin{align} \label{eq:bound_with_rigidity}
	& \varphi^\mathscr{R}(z) = \OO \left( \frac{(\log N)^{100}}{ N\eta} \right),
	\qquad 
	\mathrm{Err}^\mathscr{R}(z)
	= \OO \left( \frac{(\log N)^{200}}{(N\eta)^2} \right)
	+ \OO \left( \frac{(\log N)^{200}}{N^2} 
		\sum_{\ell=1}^m \frac{1}{\eta_\ell(\eta_{\ell}\vee|z-z_\ell|)} \right), \\
	\label{eq:from_rigidity_to_without_rigidity}
	& \varphi(z)
	= \varphi^\mathscr{R}(z) + \OO \left( \eta^{-1} e^{-(\log N)^2} \right),
	\qquad 
	\mathrm{Err}(z)
	= \mathrm{Err}^\mathscr{R}(z) + \OO \left( \eta^{-2} e^{-(\log N)^2}\right),
\end{align}
where the error terms depend only on $\beta$ and $V$. If, moreover, $E \leq A-\tilde{\eta}/2$ or $E \geq B+\tilde{\eta}/2$,
\begin{align} \label{eq:bound_with_rigidity_past_the_edge}
	\mathrm{Err}^\mathscr{R}(z)
	= \OO \left( \frac{(\log N)^{200}}{N^2} \right)
	+ \OO \left( \frac{(\log N)^{200}}{N^2} 
		\sum_{\ell=1}^m \frac{1}{\eta_\ell} \right).
\end{align}
\end{lemma}
\begin{proof} 
First note that
\begin{align} \nonumber 
	& \int_\R \frac{f'(s)}{s-z} \left( \varrho_1^{(N,t,\mathscr{R})}(s) - \varrho_{V}(s) \right) \diff s
	= \OO \left(\frac{(\log N)^{100}}{N} \int_\R \left(\frac{\abs{f''(s)}}{\abs{z-s}}+\frac{\abs{f'(s)}}{\abs{z-s}^2}\right) \diff s \right), \\
	\label{eqn:elementary3}
	& \varphi^\mathscr{R}(z) = \OO \left( \frac{(\log N)^{100}}{ N\eta} \right), 
	\qquad 
	(\varphi^\mathscr{R})'(z) = \OO \left( \frac{(\log N)^{100}}{N\eta^2} \right),
	\qquad 
	\Var_{\mu_N^{t,\mathscr{R}}} \left( s_N(z) \right) 
	= \OO \left( \frac{(\log N)^{200}}{ (N\eta)^2} \right).
\end{align}
The proof of these estimates is almost the same as that of Lemma 5.3 in \cite{BouErdYauYin2016}. The only differences are that (1) the rigidity estimate is now known with multiplicative error $(\log N)^{100}$ instead of $N^\xi$, (2) we work directly on the event $\mathscr{R}$ so we do not need to control what happens on $\mathscr{R}^c$, (3) we work with Laplace transform instead of Fourier transform.
In particular, this proves the first part of \eqref{eq:bound_with_rigidity}. The second part of \eqref{eq:bound_with_rigidity} follows from the bound (recall $f$ is supported on $[A-\tilde{\eta}/2,B+\tilde{\eta}/2]$)
\begin{align*}
\int_\R \left(\frac{\abs{f''(s)}}{\abs{z-s}}+\frac{\abs{f'(s)}}{\abs{z-s}^2}\right) \diff s
& \leq C\sum_{\ell=1}^m 
\int_{A-\tilde{\eta}/2}^{B+\tilde{\eta}/2} \left(\frac{1}{\abs{z-s}}
+\frac{1}{\abs{z_\ell-s}}\right)\frac{1}{\abs{z-s}\cdot\abs{z_\ell-s}} \diff s \\
& \leq C\sum_{\ell=1}^m 
\left(\frac{1}{\eta_\ell(\eta_{\ell}\vee\abs{z-z_\ell})}
+ \frac{1}{\eta(\eta\vee\abs{z-z_\ell})}\right).
\end{align*}
Since by Lemma \ref{lem:rigid2}, $\P_{\mu_N^t}(\mathscr{R}^c) \leq e^{-(\log N)^2}$, we have \eqref{eq:from_rigidity_to_without_rigidity}.
Finally, we obtain \eqref{eq:bound_with_rigidity_past_the_edge} similarly, noting that $\abs{z_\ell-z} \geq \tilde{\eta}/2$, 
and on the event $\mathscr{R}$ the particles $\lambda_k$ are confined in $[A-\tilde{\eta}/4,B+\tilde{\eta}/4]$ for $N$ large enough,
so we can replace $\eta$ by 1 in the bounds of \eqref{eqn:elementary3}.
\end{proof}
\begin{proof}[Proof of Lemma \ref{lem:estimate_varphi}]
Fix some $z = E + \ii \eta$ with $N^{-1} < \abs{\eta} \leq \tilde{\eta}/2$ and $A-\tilde{\eta} \leq E \leq B+\tilde{\eta}$.
We consider the rectangle with vertices $A-\tilde{\eta} \pm \ii N^{-10}$, $B+\tilde{\eta} \pm \ii N^{-10}$, and denote by $\cC$ the corresponding closed contour with positive orientation.
We decompose this contour into $\cC_{\text{hor}}$, which consists only in the horizontal pieces, and $\cC_{\text{ver}}$, which consists only in the vertical pieces.
By the loop equation \eqref{eq:loop_equation_mu_theta} and recalling $2 m_V(z) + V'(z) = 2r(z)b(z)$, we have
\begin{align*} 
	\int_{\cC_{\text{hor}}} \frac{2r(w)b(w) \varphi(w) - \psi(w) + \mathrm{Err}(w)}{r(w)(z-w)} \diff w
	= 0.
\end{align*}
Therefore, by \eqref{eq:from_rigidity_to_without_rigidity}, we have 
\begin{align} \label{eq:C_hor}
	\int_{\cC_{\text{hor}}} \frac{2r(w)b(w) \varphi^\mathscr{R}(w) - \psi(w) + \mathrm{Err}^\mathscr{R}(w)}{r(w)(z-w)} \diff w
	= \OO \left( \eta^{-2} e^{-(\log N)^2}\right).
\end{align}
On the other hand, for $w$ on $\cC_{\text{ver}}$, we have $2r(w)b(w) \varphi^\mathscr{R}(w) - \psi(w) + \mathrm{Err}^\mathscr{R}(w) = O(1)$ (using that under $\mu_N^{t,\mathscr{R}}$ particles are at a distance larger than $\tilde{\eta}/2$ from $\cC_{\text{ver}}$), so 
\begin{align} \label{eq:C_ver}
	\int_{\cC_{\text{ver}}} \frac{2r(w)b(w) \varphi^\mathscr{R}(w) - \psi(w) + \mathrm{Err}^\mathscr{R}(w)}{r(w)(z-w)} \diff w
	= \OO \left( \eta^{-1} N^{-10} \right).
\end{align}
Combining \eqref{eq:C_hor} and \eqref{eq:C_ver}, we get
\begin{align} \label{eq:C}
	\int_\cC \frac{2r(w)b(w) \varphi^\mathscr{R}(w) - \psi(w) + \mathrm{Err}^\mathscr{R}(w)}{r(w)(z-w)} \diff w
	= \OO \left( \frac{1}{(N\eta)^2} \right).
\end{align}
We now estimate each term in the last integral successively.

We start with the part involving $\varphi^\mathscr{R}(w)$. 
The function $w \mapsto 2b(w) \varphi^\mathscr{R}(w)/(z-w)$ is analytic on and outside $\cC$, except for the pole at $z$, and it behaves as $\OO(w^{-2})$ as $\abs{w} \to \infty$. Therefore, by the Cauchy integral formula with residue at infinity, we get
\begin{align} \label{eq:C_varphi}
	\int_\cC \frac{2b(w)\varphi^\mathscr{R}(w)}{(z-w)} \diff w
	= 4 \ii \pi b(z) \varphi^\mathscr{R}(z)
	= 4 \ii \pi b(z) \varphi(z) + \OO \left( \frac{1}{(N\eta)^2} \right),
\end{align}
using again \eqref{eq:from_rigidity_to_without_rigidity}.

Now we evaluate the part involving $\psi(w)$. Recall the definition of $\psi(w)$ in \eqref{eq:def_psi} and note that the third term is analytic in $w \in \Omega$. 
Moreover,by \eqref{eq:link_2m+V'_and_r}, we have $m_V'(w) = -\frac{1}{2} V''(w) + (rb)'(w)$, where $V''(w)$ is also analytic in $w \in \Omega$.
Since the contour $\cC$ is included in $\Omega$, $z$ is exterior to $\cC$ and $r$ has no zero inside $\cC$ (see the choice of $\tilde{\eta}$ in Lemma \ref{lem:preliminaries_fixed_point_equation}), these analytic terms disappear and we get
\begin{align*}
	\int_\cC \frac{\psi(w)}{r(w)(z-w)} \diff w
	& = \int_\cC \left( \frac{2t}{\beta N} \int_A^B \frac{f'(s)}{s-w} \varrho_{V}(s) \diff s - \frac{1}{N}\pa{\frac{2}{\beta}-1} (rb)'(w) \right) \frac{\diff w}{r(w)(z-w)} \\
	& = -\frac{4\ii\pi t}{\beta N} \int_A^B \frac{f'(s)}{r(s)(z-s)} \varrho_{V}(s) \diff s
	- \frac{1}{N}\pa{\frac{2}{\beta}-1} 
	\int_A^B \frac{-2i (r\tau)'(s)}{r(s)(z-s)} \diff s
\end{align*}
where, for the first term, we applied Cauchy's integral formula and, for the second term, we let the contour approach the segment $[A,B]$ and used $\lim_{y \to 0+} (rb)'(x\pm iy) = \pm i(r\tau)'(x)$ for $x \in (A,B)$, recalling $\tau(x) = \sqrt{(x-A)(B-x)}$.
Recalling the definition of $\widetilde{\varphi}(z)$ in \eqref{eq:def_phi_tilde} and that $\varrho_V = \frac{1}{\pi} r \tau$, we get
\begin{align} \label{eq:C_psi}
	\int_\cC \frac{\psi(w) \diff w}{r(w)(z-w)}
	& = 2 \ii \left(
		\frac{2t}{\beta N} \int_A^B \frac{f'(s)}{s-z} \tau(s) \diff s
		- \frac{1}{N}\pa{\frac{2}{\beta}-1} 
		\int_A^B \left( \frac{\tau'(s)}{s-z} + \frac{r'(s) \tau(s)}{r(s)(s-z)} \right) \diff s
		\right) 
	= \frac{4 \ii \pi b(z)}{N} \widetilde{\varphi}(z),
\end{align}
where we used that $\int_A^B \frac{\tau'(s)}{s-z} \diff s = \int_A^B \frac{\tau(s)}{(s-z)^2} \diff s = \pi (b'(z)-1)$
because $\int_A^B \frac{\tau(s)}{s-z} \diff s = \pi (\frac{A+B}{2} - z + b(z))$.

Finally, we deal with the part involving $\mathrm{Err}^\mathscr{R}(w)$. We deform the contour $\cC$ into $\cC'$, the positively oriented rectangle with vertices $A-\tilde{\eta} \pm \ii \eta/2$, $B+\tilde{\eta} \pm \ii \eta/2$.
The function $w \mapsto \mathrm{Err}^\mathscr{R}(w) / (r(w)(z-w))$ is analytic on and outside these contours, so
\begin{align} \label{eq:C_Err}
	\int_\cC \frac{\mathrm{Err}^\mathscr{R}(w)}{r(w)(z-w)} \diff w
	= \int_{\cC'} \frac{\mathrm{Err}^\mathscr{R}(w)}{r(w)(z-w)} \diff w
	= \OO \left( \frac{(\log N)^{201}}{(N\eta)^2} \right)
	+ \OO \left( \frac{(\log N)^{201}}{N^2} 
			\sum_{\ell=1}^m \frac{1}{\eta_\ell(\eta_{\ell}\vee|z-z_\ell|)} \right),
\end{align}
where we used that $\abs{r(w)}$ is uniformly lower bounded and we applied  \eqref{eq:bound_with_rigidity} on the horizontal pieces of $\cC'$ and \eqref{eq:bound_with_rigidity_past_the_edge} on the vertical pieces. 
Coming back to \eqref{eq:C} and combining \eqref{eq:C_varphi}, \eqref{eq:C_psi} and \eqref{eq:C_Err}, the result is proved.
\end{proof}
We now prove the following Lemma \ref{lem:estimate_Z(theta)} in which we estimate $Z(t)$ (defined in (\ref{eq:def_Z(theta)})) via the Hellfer--Sj\"{o}strand formula applied to $\E_{\mu_N^t} \left[ S_N(f) \right]$, and the estimates of Lemma \ref{lem:estimate_varphi}. Here our method essentially follows Section 5 of~\cite{BouErdYauYin2016}. 
In order to state the lemma, first introduce
\begin{align}
	\label{eq:def_sigma}
	\sigma^2(f) 
	& \coloneqq \frac{1}{ \pi^2 \beta} 
	\int_A^B \int_A^B f'(s) \pa{ \frac{f(s)-f(t)}{s-t} } \frac{\tau(s)}{\tau(t)} \diff s \diff t, \\
	\label{eq:def_delta}
	\delta(f) 
	& \coloneqq 
	\pa{\frac{2}{\beta}-1} 
	\pa{\frac{f(A) + f(B)}{4} 
		-\frac{1}{2\pi^2} \int_A^B \frac{f(x)}{\tau(x)} 
			\left( \pi + {\rm p.v.} \int_A^B \frac{r'(s) \tau(s)}{r(s)(x-s)} \diff s \right) 
			\diff x},
\end{align}
recalling $\tau(s) \coloneqq \sqrt{(s-A)(B-s)}$
We will study the asymptotic behavior of these quantities in Section \ref{subsection:covariances}.
\begin{lemma} \label{lem:estimate_Z(theta)}
For any $m \geq 1$, $E_1, \dots,E_m \in [A,B]$, $a_1,\dots,a_m,b_1,\dots,b_m \in [-1,1]$ and any $|t|<t_0$, with $t_0$ from Lemma \ref{lem:rigid2}, we have
\[
	Z(t) 
	= \exp \left(\frac{t^2}{2} \sigma^2(f) + t \delta(f) 
	+\OO\left(e^{-(\log N)^{1/4}/3}\right)\right),
\]
where the implied constant in the error term depends only on $\beta$, $V$ and $m$.
\end{lemma}

\begin{proof}
Recall $Z'(t) = Z(t) \E_{\mu_N^t} [S_N(f)]$, so we estimate $\E_{\mu_N^t} [S_N(f)]$.
By linearity, we can assume $m=1$, i.e. $f=f_1$ in $S_N(f)$ (the measure $\mu_N^t$ still depends on $f$).
Recall also $\tilde{\eta} > 0$ is the constant given by Lemma \ref{lem:preliminaries_fixed_point_equation}.
Let $\chi \colon \R \to [0,1]$ be a smooth symmetric function such that
$\chi(y) = 1$ for $y \in (-\tilde{\eta}/4,\tilde{\eta}/4)$ and $\chi(y) = 0$ for $\abs{y} > \tilde{\eta}/2$. 
By the Hellfer-Sj\"{o}strand formula,
\begin{equation}\label{eqn:HSf}
	\E_{\mu_N^t} \left[ S_N(f_1) \right]  
	 = \frac{1}{2\pi} \iint_{\R^2}
		(\ii y f_1''(x) \chi(y) + \ii(f_1(x) + \ii y f_1'(x))\chi'(y))
		N \varphi(z) \diff x \diff y,
\end{equation}
where we set $z = x + \ii y$ for brevity.
We denote 
\[
\Sigma(f_1)
\coloneqq \frac{1}{2\pi} \iint_{\R^2}
		(\ii y f_1''(x) \chi(y) + \ii(f_1(x) + \ii y f_1'(x))\chi'(y))
		\widetilde{\varphi}(z) \diff x \diff y.
\]
Then, abbreviating $r_1=(\ell(E_1)\eta_1)^{1/2} = \ell(E_1) e^{(\log N)^{1/4}/2}$, 
integrating by parts (first in $x$, then in $y$) using analyticity of $N\varphi-\widetilde{\varphi}$ on $\{y>0\}$,
we have,
\begin{align*}
	\E_{\mu_N^t} \left[ S_N(f_1) \right] -\Sigma(f_1)
	&= \frac{1}{\pi} \im\iint_{y>0}
	y f''(x) \chi(y)
	\left(N \varphi(z)-\widetilde{\varphi}(z)\right)
	+\frac{1}{\pi} \im \iint_{y>0}
	(f_1(x) + \ii y f_1'(x))\chi'(y)
	\left(N \varphi(z)-\widetilde{\varphi}(z)\right) \\
	&={\rm (I)}+{\rm (II)}+{\rm(III)}+{\rm(IV)} \\
	{\rm (I)}&\coloneqq \frac{1}{\pi} \im \iint_{0<y<r_1}
	y f_1''(x)
	\left(N \varphi(z)-\widetilde{\varphi}(z)\right) \diff y \diff x \\
	{\rm (II)}&\coloneqq -\frac{r_1}{\pi} \re\int 
	f_1'(x)
	\left( N \varphi(x+\ii r_1) - \widetilde{\varphi}(x+\ii r_1) \right) \diff x \\
	{\rm (III)}&\coloneqq -\frac{1}{\pi} \re\iint_{y>r_1}
	f_1'(x) \partial_y(y\chi(y))
	\left(N \varphi(z)-\widetilde{\varphi}(z)\right) \diff y \diff x \\
	{\rm (IV)}&\coloneqq \frac{1}{\pi} \im \iint_{y>0}
	(f_1(x) + \ii y f_1'(x))\chi'(y)
	\left(N \varphi(z)-\widetilde{\varphi}(z)\right) \diff y \diff x.
\end{align*}
We now bound each of the above four terms. 
Note that thanks to our choice of the cutoff functions $\chi$ and $\phi$ (involved in the definition of $f$), these integrals have support in the region where we can apply Lemma \ref{lem:estimate_varphi}.

The contribution of $({\rm IV})$ is trivial, as $\chi'\neq 0$ together with Lemma \ref{lem:estimate_varphi} imply $\abs{N\varphi-\widetilde{\varphi}} \leq C (\log N)^{200} N^{-1} \sum_{\ell=1}^m \eta_\ell^{-1}$, so that $\abs{({\rm IV})} \leq C (\log N)^{200} e^{-(\log N)^{1/4}} \int (\abs{f_1}+\abs{f_1'}) \leq (\log N)^{201} e^{-(\log N)^{1/4}}$.

We now bound $({\rm I})$.
A simple analysis of $s\mapsto s/\abs{s-w}^2$ shows the maximum is obtained for $s=\abs{w}$, so that
\[
	\frac{\tau(s)}{\abs{z-s}}
	\leq C\frac{\sqrt{\kappa(z)}}{\abs{z-A-\kappa(z)}\wedge\abs{B-z-\kappa(z)}}
	\leq C \frac{\sqrt{\kappa(z)}}{y},
\]
recalling $\kappa(z) = \abs{z-A} \wedge \abs{z-B}$.
Together with $\int \abs{f'} \leq C m\log N$ and $\abs{\int_A^B \frac{r'(s) \tau(s)}{r(s)(s-z)} \diff s} \leq C \log (1/y)$ and recalling the definition of $\widetilde{\varphi}$ in \eqref{eq:def_phi_tilde}, this gives
\[
	\abs{\widetilde{\varphi}(z)}
	\leq \frac{C}{\sqrt{\kappa(z)}} \left( 
		C m (\log N) \frac{\sqrt{\kappa(z)}}{y}
		+\frac{C}{\sqrt{\kappa(z)}}+C\log (1/y)
	\right)
	\leq \frac{C\log N}{y}.
\]
Using \eqref{eq:bound_with_rigidity} and \eqref{eq:from_rigidity_to_without_rigidity} to control $\varphi$, we conclude that
\[
	\abs{({\rm I})} 
	\leq C (\log N)\iint_{0<y<r_1} \abs{f_1''(x)} \diff y \diff x
	\leq C (\log N) \frac{r_1}{\eta_1}
	= \frac{C \log N}{e^{(\log N)^{1/4}/2}}.
\]

For $({\rm II})$, we bound $N\varphi-\widetilde{\varphi}$ with Lemma \ref{lem:estimate_varphi} and obtain (in the equation below $z=x+\ii r_1$)
\[
	\abs{({\rm II})} 
	\leq r_1 \frac{(\log N)^{201}}{N}
	\int \frac{\abs{f_1'(x)}}{\abs{b(z)}}\left(\frac{1}{r_1^2}+\sum_{\ell=1}^m
	\frac{1}{\eta_\ell(\eta_{\ell}\vee|z-z_\ell|)}\right) \diff x,
\]
and we now distinguish the cases $\eta_{\ell}>r_1/10$ and $\eta_{\ell}\leq r_1/10$.
If  $\eta_{\ell}>r_1/10$ the above sum is absorbed in the $r_1^{-2}$ term and these terms are bounded through 
\begin{align}
\int\frac{\rd x}{|x-z_1|\sqrt{\kappa(x+\ii r_1)}}
& \leq \int_{\kappa(x)>\kappa(z_1)/10}\frac{\rd x}{|x-z_1|\sqrt{\kappa(x)}}
+
\int_{\kappa(x)<\kappa(z_1)/10}\frac{\rd x}{|x-z_1|\sqrt{\kappa(x)}} \nonumber \\
& \leq \frac{C}{\sqrt{\kappa(z_1)}}\int\frac{\rd x}{|x-z_1|}+\frac{C}{\kappa(z_1)}\int_{\kappa(x)<\kappa(z_1)/10}\frac{\rd x}{\sqrt{\kappa(x)}}\leq \frac{C\log N}{\sqrt{\kappa(z_1)}}. \label{eq:bound_int_II}
\end{align}
If $\eta_\ell<r_1/10$, denoting $z=x+\ii r_1$ we have 
$\eta_\ell<|z_\ell-z|$ and $\kappa(E_\ell)>\kappa(E_1)+\eta_1$,
so the relevant bound is
\begin{align*}
	& \frac{1}{\eta_\ell}\left(
	\int_{\kappa(x)=0}^{\kappa(z_1)/10}
	+\int_{\kappa(z_1)/10}^{E_\ell/2}
	+\int_{E_\ell/2}^{\infty} \right)
	\frac{\rd x}{|z_1-x|\sqrt{\kappa(z)}|z-z_\ell|} \\
	& \leq
	\frac{C}{\eta_\ell\kappa(z_1)|z_1-z_\ell|}\int_0^{\kappa(z_1)} \frac{\rd u}{\sqrt{u}}
	+
	\frac{1}{\eta_\ell|z_1-z_\ell|\sqrt{\kappa(z_1)}}\int\frac{\rd x}{|z_1-x|}
	+
	\frac{1}{\eta_\ell|z_1-z_\ell|\sqrt{\kappa(z_\ell)}}\int\frac{\rd x}{|z-z_\ell|}
	\\
	& \leq \frac{\log N}{\sqrt{\kappa(z_1)}}\sup_{\kappa(E_\ell)>\kappa(E_1)+\eta_1}\frac{1}{|E_1-E_\ell|\eta_\ell}=
	\frac{N\log N}{e^{(\log N)^{1/4}}\sqrt{\kappa(z_1)}}\sup_{\kappa(E_\ell)>\kappa(E_1)+\eta_1}\frac{\sqrt{\kappa(E_\ell)}}{|E_1-E_\ell|}\leq \frac{CN\log N}{e^{(\log N)^{1/4}}\eta_1}.
\end{align*}
From the previous equations we deduce that
\[
	\abs{({\rm II})}
	\leq \frac{C(\log N)^{202}}{N r_1 \sqrt{\kappa(z_1)}}
	+ \frac{C(\log N)^{202}r_1}{e^{(\log N)^{1/4}}\eta_1}
	\leq \frac{C(\log N)^{202}}{e^{(\log N)^{1/4}/2}}.
\]

The main error comes from $({\rm III})$. 
The contribution from $\iint_{y>r_1} f_1'(x) y\chi'(y)  (N\varphi(z)-\widetilde{\varphi}(z))$ already appeared in~$({\rm IV})$.
The remaining term is ($z=x+\ii y$ below)
\[
\frac{(\log N)^{201}}{N} \iint_{y>r_1} 
\frac{|f_1'(x)|}{\abs{b(x+\ii y)}}\left(\frac{1}{y^2}+\sum_{\ell=1}^m
\frac{1}{\eta_\ell(\eta_{\ell}\vee|z-z_\ell|)} \right) \diff y \diff x.
\]
For a given $\ell$, we bound the contribution from the domain $\eta_\ell>y/10$ using
\[
\iint_{y>r_1} \frac{\rd x\rd y}{|z_1-x| \sqrt{\kappa(x+\ii y)}y^2} 
\leq\int\frac{\rd x}{|z_1-x| \sqrt{\kappa(x+\ii r_1)} r_1} 
\leq \frac{C\log N}{r_1\sqrt{\kappa(z_1)}} \leq\frac{C(\log N)}{e^{(\log N)^{1/4}/2}},
\]		
where the second inequality follows from \eqref{eq:bound_int_II}.
On the complementary domain, note that $\eta_\ell<|z_\ell-z|$ and 
\[
\iint_{y> r_1\vee 10\eta_\ell} \frac{\rd x\rd y}{\eta_\ell|z_1-x| \sqrt{\kappa(z)}|z-z_\ell|}
\leq C(\log N)\int\frac{\rd x}{\eta_\ell|z_1-x| \sqrt{\kappa(E_\ell)}}
\leq \frac{C(\log N)^{2}}{\eta_\ell \sqrt{\kappa(E_\ell)}}
= \frac{C N(\log N)^{2}}{e^{(\log N)^{1/4}}}.
\]
We have therefore proved that
\[
\E_{\mu_N^t} \left[ S_N(f) \right]  =\Sigma(f)+\OO\left(\frac{(\log N)^{203}}{e^{(\log N)^{1/4}/2}}\right).
\]
From Lemma \ref{lem:delta(f)_and_sigma(f)} below, we can also write this as
\[
\E_{\mu_N^t} \left[ S_N(f) \right]  =t\sigma^2(f)+\delta(f)+\OO\left(e^{-(\log N)^{1/4}/3}\right).
\]
Since $Z'(t)/Z(t) = \E_{\mu_N^t} \left[ S_N(f) \right]$, the result follows by integrating with respect to $t$.
\end{proof}

\begin{lemma} \label{lem:delta(f)_and_sigma(f)}
For any real function $g$ of class $C^2$ with compact support, we have
\[ 
	\Sigma(g) = \frac{t}{\beta \pi^2} \int_A^B \int_A^B f'(s) \frac{g(s)-g(t)}{s-t} \frac{\tau(s)}{\tau(t)} \diff t \diff s
	+ \delta(g).
\]
\end{lemma}
\begin{proof}
We write $\partial_{\bar{z}} = \frac{1}{2} (\partial_x + \ii \partial_y)$.
Let $\widetilde{g}(z) \coloneqq \pa{ g(x) +  \ii g'(x) y } \chi(y)$ so that
\[
	\Sigma(g) 
	= \frac{1}{\pi} 
	\iint_{\R^2} (\partial_{\bar{z}} \widetilde{g}(z)) \widetilde{\varphi}(z) \diff x \diff y
	= \lim_{\e \to 0} \frac{1}{\pi} 
	\iint_{\{ \abs{y} > \e\}} \partial_{\bar{z}} \left( \widetilde{g}(z) \widetilde{\varphi}(z) \right) \diff x \diff y,
\]
using that $\partial_{\bar{z}} \widetilde{\varphi}(z) = 0$ because $\widetilde{\varphi}(z)$ is analytic. 
Then, applying Green's formula, we get
\begin{align} \label{eq:new_expression}
	\Sigma(g) 
	= \lim_{\e \to 0^+} \frac{1}{2\ii \pi} \int_\R \pa{
		\widetilde{g}(x + \ii \e) \widetilde{\varphi}(x+\ii \e) 
		- \widetilde{g}(x-\ii \e) \widetilde{\varphi}(x-\ii \e)
	}\diff x
	= \lim_{\e \to 0^+} \frac{1}{\pi} \int_\R
	\im \left( \widetilde{g}(x + \ii \e) \widetilde{\varphi}(x+\ii \e) \right) \diff x,
\end{align}
noting that $\widetilde{g}(\overline{z}) = \overline{\widetilde{g}(z)}$ and $\widetilde{\varphi}(\overline{z}) = \overline{\widetilde{\varphi}(z)}$.

We start with the first term appearing in $\widetilde{\varphi}(z)$, that is $\widetilde{\varphi}_1(z) \coloneqq \frac{t}{\beta \pi b(z)} \int_A^B \frac{f'(s)}{s-z} \tau(s) \diff s$.
Note that, as $\e\to 0^+$, $b(x+\ii \e)$ tends to $\ii \tau(x)$ if $x \in (A,B)$ and to $\sqrt{\abs{x-A} \abs{x-B}}$ if $x \notin [A,B]$. On the other hand, we have
\begin{align*}
	\int_A^B \frac{f'(s)}{s-(x+\ii \e)} \tau(s) \diff s 
	& \xrightarrow[\e\to 0^+]{} 
	\begin{cases}
	{\rm p.v.} \int_A^B \frac{f'(s)}{s-x} \tau(s) \diff s + \ii \pi f'(x) \tau(x),
	& \text{if } x \in (A,B), \vphantom{\frac{p}{\frac{p}{p}}} \\
	\int_A^B \frac{f'(s)}{s-x} \tau(s) \diff s, & \text{if } x \notin [A,B],
	\end{cases}
\end{align*}
and finally $\widetilde{g}(x + \ii \e) \to g(x)$.
In order to apply the dominated convergence theorem, we use
\[
	\absa{\int_A^B \frac{f'(s)}{s-z} \tau(s) \diff s} 
	= \absa{\int_A^B \log(s-z) (f'\tau)'(s) \diff s}
	\leq C (\|f'\|_\infty + \|f''\|_\infty) 
	\int_A^B  \frac{1+ \log\abs{s-x}}{\tau(s)} \diff s
	\leq C (\|f'\|_\infty + \|f''\|_\infty),
\]
so that $\abs{\widetilde{g}(z) \widetilde{\varphi}_1(z)} \leq Ct \|\tilde{g}\|_\infty (\|f'\|_\infty + \|f''\|_\infty) (\abs{x-A} \abs{x-B})^{-1/2} \1_{x \in \supp g}$. Therefore, we get
\begin{align} \label{eq:phi_1}
	\lim_{\e \to 0^+} \frac{1}{\pi} \int_\R
	\im \left( \widetilde{g}(x + \ii \e) \widetilde{\varphi}_1(x+\ii \e) \right) \diff x
	& = \frac{-t}{\beta \pi^2} \int_A^B \frac{g(x)}{\tau(x)} 
	\left( {\rm p.v.} \int_A^B \frac{f'(s)}{s-x} \tau(s) \diff s \right) \diff x \nonumber \\
	& = \frac{-t}{\beta \pi^2} \int_A^B f'(s) \tau(s) \left( {\rm p.v.} \int_A^B \frac{g(x)}{s-x} \frac{1}{\tau(x)} \diff x \right) \diff s \nonumber \\
	& = \frac{t}{\beta \pi^2} \int_A^B f'(s) \tau(s) \int_A^B \frac{g(s)-g(x)}{s-x} \frac{1}{\tau(x)} \diff x\diff s,
\end{align}
where we used that ${\rm p.v.} \int_A^B \frac{1}{s-x} \frac{1}{\tau(x)} \diff x = 0$.

We now deal with the second term appearing in $\widetilde{\varphi}(z)$, that is $\widetilde{\varphi}_2(z) \coloneqq -(\frac{2}{\beta}-1) \frac{b'(z)}{2 b(z)}$.
Note that, as $\e \to 0^+$, $\widetilde{\varphi}_2(x+\ii\e)$ has a real limit for any $x \notin \{A,B\}$, but it cannot be dominated for any $x$.
Therefore, we fix some $\delta \in (0,\frac{B-A}{3}]$ and distinguish between the cases $\kappa(x) \leq \delta$ and $\kappa(x) > \delta$.
In the case $\kappa(x) > \delta$, we can bound $\widetilde{\varphi}_2(x+\ii\e)$ uniformly by some constant depending on $\delta$ and we get
\begin{align} 
	\lim_{\e \to 0^+} \frac{1}{\pi} \int_{\{\kappa(x) > \delta\}}
	\im \left( \widetilde{g}(x + \ii \e) \widetilde{\varphi}_2(x+\ii \e) \right) \diff x
	& = 0.
\end{align}
We now deal with the case $\kappa(x) \leq \delta$, that is $x \in [A-\delta,A+\delta]$ or $x \in [B-\delta,B+\delta]$.
Both parts are treated similarly, so we focus on the integral on $[B-\delta,B+\delta]$.
Integrating by parts $b'/b$ and then letting $\e \to 0^+$, we get
\begin{align*}
	\lim_{\e \to 0} \int_{B-\delta}^{B+\delta} 
    \widetilde{g}(x+\ii \e) \frac{b'(x+\ii \e)}{b(x+\ii \e)} \diff x
	& = 
	g(B+\delta) 
	\log \left(\sqrt{(B+\delta-A)\delta} \right)
	- g(B-\delta) 
	\log \left(\ii \sqrt{(B-\delta-A)\delta} \right) \\
	& \hphantom{{}=} {}
	-\int_{B-\delta}^B g'(x) \log (\ii\tau(x)) \diff x
	-\int_B^{B+\delta} g'(x) \log \left(\sqrt{(x-A)(x-B)} \right) \diff x,
\end{align*}
which converges to $-\frac{\ii \pi}{2} g(B)$ as $\delta \to 0$.
Proceeding similarly for the integral on $[A-\delta,A+\delta]$, we finally get
\begin{align} \label{eq:phi_2}
	\lim_{\e \to 0^+} \frac{1}{\pi} \int_\R
	\im \left( \widetilde{g}(x + \ii \e) \widetilde{\varphi}_2(x+\ii \e) \right) \diff x
	& = \left( \frac{2}{\beta}-1 \right) \frac{g(A) + g(B)}{4}.
\end{align}

The remaining terms in $\widetilde{\varphi}(z)$ is 
$\widetilde{\varphi}_3(z) \coloneqq \frac{1}{2\pi b(z)} (\frac{2}{\beta}-1) (\pi - \int_A^B \frac{r'(s) \tau(s)}{r(s)(s-z)} \diff s)$ 
and proceeding as for $\widetilde{\varphi}_1(z)$, we get
\begin{align} \label{eq:phi_3}
	\lim_{\e \to 0^+} \frac{1}{\pi} \int_\R
	\im \left( \widetilde{g}(x + \ii \e) \widetilde{\varphi}_3(x+\ii \e) \right) \diff x
	& = \frac{-1}{2\pi^2} \pa{\frac{2}{\beta}-1} \int_A^B \frac{g(x)}{\tau(x)} 
	\left( \pi - {\rm p.v.} \int_A^B \frac{r'(s) \tau(s)}{r(s)(s-x)} \diff s \right) 
	\diff x.
\end{align}
Coming back to \eqref{eq:new_expression} and combining \eqref{eq:phi_1}, \eqref{eq:phi_2} and \eqref{eq:phi_3}, we get the result.
\end{proof}

\subsection{Rewriting the limiting characteristic function. }
\label{subsection:covariances}

In the previous section, we have seen that the limiting behavior of $Z(t)$ can be expressed in terms of $\sigma^2(f)$ and $\delta(f)$, defined in \eqref{eq:def_sigma} and \eqref{eq:def_delta}.
Before proving Proposition \ref{prop:CLT_above_the_axis}, we prove the following lemma.
\begin{lemma} \label{lem:estimates_sigma_delta}
For any $m \geq 1$, $E_1, \dots,E_m \in (A,B)$ and $a_1,\dots,a_m,b_1,\dots,b_m \in [-1,1]$, we have 
\begin{align*}
    \delta(f) & = \frac{1}{4} \pa{\frac{2}{\beta}-1} 
    \sum_{\ell=1}^m a_\ell \log (\kappa_\ell \vee \eta_\ell) + \OO(1), \\
	\sigma^2(f) & =
	-\frac{1}{\beta} \sum_{\ell,j =1}^m 
	\Bigg(
	a_\ell a_j 
		\log\absa{\bar{z_\ell} - z_j}
	+ b_\ell b_j 
		 \log\pa{\frac{\absa{\bar{z_\ell} - z_j}}{\kappa_\ell \vee \eta_\ell} \wedge 1}
	\Bigg) + \OO\pa{\log \log N}.
\end{align*}
\end{lemma}
\begin{proof}
The estimate for $\delta(f)$ is direct from its definition in \eqref{eq:def_delta}, noting that the integral part is bounded.
Hence we focus on $\sigma^2(f)$ in this proof. 
Recall $\tau(t) = \sqrt{(B-t)(t-A)}$.
By bilinearity, it is sufficient to estimate
\begin{align*}
	\sigma^2(z,z') 
	\coloneqq \frac{1}{\pi^2 \beta} 
	\int_A^B \int_A^B \frac{1}{s-z} 
	\left( 
	\frac{\log(z'-s)-\log(z'-t)}{s-t} 
	\right) 
	\frac{\tau(s)}{\tau(t)} \diff s \diff t, 
\end{align*}
where $z = E \pm \ii \eta(E)$ for some $E \in [\frac{A+B}{2},B)$ and $z' = E' + \ii \eta(E')$ for some $E' \in (A,B)$.
Note that we assumed here w.l.o.g.\@ that $\re(z) \geq \frac{A+B}{2}$ and $\im(z') > 0$. 
In the sequel, error terms are uniform in $z$ and~$z'$.
Writing $\log(z'-s) = \int_{z'}^{E'+\ii \log N} \frac{1}{s-\omega} \diff \omega + \log(E'+\ii \log N- s)$, we get
\begin{align*}
	\sigma^2(z,z') 
	& = \frac{1}{\pi^2 \beta} 
	\int_A^B \int_A^B \frac{1}{s-z} 
	\int_{z'}^{E'+\ii \log N}
	\frac{1}{(\omega-s)(t-\omega)}
	\diff \omega 
	\frac{\tau(s)}{\tau(t)} \diff s \diff t + \OO(1) \\
	& = - \frac{1}{\beta} 
	\int_{z'}^{E'+\ii \log N}
	\frac{z-\omega + \sqrt{\omega-A}\sqrt{\omega-B} - \sqrt{z-A}\sqrt{z-B}}{(z-\omega) \sqrt{\omega-A}\sqrt{\omega-B}}
	\diff \omega 
	+ \OO(1),
\end{align*}
where the second equality results of the identities
$\int_A^B \frac{1}{t-\omega} \frac{\diff t}{\tau(t)} 
= \frac{-\pi}{\sqrt{\omega-A}\sqrt{\omega-B}}$, 
then
$\frac{1}{(s-z)(\omega-s)} 
= \frac{1}{z-\omega}(\frac{1}{s-\omega} - \frac{1}{s-z})$
and finally
$\int_A^B \frac{\tau(s)}{s-\omega} \diff s 
= \pi ( \frac{A+B}{2} - \omega + \sqrt{\omega-A}\sqrt{\omega-B})$.
Furthermore, we have the following explicit antiderivative
\begin{align} \label{eq:integral_explicit}
\int \frac{1}{z-\omega} \frac{\sqrt{z-A}\sqrt{z-B}}{\sqrt{\omega-A}\sqrt{\omega-B}} \diff \omega 
= 2 \tanh^{-1}\pa{\frac{\sqrt{\omega-A} \sqrt{z-B}}{\sqrt{\omega-B}\sqrt{z-A}}},
\end{align}
where $\tanh^{-1}(w) = \frac{1}{2} \log (\frac{1+w}{1-w})$ for $w \in \C \setminus ((-\infty,-1] \cup [1,\infty))$. Using that $\re(z) \geq \frac{A+B}{2}$ and $\im(z') > 0$, one can check that for any $\omega \in [z',E+\ii\log N]$ the argument of $\tanh^{-1}$ in \eqref{eq:integral_explicit} is in $\C \setminus ((-\infty,-1] \cup [1,\infty))$, except if $\omega = z' = z$.
Hence, assuming for now that $z \neq z'$, we get 
\begin{align} \label{eq:sigma_1}
	\sigma^2(z,z') 
	& = \frac{1}{\beta} 
	\left(
	-\log\absa{z - z'}
	- 2 \tanh^{-1}\pa{\frac{\sqrt{z'-A} \sqrt{z-B}}{\sqrt{z'-B}\sqrt{z-A}}}
	\right)
	+ \OO(\log \log N).
\end{align}
We now assume that $\abs{z-z'} \leq \abs{z-A}\abs{z-B}/\log N$. Then, letting $\theta = \pm 1$ denote the sign of $\im(z)$ and using that $\im(z') > 0$ and $E,E' \in (A,B)$, we have the following expansion
\begin{align*}
	\frac{\sqrt{z'-A} \sqrt{z-B}}{\sqrt{z'-B}\sqrt{z-A}}
	= \theta 
	\left( 1 + \frac{z'-z}{2(z-A)} (1 + \oo(1)) \right)
	\left( 1 - \frac{z'-z}{2(z-B)} (1 + \oo(1)) \right)
	= \theta + \theta \frac{(z'-z)(A-B)}{2(z-A)(z-B)} (1 + \oo(1)).
\end{align*}
Since $\tanh^{-1}(\theta+h) = - \frac{\theta}{2} \log \abs{h} + O(1)$ as $h \to 0$, it follows that
\begin{align} \label{eq:sigma_2}
	\sigma^2(z,z') 
	& = \frac{1}{\beta} 
	\left(
	-\log\absa{z-z'}
	+ \theta \log\pa{\frac{\abs{z-z'}}{\abs{z-A}\abs{z-B}}}
	\right)
	+ \OO(\log \log N).
\end{align}
On the other hand, in the case $\abs{z-z'} \geq \abs{z-A}\abs{z-B}/\log N$, 
the argument of $\tanh^{-1}$ in \eqref{eq:sigma_1} is at distance at least $c/\log N$ from $\pm 1$ for some $c >0$ and therefore $\sigma^2(z,z') = - \frac{1}{\beta} \log\absa{z-z'} + \OO(\log \log N)$.
Combining this with \eqref{eq:sigma_2}, we get, in any case such that $z \neq z'$,
\begin{align} \label{eq:sigma_3}
	\sigma^2(z,z') 
	& = \frac{1}{\beta} 
	\left(
	-\log\absa{z-z'}
	+ \theta \log\pa{\frac{\abs{z-z'}}{\kappa(E) \vee \eta(E)} \wedge 1}
	\right)
	+ \OO(\log \log N).
\end{align}
Taking the limit $z' \to z$ in the previous equation, we find $\sigma^2(z,z) = - \frac{1}{\beta} \log (\kappa(E) \vee \eta(E))$.
Recall now that 
$f (s) = \sum_{\ell=1}^k
\frac{a_\ell}{2} (\log(z_\ell-s) + \log(\bar{z_\ell}-s)) 
+ \frac{b_\ell}{2\ii} (\log(z_\ell-s) + \log(\bar{z_\ell}-s))$.
Since the main term on the right-hand side of \eqref{eq:sigma_3} is real, we get
\begin{align*}
	\sigma^2(f) = \sum_{\ell, j = 1}^m \biggr(
	\frac{a_\ell a_j}{2} \pa{ \sigma^2\pa{z_\ell, z_j} 
	+ \sigma^2\pa{\bar{z_\ell}, z_j} }
	- \frac{b_\ell b_j}{2}\pa{ \sigma^2\pa{z_\ell, z_j} 
	- \sigma^2\pa{\bar{z_\ell}, z_j} }
	\biggr) + \OO(\log \log N).
\end{align*}
For each $\ell$ and $j$, we apply \eqref{eq:sigma_3} and check that we get the desired result by distinguishing the cases 
$\frac{\abs{z_\ell-z_j}}{\kappa_\ell \vee \eta_\ell} \leq 1$ and 
$1 \leq \frac{\abs{z_\ell-z_j}}{\kappa_\ell \vee \eta_\ell}$.
In the first case, note that $\frac{\abs{\bar{z_\ell}-z_j}}{\kappa_\ell \vee \eta_\ell} \leq \frac{2 \eta_\ell + \abs{z_\ell-z_j}}{\kappa_\ell \vee \eta_\ell} \leq 3$, so we can omit the ``$\wedge 1$'' part in $\sigma^2(\bar{z_\ell}, z_j)$ as well.
In the second case, we have $\abs{z_\ell-z_j} \leq \abs{\bar{z_\ell}-z_j} \leq \abs{z_\ell-z_j} + 2 \eta_\ell \leq 3 \abs{z_\ell-z_j}$,
so $\abs{z_\ell-z_j}$ and $\abs{\bar{z_\ell}-z_j}$ are of the same order and this is sufficient to conclude.
\end{proof}
\begin{proof}[Proof of Proposition \ref{prop:CLT_above_the_axis}]
The result follows from Lemmas \ref{lem:estimates_sigma_delta} and \ref{lem:estimate_Z(theta)},
with the choice $a_\ell=\xi_\ell\sqrt{\beta/\log N}$, $b_\ell = \zeta_\ell\sqrt{\beta/\log N}$.
\end{proof}

\subsection{Proof of Lemma \ref{lem:loop}. }
\label{subsection:lemma}
In the following Lemma, for any $E \in [A+N^{-2/3},B-N^{-2/3}]$, the function $f_0$ is a smoothed version of $\1_{[E,B]}$ defined before Lemma \ref{lem:smallscale}.
\begin{lemma}\label{lem:loop}
There exists $t_1>0$ such that for any $|t|<t_1$, uniformly in $E \in [A+N^{-2/3},B-N^{-2/3}]$,
\[
\log \E \left[ e^{t \left(\sum f_0-N\int f_0\rd\mu_V\right)} \right]=\OO(\log N).
\]
\end{lemma}

\begin{proof}
The proof is identical to the proof of Proposition \ref{prop:CLT_above_the_axis} and in fact gives the following stronger estimate, analogous to Lemma \ref{lem:estimate_Z(theta)}: Denoting $\tilde Z(t) \coloneqq
	\E_{\mu_N} \left[ e^{t S_N(f_0)} \right]$, there exists $t_1>0$ such that for any $|t|<t_1$
	\begin{equation}\label{eqn:Lapf0}
\tilde Z(t) 
	= \exp \left(\frac{t^2}{2} \sigma^2(f_0) + t \delta(f_0) \right)\cdot\left(1+\oo\left((\log N)^{220}e^{-(\log N)^{1/4}}\right)\right).
\end{equation}
Indeed, denoting $z=E+\ii\eta(E)$ and  $g(s)=\phi(s)\im \log (z-s)$, for any $0\leq n\leq 2$ there exists $C_n$ such that $|f_0^{(n)}|\leq C_n|g^{(n)}|$ pointwise, so that the proof of the above equation is exactly the same as 
Lemma \ref{lem:estimate_Z(theta)}, through first the analogue of Lemma \ref{lem:rigid2}, i.e. rigidity of the measures biased by $e^{t f_0}$. 

With (\ref{eqn:Lapf0}) in hand, the end of the proof is elementary, as $\delta(f_0)=\OO(1)$ and $\sigma^2(f_0)=\OO(\log N)$ follow from their definitions.
\end{proof}

\setcounter{equation}{0}
\setcounter{theorem}{0}
\renewcommand{\theequation}{A.\arabic{equation}}
\renewcommand{\thetheorem}{A.\arabic{theorem}}
\appendix
\setcounter{secnumdepth}{0}
\section[Appendix A: Loop Equations]
{Appendix \mynameis{A}:\ \ \ Loop Equations
\label{sec:loop_equations}}


In this appendix we recall the loop equation hierarchy as stated in \cite{BorGui2013}. 
We then give equivalent forms of these equations which are combinatorially simpler and more convenient for the purpose of this paper, in terms of moments and centered moments.

\subsection{Cumulants. }

To simplify notation, we remove the dependence on $N$ and write $s(z) \coloneqq s_N(z)$.
Let
	\[
		c_n\pa{z_1, \dots, z_n}=\kappa_n\left(s\pa{z_1}, \dots, s\pa{z_n}\right) = 
		\left.\partial_{\e_1 \dots \e_n} \left(
			\log \E \left[ e^{\sum_{i=1}^n\e_i s\left( z_i \right)} \right]
		\right) \right|_{\e_i = 0}
	\]
denote the joint cumulant of $s\pa{z_1}, \dots, s\pa{z_n}$.
Recall the well known formulas relating moments and cumulants, for complex random variables $X_1,\dots,X_n$ with sufficiently high finite moments,
		\begin{align}
			\E\left[ X_1 \cdots X_n \right] 
			&= \sum_{\pi \in \cP_n} \prod_{B \in \pi} \kappa\left( X_i\,:\, i \in B\right)
			\label{moment expansion} \\
			\kappa_n\left(X_1, \dots, X_n\right) 
			&= \sum_{\pi \in \cP_n}\left(\abs{\pi} - 1 \right)! \left( -1 \right)^{\abs{\pi} - 1}
			\prod_{B \in \pi} \E\left[ \prod_{i \in B} X_i \right], \label{cumulant expansion}
		\end{align}
where $\cP_n$ denotes the set of partitions of $\{1,\dots,n\}$.
Moreover, for $n \geq 2$, joint cumulants are multilinear functions invariant by deterministic shifts:
	\begin{equation}
		\label{invariant by shifts}
		\kappa_n\left(X_1 - w_1, \dots, X_n - w_n \right) = \kappa_n\left(X_1, \dots, X_n\right),
	\end{equation}
for any constants $w_1,\dots,w_n \in \CC$, see for example the proof of \cite[Proposition 5.3.16]{AndGuiZei10}.

We now quote from \cite{BorGui2013} where the authors write loop equations for the measure \eqref{eq:beta_ensembles_density} multiplied by $\prod_{i=1}^N \1_{[a_{-}, a_{+}]}(\lambda_i)$ for some $-\infty < a_{-} < a_{+} < \infty$ (see also \eqref{eq:beta_ensembles_density_confined} below). 
Under our assumptions on $V$, we can take the limit $a_{\pm} \to \pm \infty$, and this gives us the following.
If $I$ is a set of indices, we write $z_I = (z_i)_{i\in I}$.
\begin{theorem}[Theorems 3.3 and 3.4 in \cite{BorGui2013}]
\label{loop equation for cumulants}
	For any $z \in \CC \setminus \R$, we have
		\begin{align} \label{eq:loop_equation_n=1}
			\E \left[ s(z)^2\right] 
			- \frac{1}{N}\pa{1-\frac{2}{\beta}} \E[s'(z)]
			+ \E \left[ \frac{1}{N} \sum_{k=1}^N \frac{V'\pa{\lambda_k}}{\lambda_k - z}\right] 
			= 0,
		\end{align}
	Moreover, for any $n \geq 2$, $z,z_1,\dots,z_{n-1} \in \CC \setminus \R$, we have, with $I =\{1,\dots,n-1\}$,
		\begin{align*}
			& c_{n+1}\left(z, z, z_I\right) 
			+ \sum_{J \subseteq I} c_{|J| + 1}\left(z, z_J\right) c_{n-|J|}(z, z_{I \backslash J})
			- \frac{1}{N}\left( 1 - \frac{2}{\beta}\right) \frac{\rd}{\rd z} c_n(z, z_I) \\
			& {} + \kappa_n \pa{\frac{1}{N}\sum_{k=1}^N 
				\frac{V'\pa{\lambda_k}}{\lambda_k - z}, s\pa{z_1}, \dots, s\pa{z_{n-1}}}
			+ \frac{2}{N^2\beta} \sum_{i \in I} \frac{\rd}{\rd z_i} \left(
				\frac{c_{n-1}(z, z_{I \backslash \{i\}}) - c_{n-1}(z_I)}{z-z_i}
			\right) = 0.
		\end{align*}
\end{theorem}
The following result rewrites the two first terms in the above loop equations for $n \geq 2$.
\begin{lemma}
	\label{simplified loop equation for cumulants}
	For any $n \geq 2$, $z,z_1,\dots,z_{n-1} \in \CC \setminus \R$, we have, with $I =\{1,\dots,n-1\}$,
	\[
		c_{n+1}\left(z, z, z_I\right) +
		 \sum_{J \subseteq I} c_{|J| + 1}(z, z_J) c_{n-|J|}(z, z_{I \backslash J})
		= \kappa_n\left( s^2(z), s\left(x_1\right), \dots, s\left(x_{n-1}\right) \right).
	\]
\end{lemma}

\begin{proof}
Recall that by definition,
	\begin{align*}
		c_{n+1}\left(z, z, z_1, \dots, z_{n-1}\right)
		&= \kappa_n\pa{s(z), s(z), s\pa{z_1}, \dots, s\pa{z_{n-1}}} \\
		&= \left.\partial_{\alpha \beta \e_1 \dots \e_{n-1} } \left(
		\log \E\left[ e^{\alpha s(z) + \beta s(z) 
			+ \sum_{i=1}^{n-1} \e_i s\left(z_i\right)} \right]
		\right)
		\right|_{\alpha, \beta, \e_i = 0} \\
		&= \left.\partial_{\e_1 \dots \e_{n-1}}
		\left(
			\frac{\E\Bigl[ s(z)^2 e^{\sum_{i=1}^{n-1} \e_i s\left(z_i\right)} \Bigr]}
				{\E\Bigl[ e^{\sum_{i=1}^{n-1} \e_i s\left(z_i\right)} \Bigr]}
		- \left(
			\frac{\E\Bigl[ s(z) e^{\sum_{i=1}^{n-1} \e_i s\left(z_i\right)} \Bigr] }
				{\E\Bigl[ e^{\sum_{i=1}^{n-1} \e_i s\left(z_i\right)} \Bigr] }
		\right)^2
		\right) \right|_{\e_i = 0}.
	\end{align*}
Note first that
	\begin{align*}
		\left.\partial_{\e_1 \dots \e_{n-1} } \left(
			\frac{\E\Bigl[ s(z)^2 e^{\sum_{i=1}^{n-1} \e_i s\left(z_i\right)} \Bigr]}
				{\E\Bigl[ e^{\sum_{i=1}^{n-1} \e_i s\left(z_i\right)} \Bigr]}
		\right) \right|_{\e_i = 0}
		&=
		\partial_{\alpha \e_1 \dots \e_{n-1} } \left(
		\log \E\left[ e^{\alpha s(z)^2 + \sum_{i=1}^{n-1} \e_i s\left(z_i\right)} \right] 
		\right) \biggr|_{\alpha, \e_i = 0} \\
		&=  \kappa_n\left( s^2(z), s\left(z_1\right), \dots, s\left(z_{n-1}\right) \right).
	\end{align*}
On the other hand, denoting by $\partial_J$ the partial differentiation with respect to the $\varepsilon_j$'s for $j\in J$, the general Leibniz rule implies
	\begin{align*}
		\partial_I
		\left[
		\left(
			\frac{\E\Bigl[ s(z) e^{\sum_{i=1}^{n-1} \e_i s\left(z_i\right)} \Bigr] }
				{\E\Bigl[ e^{\sum_{i=1}^{n-1} \e_i s\left(z_i\right)} \Bigr] }
		\right)^2 
		\right]
		& = \sum_{J \subseteq I} 
		\partial_J \left(
			\frac{\E\Bigl[ s(z) e^{\sum_{i=1}^{n-1} \e_i s\left(z_i\right)} \Bigr] }
				{\E\Bigl[ e^{\sum_{i=1}^{n-1} \e_i s\left(z_i\right)} \Bigr] }
		\right) 
		\partial_{I \backslash J} \left(
			\frac{\E\Bigl[ s(z) e^{\sum_{i=1}^{n-1} \e_i s\left(z_i\right)} \Bigr] }
				{\E\Bigl[ e^{\sum_{i=1}^{n-1} \e_i s\left(z_i\right)} \Bigr] }
		\right).
	\end{align*}
Therefore, we get	
	\begin{align*}
		\left.\partial_{\e_1 \dots \e_{n-1}}
		\left[
		\left(
			\frac{\E\Bigl[ s(z) e^{\sum_{i=1}^{n-1} \e_i s\left(z_i\right)} \Bigr] }
				{\E\Bigl[ e^{\sum_{i=1}^{n-1} \e_i s\left(z_i\right)} \Bigr] }
		\right)^2 
		\right]\right|_{\e_i = 0}
		& = \sum_{ J \subseteq I} 
			c_{|J| + 1}(z, z_J) c_{n-|J|}(z, z_{I\backslash J})
	\end{align*}
and this proves the result.
\end{proof}

\subsection{Moments. }
In this section, we use Theorem \ref{loop equation for cumulants} to prove loop equations in terms of moments and centered moments.
One can also prove the following proposition directly using integrations by parts.
We omit the loop equation of rank 1 for which the formulas in terms of moments and cumulants are already identical.

\begin{proposition}
	\label{uncentered loop eq moments}
	For any $n \geq 2$, $z,z_1,\dots,z_{n-1} \in \CC \setminus \R$, we have
	\begin{align*}
		\E \left[ \left( 
		s(z)^2
		- \frac{1}{N}\pa{1-\frac{2}{\beta}} s'(z)
		+ \frac{1}{N} \sum_{k=1}^N \frac{V'\pa{\lambda_k}}{\lambda_k - z}
		\right) 
		\prod_{i=1}^{n-1} s(z_i) 
		+ \frac{2}{ N^2\beta} \sum_{j=1}^{n-1} 
		\partial_{z_j} \frac{s(z) - s\left(z_j\right)}{z-z_j}
		\prod_{i \neq j} s(z_i)
		\right] 
		= 0.
	\end{align*}
\end{proposition}	
\begin{proof}
We will show that the loop equation in moments of rank $n$ is a sum over loop equations in cumulants of rank up to $n$.	
As before, $\kappa(\cdot)$ and $c(\cdot)$ denote joint cumulants, however we omit their indices which are implied by the number of its arguments. 
By (\ref{moment expansion}), we can re-write the claimed loop equation for moments as
	\begin{align} 
		& \sum_\pi \kappa\left( s(z)^2
				- \frac{1}{N}\pa{1-\frac{2}{\beta}} s'(z)
				+ \frac{1}{N} \sum_{k=1}^N \frac{V'\pa{\lambda_k}}{\lambda_k - z}, 
		\left\{s\pa{z_i}\right\}_{i \in I} \right)
			\prod_{z \notin B \in \pi} c \left( z_B \right) \nonumber \\
		& +\frac{2}{N^2\beta} \sum_{j=1}^{n-1} \partial_{z_j}\frac{1}{z-z_j} \left(
		\sum_{\pi_j} \prod_{B \in \pi_j} c\left(z_B\right)
		- \sum_{\hat{\pi}} \prod_{B \in \hat{\pi}} c\left(z_B\right)
		\right) = 0, \label{eq:loop_eq_moments_goal}
	\end{align}	
where $\pi$ denotes a partition of $\left(z,z_1\dots z_{n-1}\right)$, $I = I(\pi)$ denotes those variables that appear in the same block as $z$ in $\pi$, $\pi_j$ denotes a partition of $\left(z, z_1, \dots, z_{j-1}, z_{j+1}, \dots, z_{n-1} \right)$ for $j \in \{1, \dots, n-1\}$,
and $\hat{\pi}$ denotes a partition of $\left(z_1\dots z_{n-1}\right)$. We re-write the derivative
terms as sums over partitions $\pi$. First,
	\begin{align*}
		\sum_{j=1}^{n-1} \partial_{z_j}  \left( \frac{1}{z-z_j} 
		\sum_{\pi_j} \prod_{B \in \pi_j} c\left(z_B\right) \right)
		&=\sum_{j=1}^{n-1} \partial_{z_j} \left( \frac{1}{z-z_j} 
		\sum_{\pi_j} \kappa\left(z, z_{I\left(\pi_j\right)}\right) \prod_{z\notin B \in \pi_j} c\left(z_B\right) \right) \\
		&= \sum_\pi \sum_{j \in I(\pi)} 
		\partial_{z_j}\left( \frac{1}{z-z_j} 
		c\left(z, z_{I(\pi) \backslash\{j\}}\right) \prod_{z\notin B \in \pi} c\left(z_B\right) \right).
	\end{align*}
Similarly, 
	\[
		\sum_{j=1}^{n-1} \partial_{z_j}\frac{1}{z-z_j} \left(
		 \sum_{\hat{\pi}} \prod_{B \in \hat{\pi}} c\left(z_B\right)
		\right)
		=
		\sum_\pi \sum_{j \in I(\pi)} 
		\partial_{z_j}\frac{1}{z-z_j} \left(
		 c\left(z_{I(\pi)}\right) \prod_{z\notin B \in \pi} c\left(z_B\right)
		\right).
	\]	
Therefore the claimed rank $n$ loop equation for moments is equivalent to
	\begin{align*}
	& \sum_\pi \left( \prod_{z\notin B \in \pi} c\left(z_B\right) \right)
	\Biggl( 
	\kappa\left( s(z)^2 - \frac{1}{N}\pa{1-\frac{2}{\beta}} s'(z)
			+ \frac{1}{N} \sum_{k=1}^N \frac{V'\pa{\lambda_k}}{\lambda_k - z}, 
			\left\{s\pa{z_i}\right\}_{i \in I(\pi)}  \right) \\
	& \hphantom{\sum_\pi \left( \prod_{z\notin B \in \pi} c\left(z_B\right) \right)
		\Biggl(} {}
		+\frac{2}{N^2\beta} 
		\sum_{j \in I(\pi)} 
		\partial_{z_j}\frac{c(z, z_{I(\pi) \backslash\{j\}})
		- c(z_{I(\pi)})}{z-z_j} \Biggr) = 0.
	\end{align*}
	This indeed vanishes by Theorem \ref{loop equation for cumulants} and Lemma \ref{simplified loop equation for cumulants}.
%
%
%
\end{proof}
As a consequence of the previous proposition, it follows from a direct calculation that the product of $s(z_i)$'s can be replaced by a product of the centered versions of these random variables. We do not use these loop equations for centered moments in this paper, but they may be useful in another context.
\begin{corollary}
\label{prop:centered_loop_eq_moments}
For any $n \geq 2$, $z,z_1,\dots,z_{n-1} \in \CC \setminus \R$, letting $\mathring{s}(z) \coloneqq s(z) - \E[s(z)]$, we have
	\begin{align*}
		\E \left[ \left( 
		s(z)^2
		- \frac{1}{N}\pa{1-\frac{2}{\beta}} s'(z)
		+ \frac{1}{N} \sum_{k=1}^N \frac{V'\pa{\lambda_k}}{\lambda_k - z}
		\right) 
		\prod_{i=1}^{n-1} \mathring{s}(z_i)
		+ \frac{2}{ N^2\beta} \sum_{j=1}^{n-1} 
		\partial_{z_j} \frac{s(z) - s\left(z_j\right)}{z-z_j}
		\prod_{i \neq j} \mathring{s}(z_i)
		\right] 
		= 0.
	\end{align*}
\end{corollary}

\subsection{Confined loop equations. }
In this section we state the loop equations for particles chosen according to the following modified probability measure
\begin{equation}
	\label{eq:beta_ensembles_density_confined}
	\diff \mu_N^{[a,b]} (\lambda_1,\dots,\lambda_N)
	\coloneqq \frac{1}{Z_N^{[a,b]}}
	\cdot
	\prod_{1 \leq k < l \leq N} \absa{\lambda_k - \lambda_l}^\beta 
	\cdot
	\prod_{k=1}^N e^{- \frac{\beta N}{2} V(\lambda_k)}  \1_{\lambda_k \in [a,b]} 
	\rd \lambda_k,
\end{equation}
where we restrict ourselves to particles in an interval $[a,b]$ for some $-\infty < a < b < \infty$.
We denote by $\E^{[a,b]}$ the integral with respect to $\mu_N^{[a,b]}$.
Then the loop equations for moments can be deduced as before from the loop equations for cumulants stated in \cite[Theorems 3.3 and 3.4]{BorGui2013}.
\begin{proposition} \label{prop:loop_equations_confined}
	Let $a<b$ be real numbers.
	For any $z \in \CC \setminus [a,b]$, we have
	\begin{align*}
		\E^{[a,b]} \left[ 
		s(z)^2
		- \frac{1}{N}\pa{1-\frac{2}{\beta}} s'(z)
		+ \frac{1}{N} \sum_{k=1}^N \frac{V'\pa{\lambda_k}}{\lambda_k - z}
		\right] 
		= \frac{2}{\beta N^2} \left( \frac{\partial_a \ln Z^{[a,b]}_N}{z-a} 
			+ \frac{\partial_b \ln Z^{[a,b]}_N}{z-b} \right),
	\end{align*}
	Moreover, for any $n \geq 2$, $z,z_1,\dots,z_{n-1} \in \CC \setminus [a,b]$, we have
	\begin{align*}
		& \E^{[a,b]} \left[ \left( 
		s(z)^2
		- \frac{1}{N}\pa{1-\frac{2}{\beta}} s'(z)
		+ \frac{1}{N} \sum_{k=1}^N \frac{V'\pa{\lambda_k}}{\lambda_k - z}
		\right) 
		\prod_{i=1}^{n-1} s(z_i)
		\right] \\
		& {} 
		+ \frac{2}{ N^2\beta} \sum_{j=1}^{n-1} \E^{[a,b]} \left[
		\partial_{z_j} \frac{s(z) - s\left(z_j\right)}{z-z_j}
		\prod_{i \neq j} s(z_i)
		\right] 
		- \frac{2}{ N^2\beta} \left( \frac{\partial_a \E^{[a,b]}[\prod_{i=1}^{n-1} s(z_i)]}{z-a} 
			+ \frac{\partial_b \E^{[a,b]}[\prod_{i=1}^{n-1} s(z_i)]}{z-b} \right) \\
		& = \frac{2}{\beta N^2} \left( \frac{\partial_a \ln Z^{[a,b]}_N}{z-a} 
			+ \frac{\partial_b \ln Z^{[a,b]}_N}{z-b} \right)
		\E^{[a,b]} \left[ \prod_{i=1}^{n-1} s(z_i) \right].
	\end{align*}
\end{proposition}	
Moreover, the additional terms appearing in this new loop equation can be controlled using the following lemma. 
\begin{lemma} \label{lem:terms_derivative_wrt_a}
For any $a<A$ and $b>B$, there exists $c = c(a,b) > 0$ such that, for any $N \geq 1$, 
\begin{align*}
\absa{\partial_a \ln Z^{[a,b]}_N} & \leq e^{-cN}
\end{align*}
and, for any $n \geq 2$, $z_1,\dots,z_{n-1} \in \CC \setminus [a,b]$,
\begin{align*}
\absa{\partial_a \E^{[a,b]} \left[ \prod_{i=1}^{n-1} s(z_i) \right]} 
& \leq \frac{2e^{-cN}}{\prod_{i=1}^{n-1} d(z_i,[a,b])}.
\end{align*}
Moreover, the same inequalities hold with $\partial_b$ instead of $\partial_a$.
\end{lemma}
\begin{proof}
The first inequality is stated in \cite[Proposition 2.3]{BorGui2013}. The second inequality follows from the bound
\begin{align*}
\absa{\partial_a \E^{[a,b]} \left[ \prod_{i=1}^{n-1} s(z_i) \right]} 
& \leq \frac{2}{\prod_{i=1}^{n-1} d(z_i,[a,b])} \absa{\partial_a \ln Z^{[a,b]}_N},
\end{align*}
which can be found in the proof of the aforementioned result. 
\end{proof}

\setcounter{equation}{0}
\setcounter{theorem}{0}
\renewcommand{\theequation}{B.\arabic{equation}}
\renewcommand{\thetheorem}{B.\arabic{theorem}}
\appendix
\setcounter{secnumdepth}{0}
\section[Appendix B: Stability of the fixed point equation]
{Appendix \mynameis{B}:\ \ \ Stability of the fixed point equation\label{section:stability_lemma}}

Recall from \eqref{eq:link_2m+V'_and_r} and \eqref{eq:fixed_point_equation} that the Stieljes transform $m_V(z)$ of the equilibrium measure satisfies $m_V(z) = -\frac{V'(z)}{2} + r(z) b(z)$ for $z \in \CC \setminus [A,B]$ and is a solution of the equation $u^2 + V'(z) u + h(z) = 0$. 
It is then easy to check that the other root of this equation can be written as
\[
	\widetilde{m}_V(z) \coloneqq -\frac{V'(z)}{2} - r(z) b(z).
\]
For $z \in \CC \setminus [A,B]$, these roots are identical if and only if $r(z) = 0$.
Our goal in this section is to study the stability of the equation $u^2 + V'(z) u + h(z) = 0$ with respect to a shift $\zeta$ in the constant term. 
\begin{lemma} \label{lem:fixed_point_equation} 
	There exists $C > 0$ depending only on $V$ such that the following results hold, with $\tilde{\eta}$ from Lemma~\ref{lem:preliminaries_fixed_point_equation}.
	Let $\zeta \in \CC$ and $z = E + \ii \eta$ with $E \in [A-\tilde{\eta},B+\tilde{\eta}]$ and $\eta \in (0,\tilde{\eta})$.
	Let $u$ be a solution of
	$u^2 + V'(z) u + h(z) = \zeta$.
	Then, 
	\begin{equation} \label{eq:stability_lemma_bound_1}
		\absa{u-m_V(z)} \wedge \absa{u-\widetilde{m}_V(z)}
		\leq C \left( \frac{\abs{\zeta}}{\abs{b(z)}} \wedge \abs{\zeta}^{1/2} \right).
	\end{equation}
	If $\im(u) > 0$, then
	\begin{equation} \label{eq:stability_lemma_bound_2}
		\absa{\im(u-m_V(z))} 
		\leq 5 \left( \absa{u-m_V(z)} \wedge \absa{u-\widetilde{m}_V(z)} \right).
	\end{equation}
	If $\im(u) > 0$ and $A-\eta \leq E \leq B+\eta$, then
	\begin{equation} \label{eq:stability_lemma_bound_3}
		\absa{u-m_V(z)} 
		\leq C \left( \frac{\abs{\zeta}}{\abs{b(z)}} \wedge \abs{\zeta}^{1/2} \right).
	\end{equation}
\end{lemma}
To prove this we will first establish two preliminary technical results. 
\begin{lemma} \label{lem:bounds for_b}
	Let $z = E + \ii \eta$ with $E \in \R$ and $\eta > 0$.
	If $A-\eta \leq E \leq B+\eta$, then
	\begin{align} \label{eq:lower_bound_im_b}
		\im b(z) \geq \frac{\abs{b(z)}}{3}.
	\end{align} 
	If $E \notin (A-\eta,B+\eta)$, then
	\begin{align} 
		\frac{(B-A)\eta}{2 \abs{b(z)}} \leq \im b(z) & \leq \abs{2E-A-B} \frac{\eta}{\abs{b(z)}}, \label{eq:bound_im_b} \\
		\re b(z) & \geq \frac{\abs{b(z)}}{2}. \label{eq:lower_bound_re_b}
	\end{align} 
\end{lemma}
\begin{proof} 
Let $\theta_A = \Arg (z-A) \in (0,\pi)$ and $\theta_B = \Arg (z-B) \in (0,\pi)$.
Then $b(z) = \sqrt{z-A} \sqrt{z-B} = \abs{b(z)} e^{i(\theta_A+\theta_B)/2}$.
Thus, $\im b(z) = \abs{b(z)} \sin (\frac{\theta_A+\theta_B}{2})$.
Now, we deal with the case $\re(z) \geq \frac{A+B}{2}$, the other case being treated similarly. 
In this case, we have $\theta_A + \theta_B < \pi$, and therefore $\im b(z)$ is increasing as a function of $\theta_A + \theta_B$.

In the case $E \leq B+\eta$, we have $\theta_B \geq \frac{\pi}{4}$ and therefore $\sin (\frac{\theta_A+\theta_B}{2}) \geq \frac{1}{3}$, which proves \eqref{eq:lower_bound_im_b}. 
In the case $E \geq B+\eta$, we use that $\sin (\frac{\theta_A+\theta_B}{2}) \geq \frac{1}{2} \sin (\theta_A+\theta_B)$ to get
\[
	\im b(z) \geq \frac{\im ((z-A)(z-B))}{2 \abs{b(z)}} 
	= \frac{(2E-A-B)\eta}{2 \abs{b(z)}} 
	\geq \frac{(B-A)\eta}{2 \abs{b(z)}},
\]
using $E \geq B$. For the upper bound, we use that $\sin (\frac{\theta_A+\theta_B}{2}) \leq \sin (\theta_A+\theta_B)$ because $\theta_A+\theta_B \in [0,\frac{\pi}{2}]$ and  proceed similarly to prove \eqref{eq:bound_im_b}.
Finally, note that $\re b(z) \geq \abs{b(z)} \cdot \cos \frac{\pi}{4}$ and \eqref{eq:lower_bound_re_b} follows.
\end{proof}
\begin{lemma} \label{lem:preliminaries_fixed_point_equation}
	There exist constants $\tilde{\eta},c,C > 0$ depending only on $V$ such that for any $z = E + \ii \eta$ with $E \in [A-\tilde{\eta},B+\tilde{\eta}]$ and $\eta \in (0,\tilde{\eta}]$, the function $V$ is analytic at point $z$ and we have
	\begin{align} 
		\abs{r(z)} & \geq c,
		\label{eq:preliminary_lemma_bound_1} \\
		\im(r(z) b(z)) & \geq \abs{\im V'(z)}.
		\label{eq:preliminary_lemma_bound_2}
	\end{align} 
	Moreover, if $A-\eta \leq E \leq B+\eta$,
	\begin{equation} 
		\im(r(z) b(z)) \geq c \abs{r(z) b(z)},
		\label{eq:preliminary_lemma_bound_3}
	\end{equation} 
	and, if $E \notin [A-\eta,B+\eta]$,
	\begin{align} 
		\abs{\im m_V(z)}\vee \abs{\im \widetilde{m}_V(z)} 
		& \leq \frac{C\eta}{\abs{b(z)}},
		\label{eq:preliminary_lemma_bound_4} \\
		\re (m_V(z_0) - \widetilde{m}_V(z)) & \geq c \abs{r(z) b(z)},
		\label{eq:preliminary_lemma_bound_5}
	\end{align} 
	where $z_0 \coloneqq B+\eta$ if $E > B+\eta$ and $z_0 \coloneqq A-\eta$ if $E < A-\eta$.
\end{lemma}
\begin{proof}
Recall $r$ and $V'$ are analytic functions in an open set containing $[A,B]$, so we can choose $\tilde{\eta}$ small enough so that the point $z$ considered is included in this open set. Moreover, $r$ is positive in $[A,B]$ and $V'$ is real in $[A,B]$, so we can choose $\tilde{\eta},c,C >0$ such that, for any $z = E + \ii \eta$ with $E \in [A-\tilde{\eta},B+\tilde{\eta}]$ and $\eta \in (0,\tilde{\eta})$, 
\[
\abs{\im r(z)} \leq C \eta,
\qquad 
c \leq \abs{r(z)} \leq C,
\qquad
\re r(z) \geq c,
\qquad
\abs{\im V'(z)} \leq C \eta,
\qquad 
c \sqrt{\eta} \leq \abs{b(z)} \leq C.
\]
In particular, \eqref{eq:preliminary_lemma_bound_1} holds.
Moreover, we get
\begin{align} \label{eq:lower_bound_im_rb}
	\im(r(z) b(z)) 
	& = \im (r(z)) \re (b(z)) + \re(r(z)) \im(b(z))
	\geq -C^2 \eta + c \im b(z).
\end{align} 
In the case $A-\eta \leq E \leq B+\eta$, applying \eqref{eq:lower_bound_im_b}, we get $\im(r(z) b(z)) \geq -C^2 \eta + \frac{c}{3} \abs{b(z)}$.
Since $\abs{b(z)} \geq c \sqrt{\eta}$, the second term will dominate for $\eta$ small enough, so up to a modification of the choice of $\tilde{\eta}$ (depending on $c,C$), we have 
\[
\im(r(z) b(z)) \geq \frac{c}{6} \abs{b(z)} 
\geq \abs{\im V'(z)} \vee \frac{c}{6C} \abs{r(z) b(z)}.
\]
So up to a modification of the choice of $c$, both \eqref{eq:preliminary_lemma_bound_2} and \eqref{eq:preliminary_lemma_bound_3} hold in this region.
Now we consider the case $E \notin [A-\eta, B+\eta]$.
If $E > B+\eta$, we have 
\[
\abs{b(z)} \leq \sqrt{B-A} \cdot \abs{z-B}^{1/2}
\leq \sqrt{B-A} \cdot \sqrt{2 \tilde{\eta}},
\]
using that $E \leq B + \tilde{\eta}$. The same inequality holds in the case $E < A-\eta$.
Hence, coming back to \eqref{eq:lower_bound_im_rb} and applying \eqref{eq:lower_bound_im_b}, we get
\[
\im(r(z) b(z)) \geq -C^2 \eta + \frac{c\eta (B-A)}{2 \abs{b(z)}}
\geq -C^2 \eta + \frac{c \eta \sqrt{B-A}}{2 \sqrt{2 \tilde{\eta}}}.
\]
Recalling $\abs{\im V'(z)} \leq C \eta$, we can see that, choosing $\tilde{\eta}$ small enough, \eqref{eq:preliminary_lemma_bound_3} holds.
Then, \eqref{eq:preliminary_lemma_bound_4} is an easy consequence of \eqref{eq:bound_im_b}.
Finally, we have to prove \eqref{eq:preliminary_lemma_bound_5}. 
Proceed in the same way as for \eqref{eq:preliminary_lemma_bound_3}, we have, using in particular \eqref{eq:lower_bound_re_b},
\begin{equation} 
	\re(r(z) b(z)) \geq c \abs{r(z) b(z)},
\end{equation} 
for $E \notin [A-\eta, B+\eta]$.
Using this, we get
\begin{align*} 
	\re (m_V(z_0) - \widetilde{m}_V(z)) 
	= \re \left( \frac{V'(z) - V'(z_0)}{2} + r(z) b(z) + r(z_0) b(z_0) \right)
	\geq c \abs{r(z) b(z)} - \frac{1}{2} \abs{V'(z) - V'(z_0)}.
\end{align*} 
Since $\abs{V'(z) - V'(z_0)}$ is of order $\eta$ and $\abs{r(z) b(z)}$ of order $\sqrt{\eta}$, \eqref{eq:preliminary_lemma_bound_5} holds if we choose $c$ and $\tilde{\eta}$ small enough.
This concludes the proof.
\end{proof}
\begin{proof}[Proof of Lemma \ref{lem:fixed_point_equation}]
For brevity, in this proof, we write $V'$, $h$, $r$, $b$, $m_V$, $\widetilde{m}_V$ instead of $V'(z)$, $h(z)$, $r(z)$, $b(z)$, $m_V(z)$, $\widetilde{m}_V(z)$.
Moreover, we define $c$ and $\tilde{\eta}$ as the constants given by Lemma \ref{lem:preliminaries_fixed_point_equation}.

We first prove \eqref{eq:stability_lemma_bound_1}.
Since $m_V$ and $\widetilde{m}_V$ are the two roots of the polynomial $X^2 + V' X +h$, we have
\begin{align} \label{eq:two_roots}
	\zeta = u^2 + V' u + h = (u-m_V)(u-\widetilde{m}_V).
\end{align}
It follows immediately that
\[
	\absa{u-m_V} \wedge \absa{u-\widetilde{m}_V}
	\leq \abs{\zeta}^{1/2}.
\]
On the other hand, if $\absa{u-m_V} \leq \absa{u-\widetilde{m}_V}$, then $\absa{u-\widetilde{m}_V} \geq \frac{1}{2} \absa{m_V-\widetilde{m}_V} = \abs{rb}$, so it follows from \eqref{eq:two_roots} that $\absa{u-m_V} \leq \abs{\zeta}/\abs{rb}$.
Proceeding similarly in the case $\absa{u-m_V} \geq \absa{u-\widetilde{m}_V}$, we get
\[
	\absa{u-m_V} \wedge \absa{u-\widetilde{m}_V}
	\leq \frac{\abs{\zeta}}{\abs{rb}}.
\]
By Lemma \ref{lem:preliminaries_fixed_point_equation}, $\abs{r} \geq c$, so the two previous displayed equation prove \eqref{eq:stability_lemma_bound_1}.

Now we assume $\im(u) > 0$ and prove \eqref{eq:stability_lemma_bound_2}. 
It follows that
\begin{align*}
	\abs{\im(u-\widetilde{m}_V)}
	\geq \im(u-\widetilde{m}_V)
	\geq - \im(\widetilde{m}_V) 
	= \frac{1}{2} \im(V') + \im(rb)
	\geq \frac{1}{2} \im(rb)
\end{align*}
where we used that $\im(rb) \geq \abs{\im V'}$ by Lemma \ref{lem:preliminaries_fixed_point_equation} in the last inequality.
Note that it follows from the same inequality that $\im(rb) \geq 0$, so $\abs{\im(m_V-\widetilde{m}_V)} = 2 \im(rb) \leq 4 \abs{\im(u-\widetilde{m}_V)}$.
We deduce that 
\begin{align*}
	\abs{\im(u-m_V)}
	\leq \abs{\im(u-\widetilde{m}_V)} + \abs{\im(m_V-\widetilde{m}_V)}
	\leq 5 \abs{\im(u-\widetilde{m}_V)},
\end{align*}
and \eqref{eq:stability_lemma_bound_2} follows.

Finally we assume both $\im(u) > 0$ and $A-\eta \leq E \leq B+\eta$, and prove \eqref{eq:stability_lemma_bound_3}. As before, we have $\abs{\im(u-\widetilde{m}_V)}\geq \frac{1}{2} \im(rb)$. 
But now we can apply bound \eqref{eq:preliminary_lemma_bound_3} of
Lemma \ref{lem:preliminaries_fixed_point_equation} which states that $\im(rb) \geq c \abs{rb}$.
We get 
\[
	\abs{u-\widetilde{m}_V}
	\geq \abs{\im(u-\widetilde{m}_V)} 
	\geq \frac{c}{2} \abs{rb} 
	= \frac{c}{4} \abs{m_V-\widetilde{m}_V}.
\]
It follows that $\abs{u-m_V} \leq (1 + \frac{4}{c}) \abs{u-\widetilde{m}_V}$ and so $\abs{u-m_V} \leq (1 + \frac{4}{c}) (\abs{u-m_V} \wedge \abs{u-\widetilde{m}_V})$.
Thus \eqref{eq:stability_lemma_bound_3} follows from \eqref{eq:stability_lemma_bound_1}.
\end{proof}



\addcontentsline{toc}{section}{References}
\bibliographystyle{abbrv}

\end{document}